\documentclass{amsart}
\usepackage[T1]{fontenc}
\usepackage{amsthm,amsmath}
\usepackage{amscd}
\usepackage{todonotes}
\usepackage[capitalize]{cleveref}
\usepackage{caption}
\usepackage{subcaption}
\usepackage[normalem]{ulem}
\usepackage[margin=3cm,a4paper]{geometry}
 
\newcommand{\CC}{c}

\newcommand{\R}{\mathbb{R}}
\newtheorem{definition}{Definition}
\newtheorem{theorem}{Theorem}
\newtheorem{proposition}{Proposition}
\newtheorem{lemma}{Lemma}
\newtheorem{example}{Example}

\newtheorem{remark}{Remark}
\providecommand{\E}{\operatorname{E}}

\newcommand{\diff}{\,\mathrm{d}}
\newcommand{\norm}[1]{\left\lVert#1\right\rVert}

\begin{document}
\title{Splitting integrators for stochastic Lie--Poisson systems}
\author{Charles-Edouard Br\'ehier}
\address{Charles-Edouard Br\'ehier\\Univ Lyon, Universit\'e Claude Bernard Lyon 1, CNRS UMR 5208, Institut Camille Jordan, 43 blvd. du 11 novembre 1918, 
F--69622 Villeurbanne cedex, France}
\email{brehier@math.univ-lyon1.fr}
\urladdr{http://math.univ-lyon1.fr/~brehier/}

\author{David Cohen}
\address{David Cohen\\Mathematical Sciences, Chalmers University of Technology and University of Gothenburg, SE--41296 Gothenburg, Sweden}
\email{david.cohen@chalmers.se}
\urladdr{http://www.math.chalmers.se/~cohend/}

\author{Tobias Jahnke}
\address{Tobias Jahnke\\Department of Mathematics, Institute for Applied and Numerical Mathematics, Karlsruhe Institute of Technology, Englerstr. 2, 
D--76131 Karlsruhe, Germany}
\email[3]{tobias.jahnke@kit.edu}
\urladdr[3]{http://www.math.kit.edu/ianm3/~jahnke/}

\begin{abstract}
We study stochastic Poisson integrators for a class of stochastic Poisson systems driven by Stratonovich noise. Such geometric integrators preserve Casimir functions and the Poisson map property. For this purpose, we propose explicit stochastic Poisson integrators based on a splitting strategy, and analyse their qualitative and quantitative properties: preservation of Casimir functions, existence of almost sure or moment bounds, asymptotic preserving property, and strong and weak rates of convergence. The construction of the schemes and the theoretical results are illustrated through extensive numerical experiments for three examples of stochastic Lie--Poisson systems, namely: stochastically perturbed Maxwell--Bloch, rigid body and sine--Euler equations.
\end{abstract}


\maketitle

\section{Introduction}\label{sec:intro}

Hamiltonian ordinary differential equations and their generalisation, Poisson systems, 
are extensively used as mathematical models to describe the dynamical evolution of various physical systems 
in science and engineering. The flow of a Hamiltonian system is known to be a symplectic map, whereas the flow of a Poisson system is a Poisson map.
This is a key property that the flow of a numerical method should also have.  
The recent years have thus witnessed a large amount of research activities in the design and numerical analysis of symplectic numerical schemes, resp. Poisson integrators, for deterministic (non-canonical) Hamiltonian systems, see for instance the classical monographs \cite{Sanz-Serna1994,rl,HLW02,bcGNI15} and references therein. 

This research has naturally come to the realm of stochastic Hamiltonian systems. Without being too exhaustive, we mention the works 
\cite{Milstein2002,Milstein2002a,w07,MR2491434,m10,bb12,
mdd12,MR3094570,MR3218332,MR3195572,MR3552716,MR3579605,MR3882980,MR3873562,MR3952248,chjs21}  
on the numerical analysis of symplectic methods for stochastic Hamiltonian systems.  

Since symplectic methods for stochastic Hamiltonian systems offer advantages compared to 
standard numerical methods, as observed in the above list of references, 
it is natural to ask if one can derive numerical integrators respecting the structure of 
stochastic Poisson systems of the Stratonovich form  
\begin{equation*}
\left\lbrace
\begin{aligned}
\diff y(t)&=B(y(t))\nabla H(y(t))\,\diff t+\sum_{k=1}^mB(y(t))\nabla \widehat H_k(y(t))\circ\,\diff W_k(t)\\
y(0)&=y_0,
\end{aligned}
\right.
\end{equation*}
with Hamiltonian functions $H,\widehat{H}_1,\ldots,\widehat{H}_m\colon\R^d\to\R$, 
a structure matrix $B\colon \R^d \to \R^{d\times d}$, and independent 
standard real-valued Wiener processes $W_1,\ldots,W_m$, see Section~\ref{sec:poisson} for details on the notation. 


{Stochastic Poisson systems are popular models to describe diverse random phenomena, see below and \cite{MR629977,misawa94,misawa99,MR2198598,MR2408499,MR2502472,MR3747641,MR2644322,MR2970274,MR3210739,wwc21} for instance. However, to the best of our knowledge, there has been no general study of integrators for stochastic Poisson systems which respect their geometric properties in the literature so far except the recent work~\cite{hong2020structurepreserving}. In this manuscript, we intend to fill this gap and we study the notion of stochastic Poisson integrators 
(see Definition~\ref{defPI} and \cite[Theorem~3.1]{hong2020structurepreserving}): such integrators need to be Poisson maps (see Definition~\ref{def:Pmap}) and to preserve the Casimir functions of the system. Imposing these conditions is natural: indeed we prove that the flow of the stochastic Poisson system is a Poisson map (see Theorem~\ref{th:flowP}) and also preserves Casimir functions. In addition, the present notion of stochastic Poisson integrators is a natural generalisation of the notion of Poisson integrators for deterministic Poisson systems.

The main contribution of this manuscript is the analysis of a class of explicit stochastic Poisson integrators, see equation~\eqref{slpI}, based on a splitting strategy. The splitting strategy is often applicable for stochastic Lie--Poisson systems, which have a structure matrix $B(y)$ which depends linearly on $y$. The construction of the scheme is illustrated for stochastic perturbations of three systems which have been studied extensively in the deterministic case: Maxwell--Bloch, rigid body and sine--Euler equations. Note that these examples give stochastic differential equations (SDEs) with non-globally Lipschitz drift and diffusion coefficients, thus standard explicit schemes such as the Euler--Maruyama method are not expected to converge (strongly or weakly) or even to satisfy moment bounds. Instead, under appropriate assumptions, we prove that the proposed integrators converge strongly and weakly, with rates $1/2$ and $1$ respectively, see Theorem~\ref{thm-general}. Indeed, if one assumes that the system admits a Casimir function with compact level sets (which is the case for the rigid body and the sine--Euler equations), both the exact and the numerical solutions of the stochastic Poisson systems remain bounded almost surely, uniformly with respect to the time step size. Our main convergence result, Theorem~\ref{thm-general}, is illustrated with extensive numerical experiments.
}

{
On top of that, we study the properties of the stochastic Poisson systems (see Subsection~\ref{sec:multi}) and stochastic Poisson integrators (see Subsection~\ref{APsplit}) in a multiscale regime, namely when the Wiener processes are approximated by a smooth noise. The proposed splitting schemes are asymptotic preserving in this diffusion approximation regime, in the sense of the notion recently introduced in~\cite{BRR}. This property, which is not satisfied by standard integrators, is illustrated with numerical experiments.
}


{
Let us now compare our contributions with existing works. As already mentioned, the notion of stochastic Poisson integrators is a natural generalisation of the notion of Poisson integrators for deterministic systems. In the stochastic case, we are only aware of the recent work~\cite{hong2020structurepreserving}, where techniques which differ from ours are employed. First, in~\cite{hong2020structurepreserving}, the proof that the flow is a Poisson map consists in using the Darboux--Lie theorem to rewrite the stochastic Poisson system into a canonical form, i.\,e. as a stochastic Hamiltonian system, for which it is already known that the flow is a symplectic map. On the contrary, our approach to prove Theorem~\ref{th:flowP} below is more direct and extends the approach considered in~\cite[Chapter VII]{HLW02} for the deterministic case. Second, the authors of~\cite{hong2020structurepreserving} design stochastic Poisson integrators by starting from a stochastic symplectic scheme for the canonical version, and then by coming back to the original variables. Note that the transformations between the non canonical and canonical variables are often found by solving partial differential equations, and that symplectic schemes are usually implicit. Our approach is more direct and leads to explicit splitting schemes. In particular, for the stochastic rigid body system, the scheme proposed in~\cite{hong2020structurepreserving} is based on the midpoint rule and is thus implicit, whereas the scheme proposed in our work is explicit, and we are able to prove strong and weak convergence results.
}

The paper is organized as follows. Section~\ref{sec:poisson} is devoted to the setting and to the description of the main properties of stochastic Poisson systems, namely the preservation of Casimir functions and the Poisson map property (Theorem~\ref{th:flowP}, see Subsection~\ref{sec:Pmap}). The three main examples of stochastic Lie--Poisson systems (Maxwell--Bloch, rigid body and sine--Euler equations) are introduced in Subsection~\ref{sec:examples}. The diffusion approximation regime is presented in Subsection~\ref{sec:multi}. Section~\ref{sec:integrators} presents the main theoretical contributions of this work: we introduce the notion of stochastic Poisson integrators (Definition~\ref{defPI}) and we propose a class of such integrators using a splitting technique. The main convergence result (Theorem~\ref{thm-general}) is stated and proved in Subsection~\ref{sec:cv} (using an auxiliary result proved in Appendix~\ref{app-auxlem}): under appropriate assumptions, the proposed explicit splitting stochastic Poisson integrators converge in strong and weak senses, with orders $1/2$ and $1$ respectively. The asymptotic preserving property in the diffusion approximation regime is studied in Section~\ref{APsplit}. Finally, Section~\ref{sect-numexp} presents numerical experiments using the proposed splitting stochastic Poisson integrators and variants for the three examples of stochastic Lie--Poisson systems (Maxwell--Bloch, rigid body and sine--Euler equations). We illustrate various qualitative and quantitative properties, which show the superiority of the proposed schemes compared with existing methods.


\section{Stochastic Poisson and Lie--Poisson systems}\label{sec:poisson}

In this section, we set notation and introduce the stochastic differential equations studied in this article, 
namely \emph{stochastic (Lie--)Poisson systems}. 
We then state the main properties of such systems and give several examples
for which stochastic Poisson integrators are designed and analysed in Sections~\ref{sec:integrators}~and~\ref{sect-numexp}. 
We conclude this section with a diffusion approximation result justifying why considering stochastic Poisson systems 
with a Stratonovich interpretation of the noise is relevant.

\subsection{Setting and stochastic Poisson dynamics}

Let $d,m$ be positive integers: $d$ is the dimension of the considered system and $m$ is the dimension 
of the stochastic perturbation. We study \emph{stochastic Poisson systems} of the type
\begin{equation}\label{prob}
\left\lbrace
\begin{aligned}
\diff y(t)&=B(y(t))\nabla H(y(t))\,\diff t+\sum_{k=1}^mB(y(t))\nabla \widehat H_k(y(t))\circ\,\diff W_k(t),\\
y(0)&=y_0,
\end{aligned}
\right.
\end{equation}
with \emph{Hamiltonian functions} 
$H,\widehat{H}_1,\ldots,\widehat{H}_m\colon\R^d\to\R$, with \emph{structure matrix} 
$B\colon \R^d \to \R^{d\times d}$,
and with independent standard real-valued Wiener processes $W_1,\ldots,W_m$ 
defined on a probability space $(\Omega, \mathcal F, \mathbb P)$. 
The noise in the SDE~\eqref{prob} is understood in the Stratonovich sense. 
The initial value $y_0$ is assumed to be non-random for ease of presentation, but
the results of this paper can be extended to the case of random $y_0$ (independent of $W_1,\ldots,W_m$ and satisfying appropriate moment bounds). 

Henceforth we assume at least that $H\in\mathcal{C}^2$, $\widehat{H}_1,\ldots,\widehat{H}_m \in \mathcal{C}^3$,
and that $B\in\mathcal{C}^2$. The gradient is denoted by $\nabla$, e.g.
$\nabla H(y)=\bigl(\frac{\partial H(y)}{\partial y_1},\ldots,\frac{\partial H(y)}{\partial y_d})\in\R^d$.
The structure matrix $B$ is assumed to satisfy the following properties.
\begin{itemize}
\item Skew-symmetry: for every $y\in\R^d$ and for all $i,j\in\{1,\ldots,d\}$, one has
\[
B_{ij}(y)=-B_{ji}(y);
\]
\item Jacobi identity: for every $y\in\R^d$ and for all $i,j,k\in\{1,\ldots,d\}$, one has
\[
\sum_{\ell=1}^d\left( \frac{\partial B_{ij}(y)}{\partial y_\ell}B_{lk}(y)+
\frac{\partial B_{jk}(y)}{\partial y_\ell}B_{li}(y)+\frac{\partial B_{ki}(y)}{\partial y_\ell}B_{lj}(y) \right)=0.
\]
\end{itemize}
Sometimes the structure matrix $B$ is referred to as the Poisson matrix. 
In many applications the structure matrix $B$ depends linearly on $y$: 
if there is a family of real numbers $\bigl(b_{ij}^k\bigr)_{1\le i,j,k\le d}$ such that
\begin{align}
\label{Lie-Poisson.B}
B_{ij}(y)=\sum_{k=1}^db_{ji}^ky_k
\end{align}
for all $y\in\R^d$ and $i,j=1,\ldots,d$, then the system~\eqref{prob} is called a stochastic \emph{Lie--Poisson system}. 
Examples are provided below in Section~\ref{sec:examples}.
A stochastic \emph{Hamiltonian} system is obtained if $d$ is even and $B(y)=J^{-1}$ for all $y\in\R^d$, 
where 
$$J=\begin{pmatrix} 0 & Id \\ -Id & 0\end{pmatrix}.$$
If $ \widehat{H}_k = 0 $ for all $k=1,\ldots,m$, then the SDE \eqref{prob} reduces to a classical deterministic (Lie--)Poisson or Hamiltonian system; cf.~\cite{HLW02}.

Properties and numerical approximations of stochastic Hamiltonian systems and of deterministic Poisson systems 
have been extensively studied in the literature, see the references in the introduction. 
The results presented in this work are generalisations to the above stochastic Poisson case, with a special focus on
stochastic Lie--Poisson systems.

Under the previous regularity assumptions, the drift coefficient $y\mapsto B(y)\nabla H(y)$ is of class $\mathcal{C}^1$, and, for all $k=1,\ldots,m$, 
the diffusion coefficient $y\mapsto B(y)\nabla \widehat{H}_k(y)$ is of class $\mathcal{C}^2$. 
As a consequence, the stochastic differential equation~\eqref{prob} 
is locally well-posed: for any deterministic initial condition $y_0\in\R^d$, there exists a random time $\tau$, which is almost surely positive, such that~\eqref{prob} admits a unique solution $t\in[0,\tau)\mapsto y(t)$ with $y(0)=y_0$. Below we will present a criterion to ensure global well-posedness ($\tau=\infty$ almost surely for any initial condition $y_0$). This criterion is applied to study the examples presented below.

%

\subsection{Properties of stochastic Poisson systems}

Deterministic and stochastic Poisson systems have several geometric properties which we discuss in this section.
Let $\mathcal{H} \colon\R^d\to\R$ be a mapping of class $\mathcal{C}^2$. The evolution of 
$\mathcal{H}(y)$ along a solution $y(t)$ of the stochastic Poisson system~\eqref{prob} is described by 
\begin{align}
\label{dH.Poisson.bracket}
\diff\mathcal{H}(y(t))=\{\mathcal{H},H\}(y(t))\diff t+\sum_{k=1}^m \{\mathcal{H},\widehat H_k\}(y(t))\circ\,\diff W_k(t),
\end{align}
where the Poisson bracket $\{\cdot,\cdot\}$ associated with the structure matrix $B$ is defined by
\begin{align}
\label{Poisson.bracket}
\{F,G\}(y)=\displaystyle\sum_{i,j=1}^d\frac{\partial F(y)}{\partial y_i}B_{ij}(y)\frac{\partial G(y)}{\partial y_j}=\nabla F(y)^T B(y)\nabla G(y).
\end{align}
The identity~\eqref{dH.Poisson.bracket} is proved using  the chain rule for solutions of SDEs written in the Stratonovich formulation. The fact that the same structure matrix $B$ appears in both the deterministic and stochastic parts of the system~\eqref{prob} is important to express~\eqref{dH.Poisson.bracket} using the Poisson bracket defined by~\eqref{Poisson.bracket}, which depends only on $B$ but not on the Hamiltonian functions $H,\widehat H_1,\ldots,\widehat H_m$. This assumption on the system~\eqref{prob} is the key to study the geometric properties of such a system and its numerical discretisation, such as the preservation of Casimir functions or the Poisson map property. Most of the properties stated below would not hold if different structure matrices were considering 
in the stochastic terms. If $\{\mathcal{H},H\}=0$ and if the system is deterministic (i.e. if $ \widehat{H}_k = 0 $ for all $k=1,\ldots,m$), then
the equality~\eqref{dH.Poisson.bracket} implies that $t\mapsto\mathcal{H}(y(t))$ is constant, i.e. that
$\mathcal{H}$ is preserved by the flow of the deterministic Poisson system. Since every smooth Hamiltonian has the property that
$\{H,H\}=0$ this means, in particular, that the flow of a deterministic Poisson systems $\dot{y}=B(y)\nabla H(y)$ preserves the Hamiltonian $H$ (see for instance~\cite[Sect. IV.1 and VII.2]{HLW02}).
In the stochastic case, however, the Hamiltonian is in general not preserved. Precisely, Equation~\eqref{dH.Poisson.bracket} yields the following sufficient condition
\[
\{\mathcal{H},H\}=\{\mathcal{H},\widehat H_1\}=\ldots=\{\mathcal{H},\widehat H_m\}=0
\]
to obtain $\diff \mathcal{H}(y(t))=0$ and hence preservation of $\mathcal{H}$ by the flow of the stochastic Poisson system~\eqref{prob}.

In addition, deterministic and stochastic Poisson systems may have conserved quantities called Casimir functions.
\begin{definition}
A function $C\colon\R^d\to\R$ of class $\mathcal{C}^2$ is called a \emph{Casimir} function 
of the stochastic Poisson system~\eqref{prob} if for all $y\in\R^d$ one has
$$
\nabla C(y)^TB(y)=0.
$$
\end{definition}
Observe that the definition of a Casimir function for stochastic and deterministic Poisson systems only depends on the structure matrix $B$, but not on the Hamiltonian functions $H,\widehat{H}_1,\ldots,\widehat{H}_m$.
A Casimir function $C$ satisfies
\[
\{C,H\}=\{C,\widehat H_1\}=\ldots=\{C,\widehat H_m\}=0,
\]
see the definition of the Poisson bracket in equation~\eqref{Poisson.bracket}. As a consequence, owing to~\eqref{dH.Poisson.bracket}, any Casimir function $C$ is preserved by the flow of the stochastic Poisson system~\eqref{prob}, i.e. $C(y(t))=C(y_0)$ for all $t\in[0,\tau]$, independently of the choice of the Hamiltonian functions $H,\widehat{H}_1,\ldots,\widehat{H}_m$ (since the same structure matrix $B$ appears in both the deterministic and stochastic parts of~\eqref{prob} 
in order to have preservation of Casimir functions). The preservation of Casimir functions is a desirable feature for a numerical method when applied to the problem~\eqref{prob}.

A criterion to ensure global well-posedness of the dynamics can be stated based on the preservation of Casimir functions by solutions of stochastic Poisson systems: it suffices to assume the existence of a Casimir function with compact level sets.
\begin{proposition}\label{propo:global}
Assume that the stochastic Poisson system~\eqref{prob} admits a Casimir function $C$ such that 
for all $c\in\mathbb R$ the level sets $\{y\in\mathbb{R}^d\,\colon\, C(y)=c\}$ are compact. 
Then for any initial condition $y_0\in\R^d$ the SDE~\eqref{prob} admits a unique global solution $\bigl(y(t)\bigr)_{t\ge 0}$, with $y(0)=y_0$, 
such that almost surely one has, for all $t\ge0$, $C(y(t))=C(y_0)$ and 
\[
\|y(t)\|\le R(y_0)=\max_{y\in\mathbb R^d, C(y)=C(y_0)}\|y\|.
\]
\end{proposition}

\begin{proof}
The proof of Proposition~\ref{propo:global} follows from a straightforward truncation argument: 
let $R=R(y_0)+1$, and introduce mappings $H^R$ of class $\mathcal{C}^2$ and $\widehat H_1^R,\ldots,\widehat H_m^R:\mathbb{R}^d\to\mathbb{R}$, of class $\mathcal{C}^3$, with compact support included in the ball $\{y\in\mathbb{R}^d\,\colon\, \|y\|\le R\}$, and such that $H^R(y)=H(y)$, $\widehat H_k^R(y)=\widehat H_k(y)$ for all $y$ with $\|y\|\le R(y_0)$ and $k=1,\ldots,m$. Set $f^R(y)=B(y)\nabla H^R(y)$ and $\hat{f}_k^R(y)=B(y)\widehat{H}_k^R(y)$ for all $y\in\mathbb{R}^d$ and $k=1,\ldots,m$. Then $f^R$ is globally Lipschitz continuous and, for all $k=1,\ldots,m$, $\hat{f}_k^R$ is of class $\mathcal{C}^2$, bounded and with bounded derivatives.  By the standard well-posedness result for SDEs with globally Lipschitz continuous nonlinearities (when written in It\^o form), the SDE
\[
\diff y^R(t)=f^R(y^R(t))\,\diff t+\sum_{k=1}^m\hat{f}_k^R(y^R(t))\circ\,\diff W_k(t)
\]
admits a unique global solution $\bigl(y^R(t)\bigr)_{t\ge 0}$ with $y^R(0)=y_0$. Due to the discussion above, this solution preserves the Casimir function $C$:  $C(y^R(t))=C(y^R(0))=C(y_0)$ for all $t\ge 0$. Since the level sets of the Casimir function $C$ are assumed to be compact, by the definition of $R(y_0)$, one has $\|y^R(t)\|\le R(y_0)<R$ for all $t\ge 0$. This yields the equalities $f^R(y^R(t))=B(y^R(t))\nabla H(y^R(t))$ and $\hat{f}_k^R(y^R(t))=B(y^R(t))\widehat{H}_k(y^R(t))$ for all $t\ge 0$ and $k=1,\ldots,m$. Thus $\bigl(y^R(t)\bigr)_{t\ge 0}$ is in fact a global solution of the SDE~\eqref{prob} (without truncation parameter $R$). This concludes the proof of the existence of a global solution. Since the uniqueness is a consequence of the local well-posedness of~\eqref{prob} (by local Lipschitz continuity of the nonlinearities), the sketch of proof of Proposition~\ref{propo:global} is completed.
\end{proof}

To conclude this subsection, we would like to remark that the above analysis of the preservation properties of stochastic Poisson systems is only valid 
when considering stochastic differential equations with multiplicative noise interpreted in the
Stratonovich sense. Other behaviours are observed for It\^o SDEs, see for instance \cite{MR4077238,cv21}, where 
Hamiltonian and Poisson It\^o SDEs and their numerical discretisations are studied. Indeed, if the noise is interpreted in the It\^o sense, the Hamiltonian function $H$ or the Casimir functions $C$ are not preserved. Instead, one observes a linear drift in the expectation of these quantities, which is due to the second-order contribution appearing when using It\^o's formula instead of the chain rule. 
In the sequel, we only study stochastic Poisson systems~\eqref{prob} with noise interpreted in the Stratonovich sense. 
We refer to Subsection~\ref{sec:multi} below, where the relevance of considering the Stratonovich interpretation is justified by a diffusion approximation result.

\subsection{Examples of stochastic Poisson systems}\label{sec:examples}

In this subsection, we first give an example of a stochastic Poisson systems which is not a stochastic Lie--Poisson system. We then provide three examples of stochastic Lie--Poisson systems for which stochastic Poisson integrators based on a splitting strategy are designed and 
studied in Sections~\ref{sec:integrators}~and~\ref{sect-numexp}.

\begin{example}\textbf{Stochastic Lotka--Volterra system.} 
A two-dimensional stochastic Lotka--Volterra system of the form
\begin{equation}\label{slv}
\diff\begin{pmatrix}y_1\\y_2\end{pmatrix}=
\begin{pmatrix}0 & y_1y_2\\-y_1y_2 & 0\end{pmatrix}\left(
\nabla H(y)\,\diff t+\sum_{k=1}^m\nabla\widehat H_k(y)\circ\,\diff W_k\right)
\end{equation}
with the structure matrix $B(y_1,y_2)=\begin{pmatrix}0 & y_1y_2\\-y_1y_2 & 0\end{pmatrix}$ and the Hamiltonian function $H(y_1,y_2)=y_1-\ln(y_1)+y_2-2\ln(y_2)$ gives a stochastic Poisson system (with arbitrary Hamiltonian functions $\widehat{H}_1,\ldots,\widehat{H}_m$). The stochastic prey-predator model studied in~\cite{MR2008602} is obtained taking $m=1$ and $\widehat H_1(y)=-\ln(y_1)+\ln(y_2)$.
 
Observe that the stochastic Lotka-Volterra system~\eqref{slv} does not admit Casimir functions and that it is not a stochastic Lie--Poisson system (since the mapping $y\mapsto B(y)$ is quadratic).

Random perturbations of higher dimensional Lotka--Volterra systems, see \cite{HLW02} for deterministic problems and \cite{MR869549} for It\^o SDEs, 
could also fit in the general framework presented above. 
\end{example}

Let us now present the three examples of stochastic Lie--Poisson systems for which numerical integrators are built and studied in this article.

\begin{example}\textbf{Stochastic Maxwell--Bloch system.}\label{expl-MB} 
Let $d=3$. The deterministic Maxwell--Bloch equations from laser-matter dynamics read 
\begin{equation*}
\left\lbrace
\begin{aligned}
\dot y_1&=y_2\\
\dot y_2&=y_1y_3\\
\dot y_3&=-y_1y_2.
\end{aligned}
\right.
\end{equation*}
This system is relevant to study phenomena of self-induced transparency in lasers, see for instance \cite{MR1169593,MR1702129}. 
This system is a deterministic Lie--Poisson system with 
Poisson matrix, Hamiltonian and Casimir functions given by 
$$
B(y)=\begin{pmatrix} 0 & -y_3 & y_2 \\ y_3 & 0 & 0 \\ -y_2 & 0 & 0 \end{pmatrix}, \quad H(y)=\frac12y_1^2+y_3, 
\quad C(y)=\frac12(y_2^2+y_3^2),
$$
respectively, for all $y=(y_1,y_2,y_3)\in\R^3$. 

In this article, we consider the following stochastic version of the Maxwell-Bloch system:
\begin{equation}\label{smb}
\diff y=B(y)\left(\nabla H(y)\,\diff t+\sigma_1\nabla \widehat H_1(y)\circ\,\diff W_1(t)+\sigma_3\nabla \widehat H_3(y)\circ\,\diff W_3(t) \right),
\end{equation}
where $\widehat H_1(y)=\frac12y_1^2$ and $\widehat H_3(y)=y_3$, $\sigma_1,\sigma_3\ge0$, 
driven by two independent Wiener processes $W_1$ and $W_3$. 
Observe that the Casimir function $C$ does not have compact level sets in this example. The criterion given in Proposition~\ref{propo:global} thus 
cannot be applied to ensure well-posedness of the stochastic differential equation~\eqref{smb}. 
In addition, the theoretical strong and weak convergence results stated below cannot be applied to this example. 
However, it is legitimate to introduce a stochastic Poisson integrator and investigate its behaviour with numerical experiments; 
this will be presented in Section~\ref{sect-numexp} below.
\end{example}

\begin{example}\textbf{Stochastic rigid body system.}\label{expl-SRB}
Let $d=3$. The equations governing the deterministic rigid body motion read
\begin{equation*}
\left\lbrace
\begin{aligned}
\dot y_1&=(I_3^{-1}-I_2^{-1})y_3y_2\\
\dot y_2&=(I_1^{-1}-I_3^{-1})y_1y_3\\
\dot y_3&=(I_2^{-1}-I_1^{-1})y_2y_1,
\end{aligned}
\right.
\end{equation*}
where the unknown $y=(y_1,y_2,y_3)\in\R^3$ represents the angular momentum in the body 
frame and the positive distinct real numbers $I_1,I_2,I_3$ are referred to as the principal moments of inertia, see for instance~\cite[Sect. VII.5]{HLW02}.

The system above is a deterministic Lie--Poisson system, with $d=3$, where the Poisson matrix is given by 
$$
B(y)=\begin{pmatrix}0 & -y_3 & y_2\\ y_3 & 0 & -y_1\\-y_2 & y_1 & 0\end{pmatrix}
$$
and the Hamiltonian function is given by 
$$
H(y)=\frac12\left(\frac{y_1^2}{I_1}+\frac{y_2^2}{I_2}+\frac{y_3^2}{I_3}\right),
$$
for all $y\in\R^3$. The system admits the quadratic Casimir function given by
$$
C(y)=y_1^2+y_2^2+y_3^2,
$$
for all $y\in\R^3$.

In this article, we consider the following stochastic version of the rigid body system
\begin{align}\label{srb}
\diff\begin{pmatrix}y_1\\y_2\\y_3\end{pmatrix}&=
B(y)
\left(\nabla H(y)\,\diff t+\nabla \widehat H_1(y)\circ\,\diff W_1(t)
+\nabla \widehat H_2(y)\circ\,\diff W_2(t)\right.\nonumber\\
&\quad+\left.\nabla \widehat H_3(y)\circ\,\diff W_3(t) \right),
\end{align}
where $\widehat H_k(y)=\frac{y_k^2}{\widehat I_k}$, for $k=1,2,3$, with $\widehat I_1,\widehat I_2,\widehat I_3$  
positive and pairwise distinct real numbers. 
The system~\eqref{srb} is a stochastic Lie--Poisson system. 
Owing to Proposition~\ref{propo:global}, it is globally well-posed, since the Casimir function $C$ has compact level sets.

Observe that the system~\eqref{srb} is a generalisation of the stochastic rigid body motion equations studied 
in~\cite{MR3210739,MR3747641,MR3588726,cdr20}, where $m=1$. It would be straightforward to adapt the results presented below 
concerning the discretisation of~\eqref{srb} to the case $m=1$, this is left to the interested reader.
In the sequel, we only consider the case $m=3$, i.\,e. the system~\eqref{srb}.
\end{example}

\begin{example}\textbf{Stochastic sine--Euler system.}\label{expl-SE} 
The sine--Euler equations consist of a finite-dimensional truncation of the two-dimensional Euler equations in fluid dynamics. 
These equations were first proposed in \cite{MR1115867} and yield the deterministic Lie--Poisson system 
\begin{equation*}
\dot \omega_{\bf m}=\sum_{n_1,n_2=-M,{\bf n}\neq{\bf0}}^M\frac{\sin(\frac{2\pi}{N}{\bf m}\times{\bf n})}{|{\bf n}|^2}\omega_{{\bf m}+{\bf n}}\omega_{-\bf n},
\end{equation*}
where $M$ is an arbitrary positive integer. In the current example, all indices are understood modulo $N=2M+1$, and one sets ${\bf m}\times{\bf n}=m_1n_2-m_2n_1$ for ${\bf n}=(n_1,n_2)$ and ${\bf m}=(m_1,m_2)$.

The unknown is $\omega=(\omega_{\bf n})_{n_1,n_2=-M}^M$, where the complex numbers $\omega_{\bf n}$ satisfy the Hermitian symmetry property $\omega_{-\bf n}=\omega_{\bf n}^\star$ (where $w^\star$ denotes the complex conjugate of a complex number $w$), and $\omega_{(0,0)}=0$. The dimension of the system is 
$d=N^2-1=(2M+1)^2-1$, but under the Hermitian symmetry property there are only $d/2$ independent complex-valued components.

In this article, we consider the case $M=1$, thus $N=3$ and $d=8$. The unknown is written as
\[
\omega=(\omega_{(1,0)},\omega_{(1,1)},\omega_{(0,1)},\omega_{(-1,1)},\omega_{(1,0)}^\star,\omega_{(1,1)}^\star,\omega_{(0,1)}^\star,\omega_{(-1,1)}^\star).
\]
Introduce the fundamental cell 
\[
K=\{(0,1),(0,-1),(1,1),(-1,-1),(1,0),(-1,0),(-1,1),(1,-1)\}.
\]
Then the deterministic sine--Euler system above can be written as a Lie--Poisson system (see for instance \cite{MR1860719,MR1246065})
\begin{equation}\label{detSE}
\dot\omega=B(\omega)\nabla H(\omega),
\end{equation}
where the Poisson matrix is given by
$$
B(\omega)=\frac{\sqrt{3}}{2}
\begin{pmatrix}
0 & \omega_{(-1,1)} & \omega_{(1,1)} & \omega_{(0,1)} & 0 & -\omega_{(0,1)}^* & -\omega_{(-1,1)}^* & -\omega_{(1,1)}^* \\
-\omega_{(-1,1)} & 0 & \omega_{(-1,1)}^* & -\omega_{(0,1)}^* & \omega_{(0,1)} & 0 & -\omega_{(1,0)} & \omega_{(1,0)}^* \\
-\omega_{(1,1)} & -\omega_{(-1,1)}^* & 0 & \omega_{(1,1)}^* & \omega_{(-1,1)} & \omega_{(1,0)}^* & 0 & -\omega_{(1,0)} \\
-\omega_{(0,1)}  & \omega_{(0,1)}^* & -\omega_{(1,1)}^* & 0 & \omega_{(1,1)} & -\omega_{(1,0)} & \omega_{(1,0)}^* & 0 \\
0 & -\omega_{(0,1)} & -\omega_{(-1,1)} & -\omega_{(1,1)} & 0 & \omega_{(-1,1)}^* & \omega_{(1,1)}^* & \omega_{(0,1)}^* \\
\omega_{(0,1)}^* & 0 & -\omega_{(1,0)}^* & \omega_{(1,0)} & -\omega_{(-1,1)}^* & 0 & \omega_{(-1,1)} & -\omega_{(0,1)} \\
\omega_{(-1,1)}^* & \omega_{(1,0)} & 0 & -\omega_{(1,0)}^* & -\omega_{(1,1)}^* & -\omega_{(-1,1)} & 0 & \omega_{(1,1)} \\
\omega_{(1,1)}^* & -\omega_{(1,0)}^* & \omega_{(1,0)} & 0 & -\omega_{(0,1)}^* & \omega_{(0,1)} & -\omega_{(1,1)} & 0
\end{pmatrix}
$$
and with the Hamiltonian function given by
\begin{align*}
H(\omega)=\frac12\sum_{{\bf n}\in K}\frac{\omega_{\bf n}\omega_{-\bf n}}{|{\bf n}|^2}
&=\omega_{(1,0)}\omega_{(1,0)}^*+\frac12\omega_{(1,1)}\omega_{(1,1)}^*+
\omega_{(0,1)}\omega_{(0,1)}^*+\frac12\omega_{(-1,1)}\omega_{(-1,1)}^*\\
&=H_{(1,0)}(\omega)+H_{(1,1)}(\omega)+H_{(0,1)}(\omega)+H_{(-1,1)}(\omega).
\end{align*}
The prefactor $\sqrt{3}/2$ in $B(\omega)$ originates from the equality $\sin(\pm2\pi/3)=\sin(\mp 4\pi/3)=\pm\sqrt{3}/2$.
Recall that the deterministic Poisson system preserves the Hamiltonian $H$. In addition, it admits two Casimir functions, defined by
\[
C_1(\omega)=\frac12\sum_{{\bf n}\in K}\omega_{\bf n}\omega_{\bf n}^\star=\omega_{(1,0)}\omega_{(1,0)}^*+\omega_{(1,1)}\omega_{(1,1)}^*+
\omega_{(0,1)}\omega_{(0,1)}^*+\omega_{(-1,1)}\omega_{(-1,1)}^*
\]
and
\[
C_2(\omega)=\sum_{{\bf n},{\bf m}\in K}
\cos{\left( \frac{2\pi}{3}({\bf n}\times {\bf m})\right)}\omega_{\bf n}\omega_{\bf m}\omega_{-\bf n-\bf m}.
\]

In this article, we consider the following stochastic version of the sine--Euler equations 
\begin{align}\label{stochSE}
\diff \omega&=B(\omega)\left(\nabla H(\omega)\,\diff t+
\sigma_{(1,0)}\nabla \widehat H_{(1,0)}(\omega)\circ\,\diff W_{(1,0)}(t)+\sigma_{(1,1)} \nabla \widehat H_{(1,1)}(\omega)\circ\,\diff W_{(1,1)}(t) \right.\nonumber\\
&\quad\left.+\sigma_{(0,1)}\nabla \widehat H_{(0,1)}(\omega)\circ\,\diff W_{(0,1)}(t) 
+\sigma_{(-1,1)}\nabla \widehat H_{(-1,1)}(\omega)\circ\,\diff W_{(-1,1)}(t)\right),
\end{align}
where, for $\bf k\in\{(1,0),(1,1),(0,1),(-1,1)\}$, the Hamiltonian functions are $\widehat H_{\bf k}(\omega)=H_{\bf k}(\omega)=\frac{\omega_{\bf k}\omega_{\bf k}^\star}{2|\bf k|^2}$, 
$\sigma_{\bf k}\ge 0$ are nonnegative real numbers, and $W_{\bf k}$ are independent standard Wiener process.

The SDE~\eqref{stochSE} is a stochastic Lie--Poisson system. Owing to Proposition~\ref{propo:global}, it is globally well-posed, since the Casimir function $C_1$ has compact level sets.

Let us provide a possible physical interpretation of the stochastic sine--Euler system~\eqref{stochSE}: 
the Hamiltonian function can be decomposed as $H(\omega)=H_{(1,0)}(\omega)+H_{(1,1)}(\omega)+H_{(0,1)}(\omega)+H_{(-1,1)}(\omega)$, 
where $H_{\bf k}$ can be interpreted as the kinetic energy of the modes, or periodic waves, which are parallel to the direction $\bf k\in\{(1,0),(1,1),(0,1),(-1,1)\}$. In the stochastic version introduced above, noise acts independently on each of these modes.
\end{example}

To conclude this subsection on examples of stochastic (Lie--)Poisson systems, let us mention other systems which may be treated using the techniques developed in this work: appropriate spatial discretisations of (stochastic) Vlasov--Poisson equations~\cite{tyranowski2021stochastic}, 
random perturbations of the full rigid body~\cite{MR2970274} (with a rotation matrix giving the orientation of the body in a fixed frame), stochastic models of fluid dynamics~\cite{MR3800250}, or reduced models based on a Gaussian wavepacket of 
the time-dependent $N$-body Schr\"odinger equation~\cite[Chap VII.6.1-VII.6.4]{HLW02}. 
These examples are not considered in this article and may be studied in future works.

\subsection{The Poisson map property}\label{sec:Pmap}

We proceed with showing that the flow of stochastic Poisson systems~\eqref{prob} satisfies a property 
which is a generalisation of the symplecticity property for deterministic, respectively stochastic, Hamiltonian systems, see e.g \cite{HLW02}, respectively~\cite{MR2069903}.

Let $ D_y$ denote the Jacobian operator, i.e.
$ D_y f(y) = (\partial_j f_i(y))_{1\le i,j\le d} $ for a smooth function $f\colon\R^d\to \R^d$. 
The transpose of a matrix $M$ is denoted by $M^T$.

\begin{definition}\label{def:Pmap}
Let $U\subset\R^d$ be an open set. A transformation $\varphi\colon U\to\R^d$ is called a \emph{Poisson map} for the problem \eqref{prob}, if one has, almost surely, for all $y\in\R^d$,
$$
D_y\varphi(y)B(y)D_y\varphi(y)^T=B(\varphi(y)).
$$
\end{definition}
Observe that a composition of Poisson maps is a Poisson  map. This property will be used in the design of stochastic Poisson splitting integrators in the next sections. 

The main result of this section is stated below.
\begin{theorem}\label{th:flowP}
Introduce the flow $(t,y)\mapsto \varphi_t(y)$ of the stochastic Poisson system~\eqref{prob} 
with coefficients of class $\mathcal C^3$. Assume that the flow 
is globally well defined and of class $\mathcal C^1$ with respect to the variable $y$. 
Then, for all $t\ge 0$, $\varphi_t$ is a Poisson map: almost surely, for all $y\in\R^d$, one has
$$
D_y\varphi_t(y)B(y)D_y\varphi_t(y)^T=B(\varphi_t(y)).
$$
\end{theorem}

This result can be compared with the similar statement in \cite[Th. 2.1]{hong2020structurepreserving}. 
To prove their result, the authors of \cite{hong2020structurepreserving} use 
the Darboux--Lie theorem to perform a change of coordinates, and replace the considered stochastic Poisson system by a stochastic 
canonical Hamiltonian system. 
Our strategy to prove Theorem~\ref{th:flowP} is more direct and mimics the analysis of the deterministic case, as proposed in \cite[Ex. 6.VII.4]{HLW02}.

\begin{proof}
For all $t\ge 0$ and all $y\in\R^d$, let $\Phi_t(y)= D_y \varphi_t(y)$. To simplify notation, 
$\varphi_t$ stands for $\varphi_t(y)$ in the computations below. For convenience, we set $\widehat{H}_0=H$ in this proof.

First, it is well-known that $t\mapsto \Phi_t(y)$ is the solution of the variational equation
\[
\diff \Phi_t=
D_y
\bigl( B(\varphi_t)\nabla H(\varphi_t)\bigr)\Phi_t\diff t+\sum_{k=1}^m D_y\bigl(B(\varphi_t)\nabla\widehat H_k(\varphi_t)\bigr)\Phi_t\circ \diff W_k(t)
\]
with $\Phi_0=Id$,
see for instance~\cite[Theorem 2.3.32]{MR1723992}.
We claim that, in the case of the stochastic Poisson system~\eqref{prob}, the variational equation above may be rewritten as
\begin{equation}\label{SDE.variational}
\diff \Phi_t= \mathcal{S}_0(\varphi_t)\Phi_t\diff t +
\sum_{k=1}^m \mathcal{S}_k(\varphi_t) \Phi_t\circ \diff W_k(t),
\end{equation}
where, for all $k=0,\ldots,m$, one has 
\begin{align*}
 \mathcal{S}_k(\varphi_t) &= \mathcal{M}_k(\varphi_t) + \mathcal{L}_k(\varphi_t)\\
 \mathcal{M}_k(\varphi_t)&=
 \left[ \big(\partial_1 B(\varphi_t)\big) \nabla\widehat H_k(\varphi_t) \mid \ldots \mid 
 \big(\partial_d B(\varphi_t)\big) \nabla\widehat  H_k(\varphi_t)\right]
 \\
\mathcal{L}_k(\varphi_t)&=B(\varphi_t)\nabla^2\widehat H_k(\varphi_t).
\end{align*}
The identification of the matrices $\mathcal{S}_k(\varphi_t)$ above is based on the following computation: if $[M]_{i,j}$ denotes the $(i,j)$-th entry of a matrix $M$, then applying the chain rule (recall that the SDE is interpreted in the Stratonovich sense) yields
\begin{align*}
& \left[D_y\big( B(\varphi_t)\nabla \widehat H_k(\varphi_t) \big) \right]_{i,j}
=\sum_{l=1}^d\frac{\partial}{\partial y_j}
\big(B_{il}(\varphi_t)\partial_l \widehat H_k(\varphi_t)\big) 
\\
&= 
\sum_{l=1}^d 
\left(\sum_{n=1}^d \big(\partial_n B_{il}(\varphi_t)\big)
\left[\frac{\partial \varphi_t}{\partial y_j}\right]_n \right) 
\partial_l \widehat H_k(\varphi_t)
+
\sum_{l=1}^dB_{il}(\varphi_t)
\sum_{n=1}^d 
\big(\partial_l \partial_n \widehat H_k(\varphi_t) \big)
\left[\frac{\partial \varphi_t}{\partial y_j}\right]_n
\\
&=
\sum_{n=1}^d 
\left(
\sum_{l=1}^d \big(\partial_n B_{il}(\varphi_t)\big) \partial_l \widehat H_k(\varphi_t)
\right)
\left[\frac{\partial \varphi_t}{\partial y_j}\right]_n 
+
\sum_{n=1}^d 
\left(
\sum_{l=1}^dB_{il}(\varphi_t)
\partial_l \partial_n \widehat H_k(\varphi_t) 
\right)
\left[\frac{\partial \varphi_t}{\partial y_j}\right]_n
\\
&=
\sum_{n=1}^d 
\Big[\big(\partial_n B(\varphi_t)\big) \nabla \widehat H_k(\varphi_t)\Big]_i
[\Phi_t]_{n,j} 
+
\sum_{n=1}^d 
\Big[B(\varphi_t)\nabla^2\widehat H_k(\varphi_t) \Big]_{i,n}
[\Phi_t]_{n,j}.
\end{align*}

Let us now define
\begin{align}
\label{Def.delta}
 \delta(t):=\Phi_t B(y)\Phi_t^T-B(\varphi_t).
\end{align}
Since $\varphi_0(y)=y$ for all $y\in\R^d$ and $\Phi_0=Id$, one has $\delta(0)=0$. To prove that $\varphi_t$ is a Poisson map for all $t\ge 0$ almost surely, it suffices to check that $\delta(t)=0$ for all $t\ge 0$. 

We will show that $t\mapsto\delta(t)$ is a solution of the linear equation
\begin{align}
\label{SDE.delta}
\diff \delta(t)&=
\bigg(\mathcal{S}_0(\varphi_t) \diff t + \sum_{k=1}^m \mathcal{S}_k(\varphi_t) \circ \diff W_k(t)\bigg)
\delta(t)
\\
\notag
&\qquad
+
\delta(t)
\bigg(\mathcal{S}_0(\varphi_t) \diff t + \sum_{k=1}^m \mathcal{S}_k(\varphi_t) \circ \diff W_k(t)\bigg)^T.
\end{align}
Since the initial value is $\delta(0)=0$, by uniqueness of the solution we obtain $\delta(t)=0$ for all $t\ge 0$, thus $\varphi_t$ is a Poisson map for all $t\ge 0$.

It remains to prove that $\delta$ is indeed a solution of~\eqref{SDE.delta}. 
On the one hand, 
from the variational equation~\eqref{SDE.variational} for $\Phi_t$ above, applying the product rule yields
\begin{align}\label{SDE.delta_bis}
\diff \delta(t)&=
\bigg(\mathcal{S}_0(\varphi_t) \diff t
+ \sum_{k=1}^m \mathcal{S}_k(\varphi_t) \circ \diff W_k(t)\bigg)\Phi_t B(y)\Phi_t^T
\\
\notag
&\qquad
+\Phi_tB(y)\Phi_t^T \bigg(\mathcal{S}_0(\varphi_t) \diff t
+\sum_{k=1}^m \mathcal{S}_k(\varphi_t) \circ \diff W_k(t)\bigg)^T
-\diff B(\varphi_t).
\end{align}
On the other hand, one has the identity
\begin{align}
\label{SDE.B}
\diff B(\varphi_t)
&=
\bigg(\mathcal{S}_0(\varphi_t) \diff t + \sum_{k=1}^m \mathcal{S}_k(\varphi_t) \circ \diff W_k(t)\bigg)
B(\varphi_t)
\\
\notag
&\qquad
+
B(\varphi_t)
\bigg(\mathcal{S}_0(\varphi_t) \diff t + \sum_{k=1}^m \mathcal{S}_k(\varphi_t) \circ \diff W_k(t)\bigg)^T.
\end{align}
Combining~\eqref{SDE.delta_bis} and~\eqref{SDE.B} provides the claim that $t\mapsto \delta(t)$ is solution of~\eqref{SDE.delta}. The proof of the identity~\eqref{SDE.B} requires to exploit the assumptions on the structure matrix $B$ as follows. First, note that the $(i,j)$-th entry of the left-hand side of \eqref{SDE.B} satisfies
\begin{align*}
\big[\diff B(\varphi_t)\big]_{i,j} 
&=
\nabla B_{ij}^T(\varphi_t){\circ\,}\diff \varphi_t
\\
&=
\nabla B_{ij}^T(\varphi_t)
\left(
B(\varphi_t)\nabla H(\varphi_t)\,\diff t+\sum_{k=1}^mB(\varphi_t)\nabla\widehat H_k(\varphi_t)\circ\,\diff W_k(t)
\right),
\end{align*}
using the chain rule. Recall that $\mathcal{S}_k(\varphi_t)=\mathcal{M}_k(\varphi_t)+\mathcal{L}_k(\varphi_t)$.

On the one hand, all the terms involving 
$\mathcal{L}_k(\varphi_t)=B(\varphi_t)\nabla^2\widehat H_k(\varphi_t)$ appearing on the right-hand side of~\eqref{SDE.B} vanish: this follows from 
the symmetry of the Hessians $\nabla^2\widehat H_k$ and the skew-symmetry of $B$ which yields (omitting the argument ``$(\varphi_t)$'' everywhere)
\begin{align*}
\mathcal{L}_kB+B\mathcal{L}_k^T
&=B\nabla^2\widehat H_kB+
B\nabla^2\widehat H_kB^T
=0 
\end{align*}
for $k=0,1,\ldots,m$. 

On the other hand, using the skew-symmetry of $B$ and the Jacobi identity
(and omitting ``$(\varphi_t)$'' again), one has
\begin{align*}
\Big[\mathcal{S}_k B 
 + B \mathcal{S}_k^T
 \Big]_{ij}
&=
\Big[\mathcal{M}_k B 
 + B \mathcal{M}_k^T
 \Big]_{ij}
 \\
 &=
 \sum_{l=1}^d  \big[\mathcal{M}_k\big]_{il} B_{lj} +
 \sum_{l=1}^d  B_{il} \big[\mathcal{M}_k\big]_{jl}
 \\
 &=
 \sum_{l=1}^d  
 \sum_{n=1}^d \big(\partial_l B_{in}\big) \partial_n\widehat  H_k
 B_{lj} 
 +
 \sum_{l=1}^d  B_{il} 
 \sum_{n=1}^d \big(\partial_l B_{jn}\big) \partial_n\widehat  H_k
 \\
 &=
 \sum_{n=1}^d \sum_{l=1}^d \Big(
 \big(\partial_l B_{in}\big) B_{lj} 
 + \big(\partial_l B_{jn}\big) B_{il} 
 \Big) \partial_n\widehat  H_k
 \\
 &=
 -
 \sum_{n=1}^d \sum_{l=1}^d \Big(
 \big(\partial_l B_{ni}\big) B_{lj} 
 + \big(\partial_l B_{jn}\big) B_{li} 
 \Big) \partial_n\widehat  H_k
  \\
 &=
 \sum_{n=1}^d \sum_{l=1}^d \Big(
 \big(\partial_l B_{ij}\big) B_{ln} 
 \Big) \partial_n\widehat  H_k
  \\
 &=
\nabla B_{ij}^T B \nabla\widehat  H_k
\end{align*}
for all $k=0, \ldots, m$.

This concludes the proof of~\eqref{SDE.B}, which as already explained above gives the identity $\delta(t)=0$ for all $t\ge 0$. In the end, this concludes the proof that $\varphi_t$ is a Poisson map for all $t\ge 0$, almost surely.

%
\end{proof}

One of the objectives of this work is to design and study integrators for the stochastic Poisson system~\eqref{prob}, which preserve its geometric structure, namely the Poisson map property (Definition~\ref{def:Pmap} and Theorem~\ref{th:flowP}), and the preservation of the Casimir functions. This leads to define and analyze so-called stochastic Poisson integrators, see Section~\ref{sec:integrators}.

\subsection{Stochastic Poisson systems obtained by diffusion approximation}\label{sec:multi}

The goal of this subsection is to describe a class of multiscale stochastic systems, depending on a parameter $\epsilon\in(0,1)$, such that the stochastic Poisson system~\eqref{prob} is obtained as a limit when $\epsilon\to 0$. 
In particular, this approximation result justifies why considering stochastic Poisson systems 
with a Stratonovich interpretation of the noise is relevant.

For all $\epsilon\in(0,1)$, introduce the multiscale system
\begin{equation}\label{prob_eps}
\left\lbrace
\begin{aligned}
\diff y^\epsilon(t)&=B(y^\epsilon(t))\nabla H(y^\epsilon(t))\,\diff t+\sum_{k=1}^mB(y^\epsilon(t))\nabla \widehat H_k(y^\epsilon(t))\frac{\xi_k^\epsilon(t)}{\epsilon}\diff t,\\
\diff \xi_k^\epsilon(t)&=-\frac{\xi_k^\epsilon(t)}{\epsilon^2}\diff t+\frac{1}{\epsilon}\diff W_k(t),\quad k=1,\ldots,m,
\end{aligned}
\right.
\end{equation}
with initial values $y^\epsilon(0)=0$ and $\xi_k^\epsilon(0)=0$, for all $k=1,\ldots,m$. Note that $\xi_k^\epsilon$ is an Ornstein--Uhlenbeck process, for all $k=1,\ldots,m$. Compared with the stochastic Poisson system~\eqref{prob}, $y^\epsilon$ solves a random ordinary differential equation, since the noise $\circ \diff W_k(t)$ is replaced by $\frac{\xi_k^\epsilon(t)}{\epsilon}\diff t$. Observe that if $C$ is a Casimir function of the stochastic Poisson system~\eqref{prob}, then one has $dC(y^\epsilon(t))=0$ owing to the chain rule, thus $C(y^\epsilon(t))=C(y^\epsilon(0))$ for all $t\ge 0$. Assuming that the Casimir function $C$ has compact level sets, like in Proposition~\ref{propo:global}, ensures that the system~\eqref{prob_eps} is globally well-posed, for all $\epsilon\in(0,1)$, and that
\[
\underset{\epsilon\in(0,1)}\sup~\underset{t\ge 0}\sup~\|y^\epsilon(t)\|\le R(y_0).
\]
When $\epsilon\to 0$, one has the following result.
\begin{proposition}\label{propo:multi}
Assume that the stochastic Poisson system~\eqref{prob} admits a Casimir function $C$  such that 
for all $c\in\mathbb R$, the level sets $\{y\in\mathbb{R}^d\,\colon\, C(y)=c\}$ are compact.

For all $t\ge 0$,
\[
y^{\epsilon}(t)\underset{\epsilon\to 0}\to y(t),
\]
where the convergence holds in distribution ($y^\epsilon(t)$ and $y(t)$ are random variables). In addition, for all $y_0\in\R^d$, all $T\in(0,\infty)$ and any function $\phi$ of class $\mathcal{C}^3$, there exists a real number $\CC(T,y_0,\phi)\in(0,\infty)$ such that for all $\epsilon\in(0,1)$
\[
\underset{0\le t\le T}\sup~\big|\mathbb E[\phi(y^\epsilon(t))]-\mathbb E[\phi(y(t))]\big|\le \CC(T,y_0,\phi)\epsilon.
\]
\end{proposition}
Proposition~\ref{propo:multi} is a diffusion approximation result: roughly, a diffusion process is the solution of a SDE driven by a Wiener process. 
Here, the diffusion process $y$ is approximated by the solutions $y^\epsilon$ of ODEs.

The proof of Proposition~\ref{propo:multi} is omitted. Indeed, this is a standard result in the literature: we refer for instance to the monograph~\cite[Chapter~11 and~18]{MR2382139} for a presentation of homogenization techniques, and references therein for a historical perspective. Proposition~\ref{propo:multi} fits in the class of Wong--Zakai approximation results (where a smooth approximation of a Wiener noise leads to a SDE driven by Stratonovich noise). We also refer to~\cite[Proposition~2.6]{BRR} for a similar statement (with $m=1$), and references therein for ideas of proof. Note that the weak error estimate can be obtained using a variant of the proof of~\cite[Proposition~2.4]{BRR}, decomposing the error in terms of solutions of Kolmogorov and Poisson equations. The details of the proof are left to the interested readers.

Let us provide an heuristic argument which justifies the diffusion approximation result: for all $k=1,\ldots,m$, one has the identity
\[
\frac{\xi_k^\epsilon(t)}{\epsilon}\diff t=\diff W_k(t)-\epsilon \diff \xi_k^\epsilon(t).
\]
Then, the contribution of $\epsilon \diff \xi_k^\epsilon(t)$ vanishes in the limit $\epsilon\to 0$, and only the contribution $\diff W_k(t)$ remains at the limit. There may be different interpretations of the noise at the limit: at least, It\^o and Stratonovich interpretations are possible candidates. However, recall that Casimir functions $C$ of the stochastic Poisson system~\eqref{prob} are preserved by the solution $y^\epsilon$ of~\eqref{prob_eps}, for all $\epsilon>0$. The It\^o interpretation of the noise is not consistent with this preservation property, whereas the Stratonovich one is, due to the chain rule. As a consequence, the Stratonovich interpretation is the natural candidate for the diffusion approximation limit. Checking rigorously that indeed $y^\epsilon(t)\to y(t)$ in distribution requires additional arguments which are omitted in this work.

\begin{remark}
Let $(t,y,\xi_1,\ldots,\xi_m)\mapsto \varphi^\epsilon(t,y,\xi_1,\ldots,\xi_m)$ define the flow map associated with~\eqref{prob_eps}. Then for all $t\ge 1$, $\epsilon\in(0,1)$ and all $\xi_1,\ldots,\xi_m\in\R$, the mapping
\[
y\in \R^d\mapsto \varphi^\epsilon(t,y,\xi_1,\ldots,\xi_m)
\]
is a Poisson map in the sense of Definition~\ref{def:Pmap}. This may be proved by modifications of the proof of Theorem~\ref{th:flowP}, using the chain rule. The details are left to the reader.
\end{remark}

The multiscale system~\eqref{prob_eps} has components evolving at different time scales: the component $y^\epsilon$ evolves at a time scale of order ${\rm O}(1)$, whereas the Ornstein--Uhlenbeck processes $\xi_1^\epsilon,\ldots,\xi_m^\epsilon$ evolve at a time scale of order ${\rm O}(\epsilon^{-2})$. The definition of effective integrators for the multiscale system~\eqref{prob_eps}, which avoid prohibitive time step size restrictions of the type $h={\rm O}(\epsilon^2)$, and which lead to consistent discretisation of $y(t)$ when $\epsilon\to 0$, is a crucial and challenging problem. This question is briefly studied in Section~\ref{APsplit} below: we define so-called asymptotic preserving schemes (in the spirit of~\cite{BRR}), 
employing the preservation of the geometric structure satisfied by the stochastic Poisson integrators introduced in the next section.

\section{Stochastic Poisson integrators}\label{sec:integrators}

In Section~\ref{sec:poisson}, we have proved that the flow of a stochastic Poisson system of the type~\eqref{prob} satisfies two key properties: it is a Poisson map (see Theorem~\ref{th:flowP}) and it preserves Casimir functions which are associated with the structure matrix $B$. Having the methodology of geometric numerical integration in mind, this motivates us to introduce the concept of a stochastic Poisson integrator for the stochastic Poisson system~\eqref{prob}, see Definition~\ref{defPI}. We then present and analyse a general strategy to derive 
efficient stochastic Poisson integrators, based on a splitting technique, which can be implemented easily for some stochastic Lie--Poisson systems. 
We then proceed with a convergence analysis of the proposed splitting integrators: Theorem~\ref{thm-general} states that, under appropriate conditions, the scheme has in general strong and weak convergence rates equal to $1/2$ and $1$, respectively. First, the analysis is performed for an auxiliary problem~\eqref{auxSDE} with globally Lipschitz continuous nonlinearities, see Lemma~\ref{lemm-aux}. Second, if the system admits a Casimir function with compact level sets, the auxiliary convergence result is applied to get strong and weak error estimates for the SDE~\eqref{prob}. Finally, we show that the proposed stochastic Poisson integrators based on a splitting technique satisfy an asymptotic preserving property when considering the multiscale SDE~\eqref{prob_eps} in the diffusion approximation regime.

\subsection{Definition and splitting integrators for stochastic (Lie--)Poisson systems}\label{ssec:LP}

Let us recall that symplectic, respectively Poisson, integrators preserve the key features of deterministic and stochastic Hamiltonian systems, respectively deterministic Poisson systems. Such geometric numerical integrators offer various benefits over 
classical time integrators in the deterministic setting, see for instance \cite{HLW02,rl,bcGNI15}. 
We shall now state the definition and study the properties of stochastic Poisson integrators for stochastic Poisson systems~\eqref{prob}. On the one hand, this extends the definition and analysis of deterministic Poisson integrators (see~\cite[Th. 3.1]{hong2020structurepreserving} for another approach). On the other hand, this extends the definition and analysis of stochastic symplectic integrators for stochastic Hamiltonian systems.


We first consider general stochastic Poisson integrators, and then focus the discussion on a class of splitting integrators.

\subsubsection{Stochastic Poisson integrators}\label{sssec:PI}
The following notation is used below. The time step size is denoted by $h>0$. 
A numerical scheme is defined as follows: for all $n\ge 1$,
\begin{equation}\label{eq:integrator}
y^{[n]}=\Phi_h(y^{[n-1]},\Delta_n W_1,\ldots,\Delta_n W_m),
\end{equation}
with Wiener increments $\Delta_nW_k=W_k(nh)-W_k((n-1)h)$, $k=1,\ldots,m$. The Wiener increments are independent centered real-valued Gaussian random variables with variance $h$. The mapping $\Phi_h$ is referred to as the integrator.

\begin{definition}\label{defPI}
A numerical scheme~\eqref{eq:integrator} for the stochastic Poisson system~\eqref{prob} is called 
a \emph{stochastic Poisson integrator} if
\begin{itemize}
\item for all $h>0$ and all $\Delta w_1,\ldots,\Delta w_m\in \R$, 
the mapping $$y\mapsto \Phi_h(y,\Delta w_1,\ldots,\Delta w_m)$$ is a Poisson map (in the sense of Definition~\ref{def:Pmap}),
\item if $C$ is a Casimir of the stochastic Poisson system~\eqref{prob}, then $\Phi_h$ preserves $C$, precisely
\[
C(\Phi_h(y,\Delta w_1,\ldots,\Delta w_m))=C(y)
\]
for all $y\in\R^d$, $h>0$ and $\Delta w_1,\ldots,\Delta w_m\in\R$.
\end{itemize}
\end{definition}
As in the deterministic case, it is seen that standard integrators like the Euler--Maruyama scheme are 
not (stochastic) Poisson integrators. In addition, it is a difficult task to construct Poisson integrators  for the general Poisson systems, see \cite[Chapter VII.4.2]{HLW02} for deterministic problems. Therefore, the design of stochastic Poisson integrators requires to exploit the special structure for each considered problem. In this article, we focus 
on constructing and analyzing stochastic Poisson integrators for stochastic Lie--Poisson systems. More precisely, we propose 
explicit Poisson integrators for a large class of stochastic 
Lie--Poisson systems using a splitting strategy. In Section~\ref{sect-numexp}, we will exemplify this strategy for three models introduced in Section~\ref{sec:poisson}: 
the stochastic Maxwell--Bloch equations (Example~\ref{expl-MB} and Subsection~\ref{ssec:PMB}), 
the stochastic free rigid body equations~\eqref{srb} (Example~\ref{expl-SRB} and Subsection~\ref{ssec:PRB}), 
as well as the stochastic sine--Euler equations (Example~\ref{expl-SE} and Subsection~\ref{ssec:PSE}). 

\subsubsection{Splitting integrators for stochastic Poisson systems}\label{sssec:SP}

We first propose an abstract splitting integrator for general stochastic Poisson systems~\eqref{prob}. 
We then focus on stochastic Lie--Poisson systems~\eqref{slp} and propose implementable stochastic Poisson integrators for this class of SDEs, which includes 
the three examples mentioned above.

%

The key observation made in \cite[p.3044]{MR1246065} is that a wide class of deterministic Lie--Poisson systems can be split into subsystems which are all linear. This was used in \cite{MR1246065}
for the construction of very efficient geometric integrators for deterministic Lie--Poisson systems.
Inspired by~\cite{MR1246065}, we propose and analyse efficient explicit Poisson integrators for stochastic Lie--Poisson systems. On an abstract level, our splitting approach is not restricted to Lie--Poisson systems and could also be applied to general stochastic Poisson systems~\eqref{prob}.



Let us consider a stochastic Poisson system of the type~\eqref{prob}, and assume that the Hamiltonian $H$ can be split as follows:
\[
H=\sum_{k=1}^{p}H_k.
\]
for some $p\ge 1$, where the Hamiltonian functions $H_1,\ldots,H_p$ have the same regularity as $H$.

To define the abstract splitting schemes for~\eqref{prob}, it is convenient to define the flows associated to the subsystems:
\begin{itemize}
\item for each $k=1,\ldots,p$, let $(t,y)\in\R^+\times\R^d\mapsto \varphi_k(t,y)$ be the flow associated with the ordinary differential equation $\dot{y}_k=B(y_k)\nabla H_k(y_k)$;
\item for each $k=1,\ldots,m$, let $(t,y)\in\R\times\R^d\mapsto \widehat{\varphi}_k(t,y)$ be the flow associated with the ordinary differential equation $\dot{y}_k=B(y_k)\nabla \widehat H_k(y_k)$.
\end{itemize}
Note that it is sufficient to consider $\varphi_1(t,\cdot),\ldots,\varphi_p(t,\cdot)$ for $t\ge 0$, however the mappings $\widehat{\varphi}_1(t,\cdot),\ldots,\widehat{\varphi}_k(t,\cdot)$ need to be considered for $t\in\R$.

Below, we shall also use the notation $\exp(hY_{H_k})=\varphi_k(h,\cdot)$ and 
$\exp(hY_{\widehat H_k})=\widehat \varphi_k(h,\cdot)$, where $Y_{H_k}=B\nabla H_k$, resp. $Y_{\widehat H_k}=B\nabla\widehat H_k$, to denote the vector fields of the corresponding differential equations. 
For the definition of the splitting integrators below, it is essential to note 
that the exact solution of the Stratonovich stochastic differential equation 
$\diff y_k=B(y_k)\nabla \widehat H_k(y_k)\circ \diff W_k(t)$ is given by 
$y_k(t)=\widehat \varphi_k(W_k(t),y_k(0))$.

As explained above, closed-form expressions for the flows $\varphi_k$ and $\widehat\varphi_k$ are unknown in general but can be obtained for a wide class of stochastic Lie--Poisson systems
\begin{equation}\label{slp}
\left\lbrace
\begin{aligned}
&\diff y(t)=B(y(t))\nabla H(y(t))\,\diff t+
B(y(t))\sum_{k=1}^m\nabla \widehat H_k(y(t))\circ\,\diff W_k(t),\\
&B_{ij}(y)=\sum_{k=1}^db_{ji}^ky_k\quad\text{for}\quad i,j=1,\ldots,d, 
\end{aligned}
\right.
\end{equation}
where the structure matrix $B(y)$ depends linearly on $y$. For the examples of stochastic Lie--Poisson systems introduced in Section~\ref{sec:examples}, 
below we design explicit splitting schemes which can be easily implemented by a splitting strategy. In the sequel, we analyse the geometric and convergence properties of splitting integrators in an abstract framework, where it is not assumed that the flows $\varphi_k$ and $\hat{\varphi}_k$ can be computed exactly. In particular, the assumption that the structure matrix $B$ depends linearly on $y$ is not required in the analysis. Note also that expressions of the flows may also be known for some stochastic Poisson systems which are not Lie--Poisson problems, in which case the abstract analysis would also be applicable.



We are now in position to define splitting integrators for the stochastic Poisson system~\eqref{prob}, 
which will be exemplified in the case of stochastic Lie--Poisson systems~\eqref{slp}. This general splitting integrator is given by
\begin{align}\label{slpI}
\Phi_h(\cdot)&=\Phi_h(\cdot,\Delta W_1,\ldots,\Delta W_m)=\exp(hY_{H_p})\circ\exp(hY_{H_{p-1}})\circ\ldots\circ\exp(hY_{H_1}) \nonumber\\
&\circ\exp(\Delta W_mY_{\widehat H_m})\circ\exp(\Delta W_{m-1}Y_{\widehat H_{m-1}})
\circ\ldots\circ\exp(\Delta W_1Y_{\widehat H_1}).
\end{align}

It is immediate to check the following fundamental result.
\begin{proposition}\label{propo:sPi}
The splitting integrator~\eqref{slpI} is a stochastic Poisson integrator, in the sense of Definition~\ref{defPI}, 
for the stochastic Poisson system~\eqref{prob}.
\end{proposition}
\begin{proof}
Observe that for any $h>0$ and any real numbers $\Delta w_1,\ldots,\Delta w_m$, the mapping $\Phi_h(\cdot,\Delta w_1,\ldots,\Delta w_m)$ is a composition of flow maps $\varphi_k(h,\cdot)$, $k=1,\ldots,p$ and $\widehat\varphi_k(\Delta w_k,\cdot)$, $k=1,\ldots,m$.
Owing to Theorem~\ref{th:flowP}, all of these flow maps are Poisson maps (since they are flow maps of either deterministic or stochastic Poisson systems).

In addition, if $C$ is a Casimir function of the stochastic Poisson system~\eqref{prob}, then $C$ is preserved by each of the flow maps $\varphi_k(h,\cdot)$, $k=1,\ldots,p$ and $\widehat\varphi_k(\Delta w_k,\cdot)$, $k=1,\ldots,m$. Indeed, recall that the definition of a Casimir only depends on the structure matrix $B$, and not on the Hamiltonian functions, and all the associated vectors fields are of the type $Y_{H_k}=B\nabla H_k$ and $Y_{\widehat H_k}=B\nabla\widehat H_k$: the associated flow maps thus preserve $C$.
As a consequence, the general splitting integrator $\Phi_h(\cdot,\Delta w_1,\ldots,\Delta w_m)$ also preserves the Casimir functions $C$ of the stochastic Poisson system~\eqref{prob}.
This concludes the proof that the splitting scheme~\eqref{slpI} is a stochastic Poisson integrator.
\end{proof}

Before proceeding to the convergence analysis for the splitting integrators~\eqref{slpI}, 
it is worth exploiting the fact that they are stochastic Poisson integrators 
to state that the numerical solution remains bounded if the considered stochastic Poisson system~\eqref{prob} admits a Casimir function $C$ with compact level sets. We refer to Proposition~\ref{propo:global} for the statement of a similar result for the solution of the stochastic Poisson system~\eqref{prob}, in  particular the assumption on compact level sets.

\begin{proposition}\label{propo:numerik}
Assume that the stochastic Poisson system~\eqref{prob} admits a Casimir function $C$ which has compact level sets. Consider the stochastic Poisson integrator $y^{[n+1]}=\Phi_h(y^{[n]})$ 
given by \eqref{slpI}. Then, for any initial condition $y^{[0]}=y_0\in\mathbb R^d$, for all $t\geq0$, 
almost surely one has the following bound for the numerical solution
\[
\underset{h>0}\sup~\underset{n\ge 0}\sup~\|y^{[n]}\|\le  R(y^{[0]})=\max_{y\in\mathbb R^d, C(y)=C(y^{[0]})}\|y\|.
\]
\end{proposition}
\begin{proof}
The splitting scheme~\eqref{slpI} is a stochastic Poisson integrator (owing to Proposition~\ref{propo:sPi}), thus it preserves the Casimir function $C$: therefore for all $n\ge 1$,
\[
C(y^{[n]})=C(y^{[n-1]})=\ldots=C(y^{[0]}).
\]
Note that $R(y^{[0]})<\infty$, since by assumption the Casimir function $C$ has compact level sets. Therefore one obtains
\[
\|y^{[n]}\|\le R(y^{[0]})
\]
for all $n\ge 0$ by the definition of $R(y^{[0]})$. This concludes the proof.
\end{proof}


\begin{remark}\label{rem-xchange}
The stochastic Poisson integrator~\eqref{slpI} employs a  Lie--Trotter splitting strategy. 
Changing the orders of integration of the deterministic and stochastic parts yields the following alternative to~\eqref{slpI}
\begin{align*}
\Phi_h(\cdot)&=\Phi_h(\cdot,\Delta W_1,\ldots,\Delta W_m)\\
=&\exp(\Delta W_mY_{\widehat H_m})\circ\exp(\Delta W_{m-1}Y_{\widehat H_{m-1}})\circ\ldots\circ\exp(\Delta W_1Y_{\widehat H_1})\\
&\circ\exp(hY_{H_p})\circ\exp(hY_{H_{p-1}})\circ\ldots\circ\exp(hY_{H_1}).
\end{align*}
This alternative scheme is also a stochastic Poisson integrator, which satisfies Propositions~\ref{propo:sPi} and~\ref{propo:numerik}. The theoretical analysis of that scheme and associated numerical experiments are not reported in the present article.
\end{remark}

\begin{remark}\label{rem:order2}
A numerical method of weak order $2$ can be designed using the strategy developed in~\cite{MR3570281}. The integrator is a combination of three mappings and depends on an additional random variable $\gamma_n$, uniformly distributed in $\{-1,1\}$:
\begin{equation}\label{eq:order2}
y^{[n]}=\Phi_{h,\gamma_n}(y^{[n-1]})=\Phi_{h/2}^{det,S}\circ \Phi_{h,\gamma_n}^{sto}(\cdot,\Delta_n W_1,\ldots,\Delta_n W_m)\circ \Phi_{h/2}^{det,S}(y^{[n-1]}),
\end{equation}
where
\[
\Phi_{h/2}^{det,S}=\exp(\frac{h}{4}Y_{H_{1}})\circ\ldots\circ\exp(\frac{h}{4}Y_{H_{p-1}})\circ\exp(\frac{h}{2}Y_{H_p})\circ\exp(\frac{h}{4}Y_{H_{p-1}})\circ\ldots\circ\exp(\frac{h}{4}Y_{H_1})
\]
is obtained using a Strang splitting integrator with time step size $h/2$ for the deterministic part of the equation, and
\[
\Phi_{h,\gamma_n}^{sto}(\cdot,\Delta_n W_1,\ldots,\Delta_n W_m)=\begin{cases}
\exp(\Delta W_mY_{\widehat H_m})\circ\ldots\circ\exp(\Delta W_1Y_{\widehat H_1}),\quad \gamma_n=1\\
\exp(\Delta W_1Y_{\widehat H_1})\circ\ldots\circ\exp(\Delta W_mY_{\widehat H_{m}}),\quad \gamma_n=-1
\end{cases}
\]
is obtained using a Lie--Trotter splitting integrator $\Phi_{h,\gamma_n}^{sto}$ with time step size $h$ applied to the stochastic part of the equation, where the order of the integration depends on $\gamma_n$.

It is straightforward to check that the numerical scheme~\eqref{eq:order2} is a stochastic Poisson integrator, using the same arguments as in the proof of Proposition~\ref{propo:sPi}. Numerical experiments which illustrate the behaviour of this scheme and weak convergence with order $2$ will be reported below in Section~\ref{sect-numexp}. However, we do not give details concerning the theoretical analysis of the scheme~\eqref{eq:order2}. 

We also refer to~\cite{MR3927434,MR2409419} for other possible constructions of higher order splitting methods for SDEs. Finally, another possible strategy to design higher order integrators would be to use modified equations, like in~\cite{MR2970274}.
\end{remark}

\subsection{Convergence analysis}\label{sec:cv}
The objective of this section is to prove a general strong and weak convergence result for stochastic Poisson integrators~\eqref{slpI} defined by the splitting strategy. Note that we assume that the stochastic Poisson system~\eqref{prob} 
admits a Casimir function with compact level sets: as explained above, this condition ensures global well-posedness for the continuous problem, and provides almost sure bounds for the exact and numerical solutions (Propositions~\ref{propo:global} and~\ref{propo:numerik}). As a consequence, the general convergence result can be applied to get strong and weak convergence rates for the proposed explicit stochastic Poisson integrator~\eqref{slpI}, when applied to the stochastic rigid body system (Example~\ref{expl-SRB}) and to the stochastic sine--Euler system (Example~\ref{expl-SE}), see Theorems~\ref{thm-srb} and~\ref{thm-se} below respectively. Note that these two SDEs do not have globally Lipschitz continuous coefficients, so for those examples standard explicit schemes such as the Euler--Maruyama  method may fail to converge strongly. The fact that the proposed scheme is a stochastic Poisson integrator is essential to perform the convergence analysis.
However, the general convergence result below cannot be applied to the stochastic Maxwell--Bloch system -- 
the generalisation of the result to that example is not treated in the present work.


\begin{theorem}\label{thm-general}
Assume that the stochastic Poisson system~\eqref{prob} admits a Casimir function with compact level sets.


\begin{itemize}
\item[Strong convergence] Assume that $B$ is of class $\mathcal{C}^2$, that the mappings $H_1,\ldots,H_p$ are of class $\mathcal{C}^2$, and that the mappings $\widehat{H}_1,\ldots,\widehat{H}_m$ are of class $\mathcal{C}^3$. 
Then the stochastic Poisson integrator~\eqref{slpI} has strong order of convergence equal to $1/2$: for all $T\in(0,\infty)$ and all $y_0\in\R^d$, there exists a real number $\CC(T,y_0)\in(0,\infty)$ such that
\[
\underset{0\le n\le N}\sup~\left(\mathbb E\left[ \norm{ y\left(nh\right)-y^{[n]} }^2 \right] \right)^{1/2}\le \CC(T,y_0)h^{\frac12},
\]
with time step size $h=T/N$, and $y^{[0]}=y_0=y(0)$. 

If $m=1$, then the strong order of convergence is equal to $1$.


\item[Weak convergence] Assume that $B$ is of class $\mathcal{C}^5$, that the mappings $H_1,\ldots,H_p$ are of class $\mathcal{C}^5$, and that the mappings $\widehat{H}_1,\ldots,\widehat{H}_m$ are of class $\mathcal{C}^6$. Then the stochastic Poisson integrator~\eqref{slpI} has weak order of convergence equal to $1$: for all $T\in(0,\infty)$ and all $y_0\in\R^d$, and any test function $\phi\colon\R^d\to\R$ of class $\mathcal{C}^4$ with bounded derivatives, there exists a real number $\CC(T,y_0,\phi)\in(0,\infty)$ such that
\[
\underset{0\le n\le N}\sup~\left|\mathbb E\left[\phi\left(y\left(nh\right)\right)\right]-\mathbb E\left[\phi\left(y^{[n]}\right)\right]\right|\leq \CC(T,y_0,\phi)h.
\]
\end{itemize}
\end{theorem}

The convergence theorem stated above concerning the strong and weak rates of convergence of the stochastic Poisson integrator~\eqref{slpI} applied to the stochastic Poisson system~\eqref{prob} is an immediate consequence of the following auxiliary result, which is stated for a general SDE of the type
\begin{equation}\label{auxSDE}
\diff z(t)=\sum_{k=1}^{p}f_k(z(t))\,\diff t+\sum_{k=1}^m\widehat f_k(z(t))\circ\,\diff W_k(t),
\end{equation}
with functions $f_k$ and $\widehat{f}_k$ which are globally Lipschitz continuous.

\begin{lemma}\label{lemm-aux}
Consider the auxiliary splitting scheme
\begin{equation}\label{auxscheme}
z^{[n]}=\varphi_p(h,\cdot)\circ\ldots\circ \varphi_1(h,\cdot)\circ\widehat\varphi_m(\Delta W_m^n,\cdot)\circ \ldots\circ \widehat\varphi_1(\Delta W_1^n,\cdot)(z^{[n-1]}),
\end{equation}
with $z^{[0]}=z_0=z(0)$, associated with the auxiliary SDE~\eqref{auxSDE}, where $\varphi_k$ is the flow associated with the ODE $\dot{z}_k=f_k(z_k)$, $k=1,\ldots,p$, and $\widehat \varphi_k$ is the flow associated with the ODE $\dot{z}_k=\widehat f_k(z_k)$.


\begin{itemize}
\item[Strong convergence] Assume that the mappings $f_1,\ldots,f_p$ are of class $\mathcal{C}^1$ with bounded derivatives, and that the mappings $\widehat{f}_1,\ldots,\widehat{f}_m$ are bounded and of class $\mathcal{C}^2$ with bounded first and second order derivatives. Then the auxiliary scheme~\eqref{auxscheme} has strong order of convergence equal to $1/2$: for all $T\in(0,\infty)$ and all $z_0\in\R^d$, there exists a real number $\CC(T,z_0)\in(0,\infty)$ such that
\begin{equation}\label{eq:strongaux}
\underset{0\le n\le N}\sup~\left(\mathbb E\left[ \norm{ z\left(nh\right)-z^{[n]} }^2 \right] \right)^{1/2}\le \CC(T,z_0)h^{\frac12}.
\end{equation}

In the commutative noise case, {\it i.e.} if $\widehat{f}_k'(z)\widehat{f}_\ell(z)=\widehat{f}_\ell'(z)\widehat{f}_k(z)$ for all $k,\ell=1,\ldots,m$, the strong order of convergence is equal to $1$.

\item[Weak convergence] Assume that the mappings $f_1,\ldots,f_p$ are of class $\mathcal{C}^4$ with bounded derivatives, and that the mappings $\widehat{f}_1,\ldots,\widehat{f}_m$ are bounded and of class $\mathcal{C}^5$ with bounded first and second order derivatives. Then the auxiliary scheme~\eqref{auxscheme} has weak order of convergence equal to $1$: for all $T\in(0,\infty)$ and all $z_0\in\R^d$, and any test function $\phi:\R^d\to\R$ of class $\mathcal{C}^4$, there exists a real number $\CC(T,z_0,\phi)\in(0,\infty)$ such that
\begin{equation}\label{eq:weakaux}
\underset{0\le n\le N}\sup~\left|\mathbb E\left[\phi\left(z\left(nh\right)\right)\right]-\mathbb E\left[\phi\left(z^{[n]}\right)\right]\right|\leq \CC(T,z_0,\phi)h.
\end{equation}
\end{itemize}
\end{lemma}
The proof of Lemma~\ref{lemm-aux} is postponed to Appendix~\ref{app-auxlem}. Let us now check how Theorem~\ref{thm-general} is a straightforward corollary of Lemma~\ref{lemm-aux}. Note that if $m=1$, the commutative noise case condition is satisfied.

\begin{proof}[Proof of Theorem~\ref{thm-general}]
Owing to Propositions~\ref{propo:global} and~\ref{propo:numerik}, the exact and numerical solutions of the SDE~\eqref{prob}, resp. scheme~\eqref{slpI}, 
satisfy the almost sure bounds
\[
\underset{t\in[0,T]}\sup~\norm{y(t)}\le R(y_0),\quad \underset{N\ge 1}\sup~\underset{0\le n\le N}\sup~\norm{y^{[n]}}\le R(y_0),
\]
where $R(y_0)=\max_{y\in\mathbb R^d, C(y)=C(y_0)}\|y\|$, and $R(y_0)<\infty$ since $C$ has compact level sets by assumption.

Using the same construction as in the proof of Proposition~\ref{propo:global}, one can define compactly supported functions $f_k$ and $\widehat{f}_k$, such that $f_k(y)=B(y)\nabla H_k(y)$ and $\widehat{f}_k(y)=B(y)\nabla \widehat H_k(y)$ for all $y\in \R^d$ such that $\|y\|\le R(y_0)$. 
In addition, $f_k$ is at least of class $\mathcal{C}^1$ and $\widehat{f}_k$ is at least of class $\mathcal{C}^2$.

Note that with this choice, $y(t)=z(t)$ and $y^{[n]}=z^{[n]}$ for all $t\in[0,T]$ and all $n\in\{0,\ldots,N\}$, where $\bigl(z(t)\bigr)_{t\ge 0}$ is the solution of the auxiliary SDE~\eqref{auxSDE} and $\bigl(z^{[n]}\bigr)_{n\ge 0}$ is obtained by the auxiliary scheme~\eqref{auxscheme}. It remains to apply Lemma~\ref{lemm-aux} to conclude. Note also that it is not necessary to assume that the functions $\phi$, $B$, $H_1,\ldots,H_p$, $\widehat H_1,\ldots,\widehat H_m$ and their derivatives are bounded. This is due to the boundedness of the exact and numerical solutions provided by the preservation of the Casimir function $C$ and the compact level sets assumption.
%
%
\end{proof}

\begin{remark}\label{rem-consistent}
If one considers the following variant of the stochastic Poisson system~\eqref{prob}
\[
\diff y(t)=B(y(t))\nabla H(y(t))\,\diff t+\sum_{k=1}^mB(y(t))\nabla \widehat H_k(y(t))\circ\,\diff W(t)
\]
driven by a single Wiener process $W$ (that is $W_1=\ldots=W_m=W$), the associated variant of the proposed stochastic Poisson 
integrator~\eqref{slpI} reads
\begin{align*}
\Phi_h&=\exp(hY_{H_p})\circ\exp(hY_{H_{p-1}})\circ\ldots\circ\exp(hY_{H_1}) \nonumber\\
&\circ\exp(\Delta WY_{\widehat H_m})\circ\exp(\Delta WY_{\widehat H_{m-1}})
\circ\ldots\circ\exp(\Delta WY_{\widehat H_1}).
\end{align*}
This scheme is not consistent when $m\ge 2$. In the proof of the convergence result Theorem~\ref{thm-general}, more precisely in the proof of Lemma~\ref{lemm-aux}, the independence of the Wiener processes $W_1,\ldots,W_m$ plays a crucial role.
\end{remark}

\subsection{Asymptotic preserving schemes in the diffusion approximation regime}\label{APsplit}

The objective of this section is to use the proposed splitting stochastic Poisson integrators~\eqref{slpI} in order to define effective numerical schemes for the discretisation of the multiscale system~\eqref{prob_eps} described in Section~\ref{sec:multi}. The challenge is to obtain a good behaviour of the numerical scheme when $\epsilon\to 0$. On the one hand, one needs to avoid time step size restrictions of the type $h={\rm O}(\epsilon)$ or $h={\rm O}(\epsilon^2)$, which would be prohibitive when $\epsilon$ is small. On the other hand, it would be desirable to have a convergence (in distribution) of the type $y^{\epsilon,[n]}\underset{\epsilon\to 0}\to y^{[n]}$, for all fixed $h>0$ and $n\ge 1$, to reproduce the diffusion approximation result Proposition~\ref{propo:multi} at the discrete time level. Indeed, if the two requirements above are satisfied, the integrator can be used to approximate both~\eqref{prob} and~\eqref{prob_eps}, without the need to adapt the time step size $h$ when $\epsilon$ vanishes.

The class of numerical methods which satisfy the two requirements above is known as \emph{asymptotic preserving} schemes. We refer to the recent work~\cite{BRR} where asymptotic preserving schemes were introduced for a class of stochastic differential equations of the type~\eqref{prob_eps}. Note that a standard Euler--Maruyama scheme does not satisfy the asymptotic preserving property. Recall that for this notion of asymptotic preserving schemes, the convergence is understood in the sense of convergence in distribution of random variables. Using the splitting strategy allows us to design other examples of asymptotic preserving schemes for~\eqref{prob_eps}, such that the corresponding limit scheme (obtained when $\epsilon\to 0$ with fixed time step size $h>0$) is the splitting stochastic Poisson integrator~\eqref{slpI}.

We propose the following integrator for the multiscale system~\eqref{prob_eps}: for any $\epsilon\in(0,1)$ and any time step size $h>0$, for all $n\ge 1$, set
\begin{align}\label{APscheme}
y^{\epsilon,[n]}&=\exp(hY_{H_p})\circ\ldots\circ\exp(hY_{H_1}) \nonumber\\
&\circ\exp\left(\frac{h\xi_{m}^{\epsilon,[n]}}{\epsilon}Y_{\widehat H_m}\right)\circ\ldots
\circ\exp\left(\frac{h\xi_{1}^{\epsilon,[n]}}{\epsilon}Y_{\widehat H_1}\right)(y^{\epsilon,[n-1]}),
\end{align}
where, for each $k=1,\ldots,m$, the Ornstein--Uhlenbeck process $\xi_k^\epsilon$ is discretised using the linear implicit Euler scheme
\[
\xi_k^{\epsilon,[n]}=\xi_{k}^{\epsilon,[n-1]}-\frac{h}{\epsilon^2}\xi_{k}^{\epsilon,[n]}+\frac{\Delta_n W_k}{\epsilon}=\frac{1}{1+\frac{h}{\epsilon^2}}\Bigl(\xi_k^{\epsilon,[n-1]}+\frac{\Delta_n W_k}{\epsilon}\Bigr).
\]
Note that $C(y^{\epsilon,[n]})=C(y^{\epsilon,[0]})$ for all $n\ge 0$, if $C$ is a Casimir function of the stochastic Poisson system~\eqref{prob}. If $C$ has compact level sets, this yields the following variant of the bound of Proposition~\ref{propo:numerik},
\[
\underset{\epsilon\in(0,1)}\sup~\underset{h>0}\sup~\underset{n\ge 0}\sup~\|y^{[n]}\|\le  R(y^{[0]})=\max_{y\in\mathbb R^d, C(y)=C(y^{[0]})}\|y\|,
\]
which is uniform over $\epsilon$.

Observe that for all $n\ge 1$ and $h>0$, one has
\[
\frac{h\xi_{k}^{\epsilon,[n]}}{\epsilon}=\Delta_nW_k+\epsilon\bigl(\xi_k^{\epsilon,[n-1]}-\xi_k^{\epsilon,[n]}\bigr)\underset{\epsilon\to 0}\to \Delta_nW_k.
\]

By a recursion argument, it is then straightforward to check that
\[
y^{\epsilon,[n]}\underset{\epsilon\to 0}\to y^{[n]},
\]
for all $n\ge 0$ and for all fixed $h>0$, where $y^{[n]}$ is given by the splitting scheme~\eqref{slpI}. 
As a consequence, the scheme~\eqref{APscheme} is an asymptotic preserving scheme, in the sense of~\cite{BRR}: the following diagram commutes
\[
\begin{CD}
y^{\epsilon,[N]}     @>{N\to\infty}>> y^\epsilon(T) \\
@VV{\epsilon\to 0}V        @VV{\epsilon\to 0}V\\
y^{[N]}     @>{\Delta t\to 0}>> y(T)
\end{CD}
\]
when the time step size is given by $h=T/N$. In other words:
\begin{itemize}
\item for each fixed $\epsilon>0$, the scheme~\eqref{APscheme} is consistent with~\eqref{prob_eps} when $h\to 0$,
\item for each fixed $h>0$, the proposed scheme~\eqref{APscheme} converges to a limiting scheme when $\epsilon\to 0$, which is given here by 
the abstract splitting scheme~~\eqref{slpI},
\item the limiting scheme~\eqref{slpI} is consistent when $h\to 0$ with~\eqref{prob}, which is the limit when $\epsilon\to 0$ of~\eqref{prob_eps}.
\end{itemize}
We refer to the recent work~\cite{BRR} for a general analysis of asymptotic preserving schemes for stochastic differential equations. As explained in~\cite{BRR}, the construction of asymptotic preserving schemes for SDEs, in particular to obtain equations interpreted in the Stratonovich sense, may be subtle. Here we do not employ a predictor-corrector strategy as in~\cite{BRR} (which is used to get the Stratonovich interpretation instead of the It\^o one), 
since we directly use exact flows of the appropriate subsystems in the splitting procedure: in the present paper, the Stratonovich interpretation is obtained in a natural way.

The property of being asymptotic preserving is a qualitative property of a numerical scheme. Let us now briefly discuss the behaviour of weak error estimates of the asymptotic preserving scheme~\eqref{APscheme} when $\epsilon$ is small. For each fixed $\epsilon>0$, it is expected that the proposed asymptotic preserving scheme~\eqref{APscheme} has a weak order of convergence equal to $1$ in general: for test functions $\phi\colon\R^d\to\R$ of class $\mathcal{C}^4$, one has
\[
\left|\mathbb E[\varphi(y^{\epsilon,[N]})]-\mathbb E[\varphi(y^\epsilon(T))]\right|\le \CC_\epsilon(T,\varphi)h,
\]
where $h=\frac{T}{N}$ and the real number $\CC_\epsilon(T,\varphi)$ may depend on $\epsilon$ and diverge when $\epsilon\to 0$. 
In order to have a computational cost independent of the parameter $\epsilon$, it would be desirable to establish that the proposed scheme is \emph{uniformly accurate}: one would need to prove error estimates of the type
\[
\underset{\epsilon\in(0,\epsilon_0)}\sup~\big|\mathbb E[\varphi(y^{\epsilon,[N]})]-\mathbb E[\varphi(y^\epsilon(T))]\big|\le \CC(T,\varphi)h^\alpha,
\]
which are uniform with respect to $\epsilon\in(0,\epsilon_0)$ (with arbitrary $\epsilon_0>0$), in other words $\CC(T,\varphi)$ is independent of $\epsilon$. 
Observe that a reduction of the order of convergence, namely $\alpha<1$, may happen. 
Proving the uniform accuracy property of the scheme~\eqref{APscheme} is beyond the scope of this work. 
However, in the numerical experiments reported below, we investigate whether such uniform weak error estimates hold for the considered problems.

\begin{remark}
It is possible to define a variant of the asymptotic preserving scheme~\eqref{APscheme}, using a midpoint approximation for the the Ornstein--Ulenbeck components:
\[
\xi_k^{\epsilon,[n]}=\xi_{k}^{\epsilon,[n-1]}-\frac{h}{2\epsilon^2}\left(\xi_k^{\epsilon,[n-1]}+\xi_{k}^{\epsilon,[n]}\right)+\frac{\Delta_n W_k}{\epsilon},
\]
in which case the definition of $y^{\epsilon,[n]}$ needs to be modified as follows:
\begin{align*}
y^{\epsilon,[n]}&=\exp(hY_{H_p})\circ\exp(hY_{H_{p-1}})\circ\ldots\circ\exp(hY_{H_1}) \nonumber\\
&\circ\exp\left(\frac{h(\xi_k^{\epsilon,[n-1]}+\xi_{k}^{\epsilon,[n]})}{2\epsilon}Y_{\widehat H_m}\right)\circ\ldots\circ\exp\left(\frac{h(\xi_k^{\epsilon,[n-1]}+\xi_{k}^{\epsilon,[n]})}{2\epsilon}Y_{\widehat H_1})(y^{\epsilon,[n-1]}\right).
\end{align*}
That scheme is also asymptotic preserving.
\end{remark}

\section{Numerical experiments}\label{sect-numexp}

In this section, we illustrate the behaviour of the stochastic Poisson integrators which have been proposed and analysed in Section~\ref{sec:integrators}. 
We choose to present numerical experiments for the three examples of stochastic Lie--Poisson systems introduced in Section~\ref{sec:poisson}.  
On the one hand, we illustrate the qualitative properties of the proposed splitting stochastic Poisson integrators, compared with standard methods, 
by considering the temporal evolution of Casimir functions. On the other hand, we investigate and state strong and weak orders of convergence (which are consequences of Theorem~\ref{thm-general}), 
and we illustrate the quantitative error estimates obtained above. In addition, we illustrate the asymptotic preserving property for the multiscale versions of the considered systems. Note that, in general, the theoretical convergence results cannot be 
applied to the stochastic Maxwell--Bloch system (Example~\ref{expl-MB}), 
since no Casimir functions with compact level sets is known for that example. 
However, the theoretical results can be applied to the stochastic rigid body system (Example~\ref{expl-SRB}) and to the stochastic sine--Euler system (Example~\ref{expl-SE}). 

\subsection{Explicit stochastic Poisson integrators for stochastic Maxwell--Bloch equations}\label{ssec:PMB} 

This subsection presents explicit stochastic Poisson integrators for the stochastic Maxwell--Bloch system~\eqref{smb} (Example~\ref{expl-MB}). We first give a detailed construction of the splitting scheme, which gives a stochastic Poisson integrator. We then illustrate its qualitative properties (preservation of the Casimir function) and strong and weak error estimates of the proposed scheme by numerical experiments. Finally, we illustrate the asymptotic preserving property (Section~\ref{APsplit}) for a multiscale version of the system.

\subsubsection{Presentation of the splitting scheme for the stochastic Maxwell--Bloch system}

Recall that the stochastic Maxwell--Bloch system~\eqref{smb} introduced in Example~\ref{expl-MB} is of the type
\begin{equation*}
\diff y=B(y)\left(\nabla H(y)\,\diff t+\sigma_1\nabla \widehat H_1(y)\circ\,\diff W_1(t)+\sigma_3\nabla \widehat H_3(y)\circ\,\diff W_3(t) \right).
\end{equation*}

To apply the strategy described in Section~\ref{ssec:LP} and construct explicit stochastic Poisson integrators, we follow the approach from~\cite{MR1702129} for the deterministic Maxwell--Bloch system. The Hamiltonian function $H$ is split as $H=H_1+H_3$, with $H_1(y)=\widehat H_1(y)=\frac12y_1^2$ and $H_3(y)=\widehat H_3(y)=y_3$. The two associated  deterministic subsystems can be solved exactly as follows. 
On the one hand, the deterministic subsystem corresponding with the vector field $Y_{H_1}=B\nabla H_1$ is given by
\begin{equation*}
\left\lbrace
\begin{aligned}
\dot y_1&=0\\
\dot y_2&=y_3y_1\\
\dot y_3&=-y_2y_1.
\end{aligned}
\right.
\end{equation*}
Observe that $y_1$ may be considered as a constant and thus $(y_2,y_3)$ is solution of a linear ordinary differential equation (it is the standard harmonic oscillator): the exact solution of the first subsystem is thus given by
\begin{equation*}
\exp(tY_{H_1})y(0)
=
\begin{pmatrix}1 & 0 & 0\\ 0 & \cos(y_1(0)t) & \sin(y_1(0)t)\\ 0 & -\sin(y_1(0)t) & \cos(y_1(0)t)\end{pmatrix}y(0)
\end{equation*}
for all $t\in\mathbb{R}$ and $y(0)\in\mathbb{R}^3$.

On the other hand, the deterministic subsystem corresponding with the vector field $Y_{H_3}=B\nabla H_3$ is given by
\begin{equation*}
\left\lbrace
\begin{aligned}
\dot y_1&=y_2\\
\dot y_2&=0\\
\dot y_3&=0.
\end{aligned}
\right.
\end{equation*}
The exact solution of the second subsystem is thus given by
\begin{equation*}
\exp(tY_{H_3})y(0)
=
\begin{pmatrix}1 & t & 0\\ 0 & 1 & 0\\ 0 & 0 & 1\end{pmatrix}y(0)
\end{equation*}
for all $t\in\mathbb{R}$ and $y(0)\in\mathbb{R}^3$.

In the case of the stochastic Maxwell--Bloch system~\eqref{smb}, the splitting integrator~\eqref{slpI} then reads
\begin{equation}\label{smbI}
\Phi_h=\exp(hY_{H_3})\circ\exp(hY_{H_1})\circ\exp(\sigma_3\Delta W_3Y_{\widehat H_3})\circ\exp(\sigma_1\Delta W_1Y_{\widehat H_1}),
\end{equation} 
where for all $y\in\R^3$ one has
\[
\exp(\sigma_1\Delta W_1Y_{\widehat H_1})y
=
\begin{pmatrix}1 & 0 & 0\\ 0 & \cos(y_1\sigma_1\Delta W_1) & \sin(y_1\sigma_1\Delta W_1)\\ 0 & -\sin(y_1\sigma_1\Delta W_1) & \cos(y_1\sigma_1\Delta W_1)\end{pmatrix}y
\]
and
\[
\exp(\sigma_3\Delta W_3Y_{\widehat H_3})y=\begin{pmatrix}1 & \sigma_3\Delta W_3 & 0\\ 0 & 1 & 0\\ 0 & 0 & 1\end{pmatrix}y.
\]
Owing to Proposition~\ref{propo:sPi}, the explicit splitting scheme~\eqref{smbI} is a stochastic Poisson integrator.

\subsubsection{Preservation of the Casimir of the stochastic Maxwell--Bloch system}

Let us first illustrate the qualitative behaviour of the stochastic Poisson integrator~\eqref{smbI} introduced above. Figure~\ref{fig:CasMaxBloc} illustrates the preservation of the Casimir $C(y)=\frac{1}{2}(y_2^2+y_3^2)$ by the stochastic Poisson integrator~\eqref{smbI}. In this numerical experiment, the initial value is $y(0)=(1,2,3)$ and the final time is $T=1$. We consider the two cases where the system~\eqref{smb} is driven by a single Wiener process: $(\sigma_1,\sigma_3)=(1,0)$ and $(\sigma_1,\sigma_3)=(0,1)$. Similar results would be obtained if the system was driven by two independent Wiener processes ($\sigma_1=\sigma_3=1$ for instance). In Figure~\ref{fig:CasMaxBloc}, we compare the numerical solutions given by the classical Euler--Maruyama scheme (applied to the It\^o formulation of the system), the stochastic midpoint scheme from~\cite{Milstein2002a}, and the explicit splitting scheme~\eqref{smbI}. The time step size is equal to $h=0.01$. To implement the implicit stochastic midpoint scheme, a truncation of the noise with threshold $A=\sqrt{4|\log(h)|}$ is applied (see~\cite{Milstein2002a} for details). To be able to compare the results for different schemes, we use this truncation in all experiments where the implicit stochastic midpoint scheme is involved. As shown in Proposition~\ref{propo:sPi}, we observe that the Casimir function $C(y)=\frac12(y_2^2+y_3^2)$ is preserved when using the stochastic Poisson integrator~\eqref{smbI}. The Casimir function is also preserved when using the stochastic midpoint scheme: indeed, this integrator is known to preserve quadratic invariants, see~\cite{MR3210739}.

\begin{figure}[h]
\centering
\includegraphics*[width=0.48\textwidth,keepaspectratio]{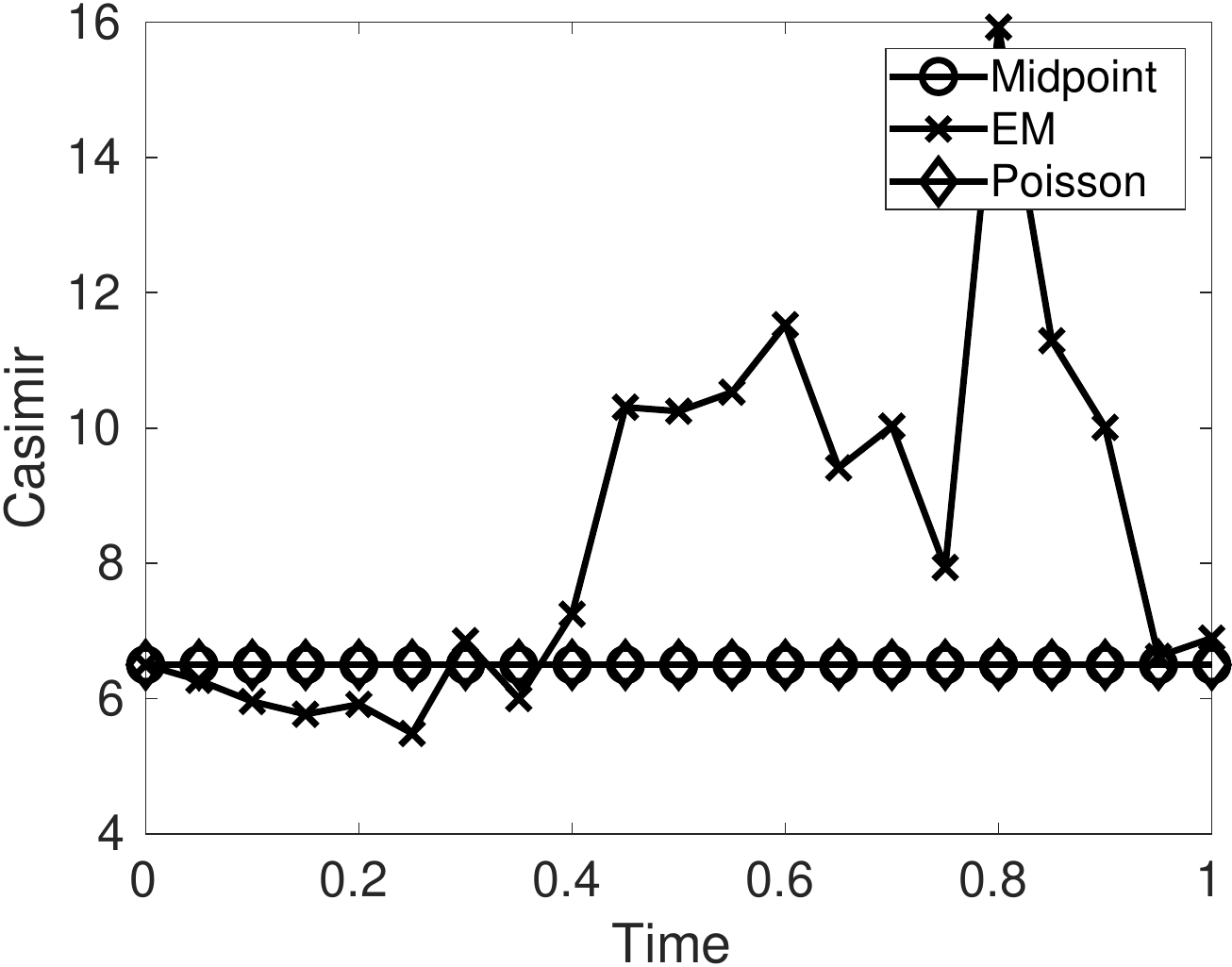}
\includegraphics*[width=0.48\textwidth,keepaspectratio]{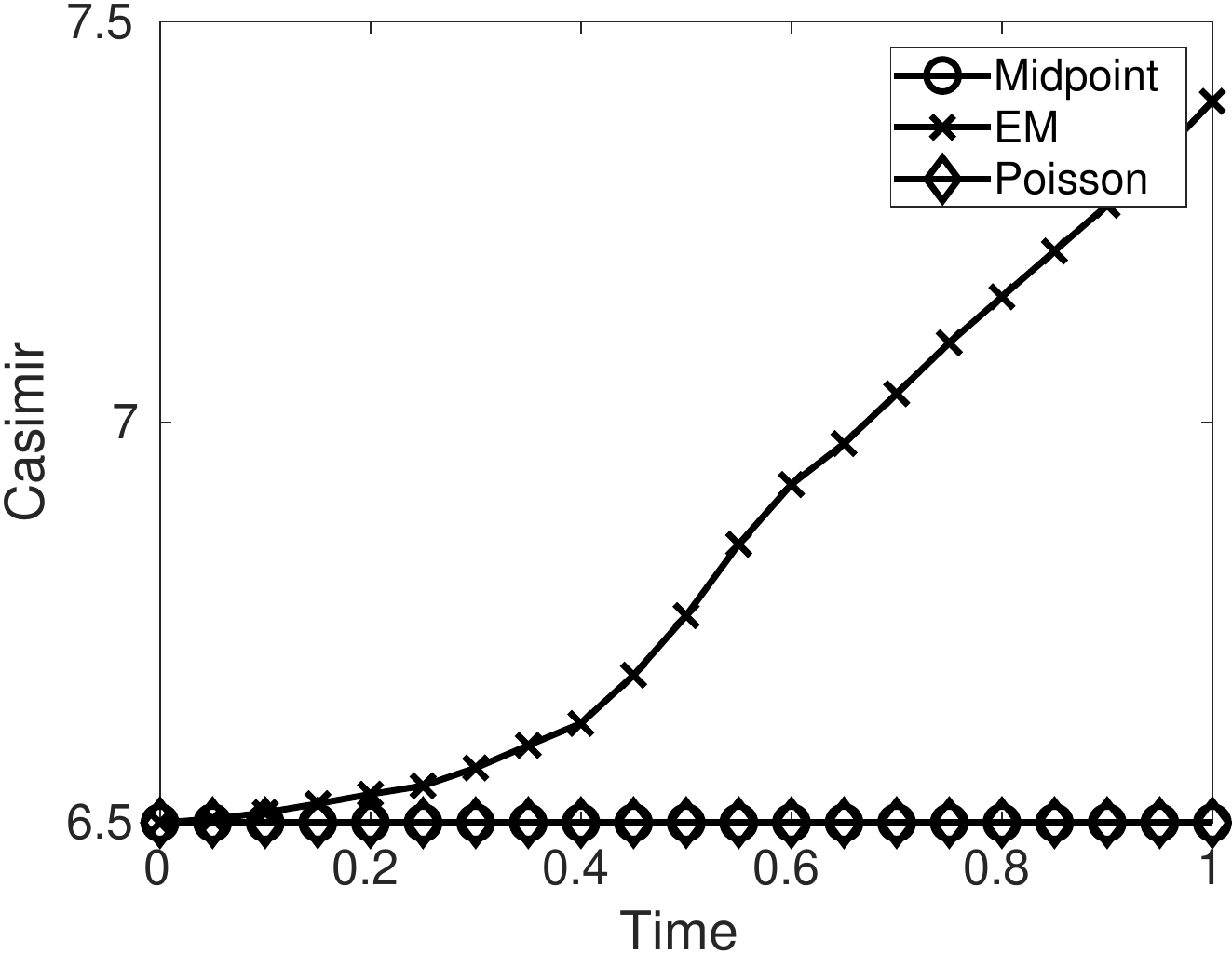}
\caption{Stochastic Maxwell--Bloch system: preservation of the Casimir by
the Euler--Maruyama scheme ($\times$), the midpoint scheme ($\circ$), 
and the explicit stochastic Poisson integrator ($\diamond$). Left: $(\sigma_1,\sigma_3)=(1,0)$. Right: $(\sigma_1,\sigma_3)=(0,1)$.}
\label{fig:CasMaxBloc}
\end{figure}

\subsubsection{Strong and weak convergence of the explicit stochastic Poisson integrator for the stochastic Maxwell--Bloch system}

The preservation of the Casimir $C$ is not sufficient to ensure almost sure boundedness of the numerical solution, which is instrumental to deduce Theorem~\ref{thm-general} from Lemma~\ref{lemm-aux}. As a consequence, we are not able to state a convergence result for the stochastic Poisson integrator~\eqref{smbI} in general. However, when $\sigma_3=0$, it is possible to show the following result.
\begin{proposition}\label{thm-smb}
Consider a numerical discretisation of the stochastic Maxwell--Bloch system~\eqref{smb} by the stochastic Poisson integrator~\eqref{smbI}. Assume that $\sigma_3=0$. Then, the strong order of convergence and the weak order of convergence of this scheme are equal to $1$.
\end{proposition}

\begin{proof}
Let us prove that, when $\sigma_3=0$, the following  bounds are satisfied almost surely:
\begin{equation}\label{boundsMB}
\left\lbrace
\begin{aligned}
\|y^{[n]}\|&\le (1+h)^n\|y^{[0]}\|\\
\|y(t)\|&\le e^{t}\|y(0)\|
\end{aligned}
\right.
\end{equation}
for all $n\ge 0$, $h\in(0,h_0)$ and $t\ge 0$.

The proof of the bounds above is straightforward. On the one hand, for all $y\in\R^3$, all $h>0$ and $t\in\R$, one has
\[
\|e^{tY_{H_1}}y\|=\|e^{tY_{\widehat H_1}}y\|=\|y\|
\]
and
\begin{align*}
\|e^{hY_{H_3}}y\|^2&=(y_1+hy_2)^2+y_2^2+y_3^2=\|y\|^2+2hy_1y_2+h^2y_2^2\\
&\le (1+h)^2\|y\|^2.
\end{align*}
Therefore
\[
\|y^{[n]}\|\le \|e^{hY_{H_3}}\circ e^{hY_{H_1}}\circ e^{\sigma_1\Delta_n W_1 Y_{\widehat H_1}}(y^{[n-1]})
\|
\le (1+h)\|y^{[n-1]}\|
\]
thus $\|y^{[n]}\|\le (1+h)^n\|y^{[0]}\|$.

On the other hand, let $\mathcal{H}(y)=\|y\|^2$. Then one has $\{\mathcal{H},H_1\}=\{\mathcal{H},\widehat H_1\}=0$ and $\{\mathcal{H},H_3\}(y)=2y_1y_2\le \mathcal{H}(y)$ for all $y\in\R^3$ (recall that the Poisson bracket is defined by~\eqref{Poisson.bracket}). Using~\eqref{dH.Poisson.bracket}, one thus obtains $\diff \mathcal{H}(y(t))\le \mathcal{H}(y(t))$ for all $t\ge 0$ and the bound for the exact solution follows from Gronwall's lemma.

Let $T\in(0,\infty)$, one can then repeat the arguments used in the proof of Theorem~\ref{thm-general} as a corollary of Lemma~\ref{lemm-aux}, using the almost sure bounds
\[
\underset{t\in[0,T]}\sup~\norm{y(t)}\le R(y_0,T),\quad \underset{N\ge 1}\sup~\underset{0\le n\le N}\sup~\norm{y^{[n]}}\le R(y_0,T)
\]
with $R(y_0,T)=e^{T}\|y_0\|$. The details are omitted. Note that the strong order of convergence is equal to $1$ since the system is driven by a single Wiener process ($m=1$, the commutative noise case condition is satisfied). This concludes the proof of Proposition~\ref{thm-smb}.
\end{proof}

When $\sigma_3>0$, one can prove the following moment bound for the numerical solution:
\[
\E[\|y^{[n]}\|^2]\le e^{(1+\frac{\sigma_3^2}{2})nh}\E[\|y^{[0]}\|^2].
\]
This follows from the inequality
\begin{align*}
\E[\|e^{\sigma_3\Delta W_3Y_{H_3}}y\|^2]&=\E\bigl[(y_1+\sigma_3\Delta W_3y_2)^2+y_2^2+y_3^2\bigr]=\|y\|^2+\sigma_3^2h|y_2|^2\\
&\le (1+\sigma_3^2h)\|y\|^2
\end{align*}
and a recursion argument. A similar moment bound holds for the exact solution, however these moment bounds are not sufficient to prove a strong convergence result.

The objectives of this subsection are first to illustrate Proposition~\ref{thm-smb}, and second to investigate the behaviour of the strong and weak errors when the condition $\sigma_3=0$ is removed. Whether it is possible to prove strong and weak convergence estimates, or convergence in probability results, for this problem in the general case is left open for future works.

We first illustrate the convergence of the strong error. In this numerical experiment, the initial value is $y(0)=(1,2,3)$ and the final time is $T=1$. We consider the two cases where the system~\eqref{smb} is driven by a single Wiener process: $(\sigma_1,\sigma_3)=(1,0)$ and $(\sigma_1,\sigma_3)=(0,1)$. Similar results would be obtained if the system was driven by two independent Wiener processes ($\sigma_1=\sigma_3=1$ for instance). The reference solution is computed using each scheme with time step size $h_{\text{ref}}=2^{-16}$, and the schemes are applied with the range of time step sizes $h=2^{-5},\ldots,2^{-13}$. The expectation is approximated averaging the error over $M_s=500$ independent Monte Carlo samples.

Like in Figure~\ref{fig:CasMaxBloc}, we compare the splitting integrator~\eqref{smbI} with the standard Euler--Maruyama scheme, and the stochastic midpoint scheme from \cite{Milstein2002a}. A truncation of the noise is used, see above for details. To be able to compare the results for different schemes, we use this truncation in all experiments where the implicit stochastic midpoint scheme is involved. 

For the cases where the system is driven by a single Wiener process ($\sigma_3=0$ or $\sigma_1=0$), the results of the numerical experiment are presented in Figure~\ref{fig:msMaxBloc}: we observe a strong order of convergence equal to $1$ for the proposed explicit stochastic Poisson integrator~\eqref{smbI}. This confirms the result of Proposition~\ref{thm-smb} when $\sigma_3=0$. We also conjecture that the stochastic Poisson integrator~\eqref{smbI} has strong order of convergence equal to $1$ when $\sigma_1=0$.

\begin{figure}[h]
\centering
\includegraphics*[width=0.48\textwidth,keepaspectratio]{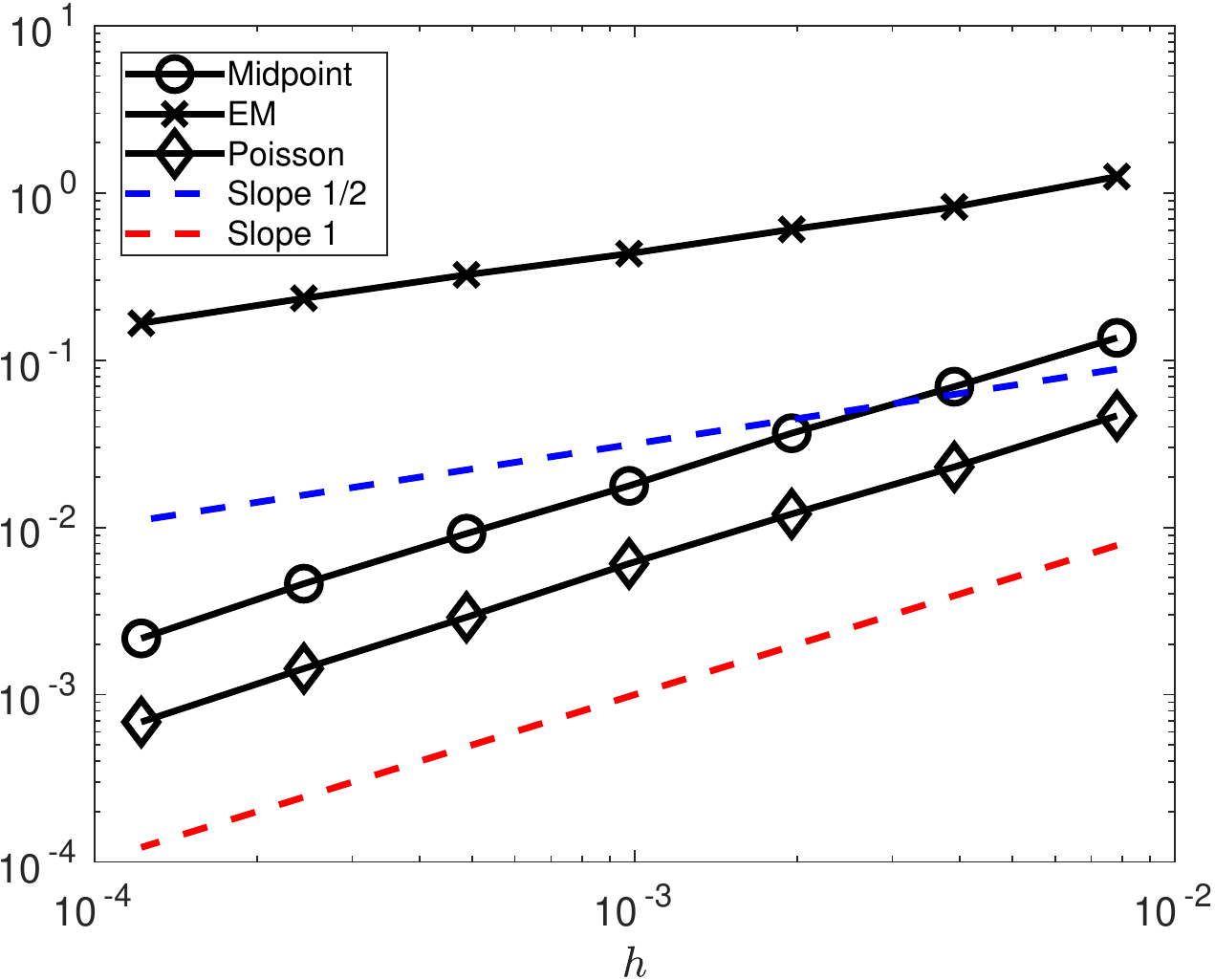}
\includegraphics*[width=0.48\textwidth,keepaspectratio]{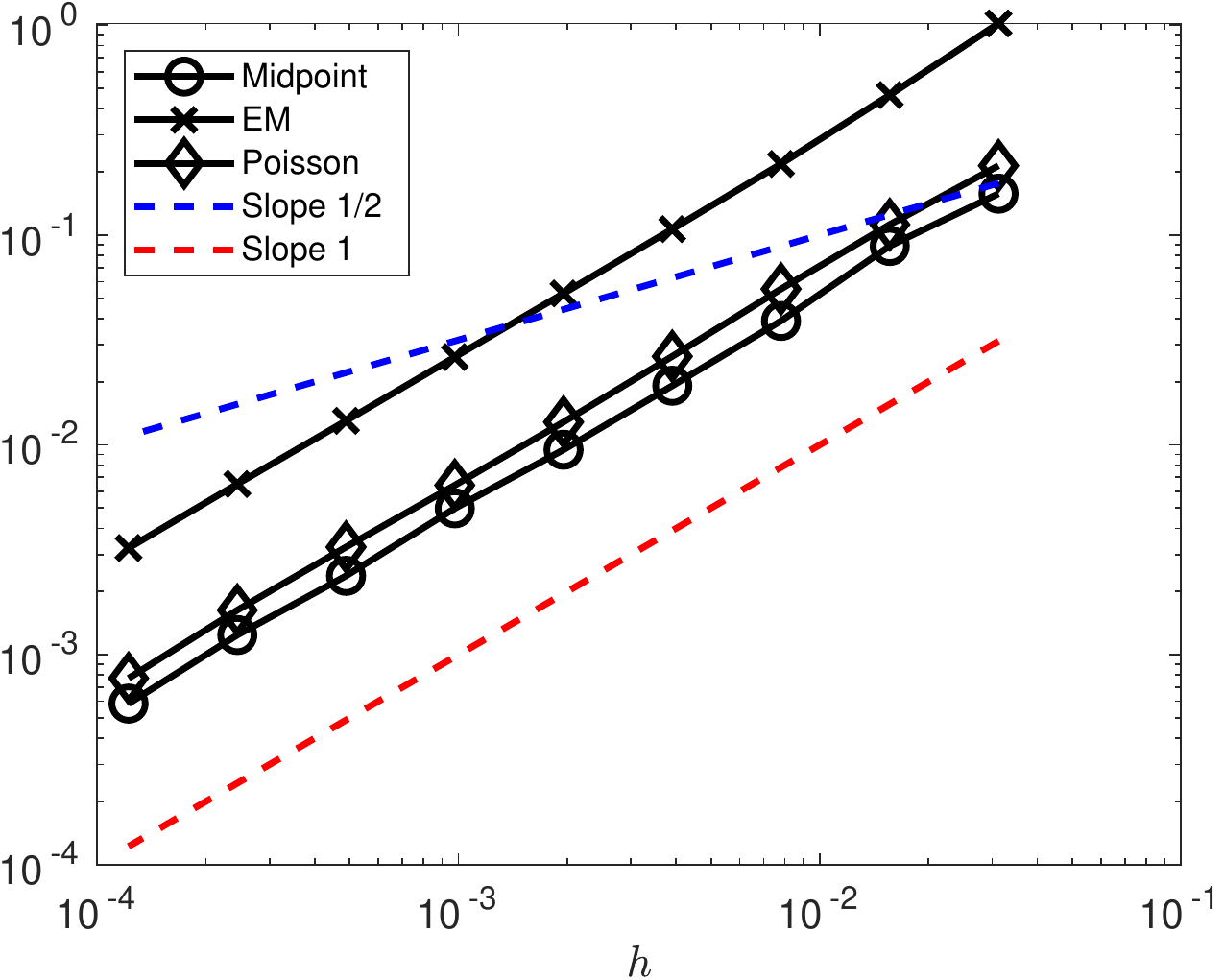}
\caption{Stochastic Maxwell--Bloch system with a single Wiener process: Strong errors of the Euler--Maruyama scheme ($\times$), the midpoint scheme ($\circ$), 
and the explicit stochastic Poisson integrator ($\diamond$). Left: $(\sigma_1,\sigma_3)=(1,0)$. Right: $(\sigma_1,\sigma_3)=(0,1)$.}
\label{fig:msMaxBloc}
\end{figure}


For the case where the system is driven by two Wiener processes, the results of the numerical experiment are presented in Figure~\ref{fig:msMaxBloc2}, with $\sigma_1=\sigma_3=1$.
Based on the observed convergence behaviour, we conjecture that the strong order of the proposed integrator is $1/2$.
This result is not covered by the theoretical analysis performed in this article.

\begin{figure}[h]
\centering
\includegraphics*[width=0.48\textwidth,keepaspectratio]{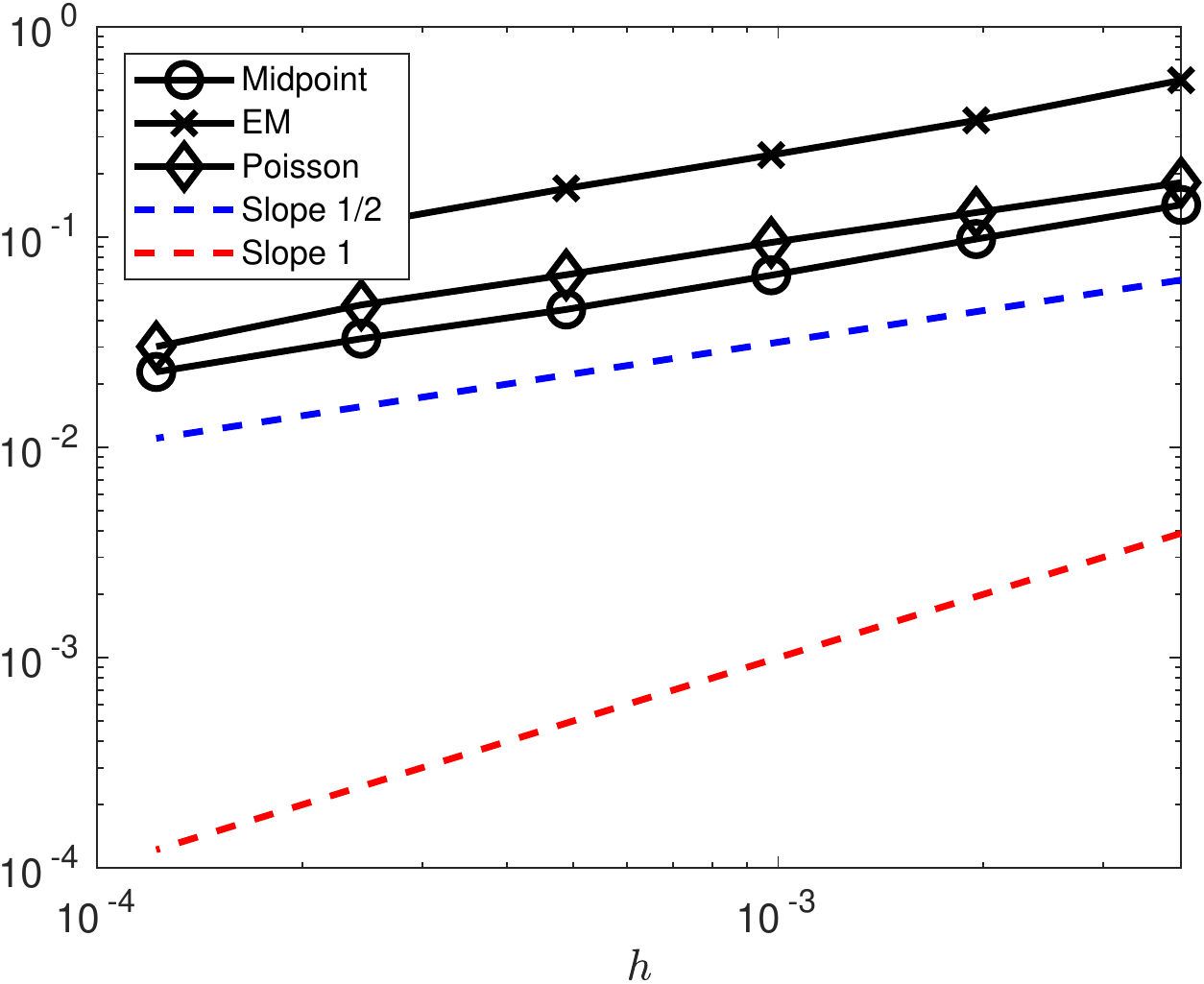}
\caption{Stochastic Maxwell--Bloch system with $(\sigma_1,\sigma_3)=(1,1)$: 
Strong errors of the Euler--Maruyama scheme ($\times$), the midpoint scheme ($\circ$), 
and the explicit stochastic Poisson integrator ($\diamond$).}
\label{fig:msMaxBloc2}
\end{figure}

We now illustrate the weak convergence of the stochastic Poisson integrator~\eqref{smbI}. For this numerical experiment, 
we set $\sigma_1=\sigma_3=1$, the initial value is $y(0)=(1,2,3)$ and the final time is $T=1$. The reference solution is computed using each scheme with time step size $h_{\text{ref}}=2^{-16}$, and the schemes are applied with the range of time step sizes $h=2^{-6},\ldots,2^{-12}$. The expectation is approximated averaging the error over $M_s=10^9$ independent Monte Carlo samples. Finally, the test function is given by $\phi(y)=\sin(2\pi y_1)+\sin(2\pi y_2)+\sin(2\pi y_3)$.


The results are presented in Figure~\ref{fig:weakMaxBloc}. According to the observed rate of convergence, we conjecture that the weak order of the proposed integrator is $1$, but this result is not covered by the theoretical analysis performed in this article.

\begin{figure}[h]
\centering
\includegraphics*[width=0.48\textwidth,keepaspectratio]{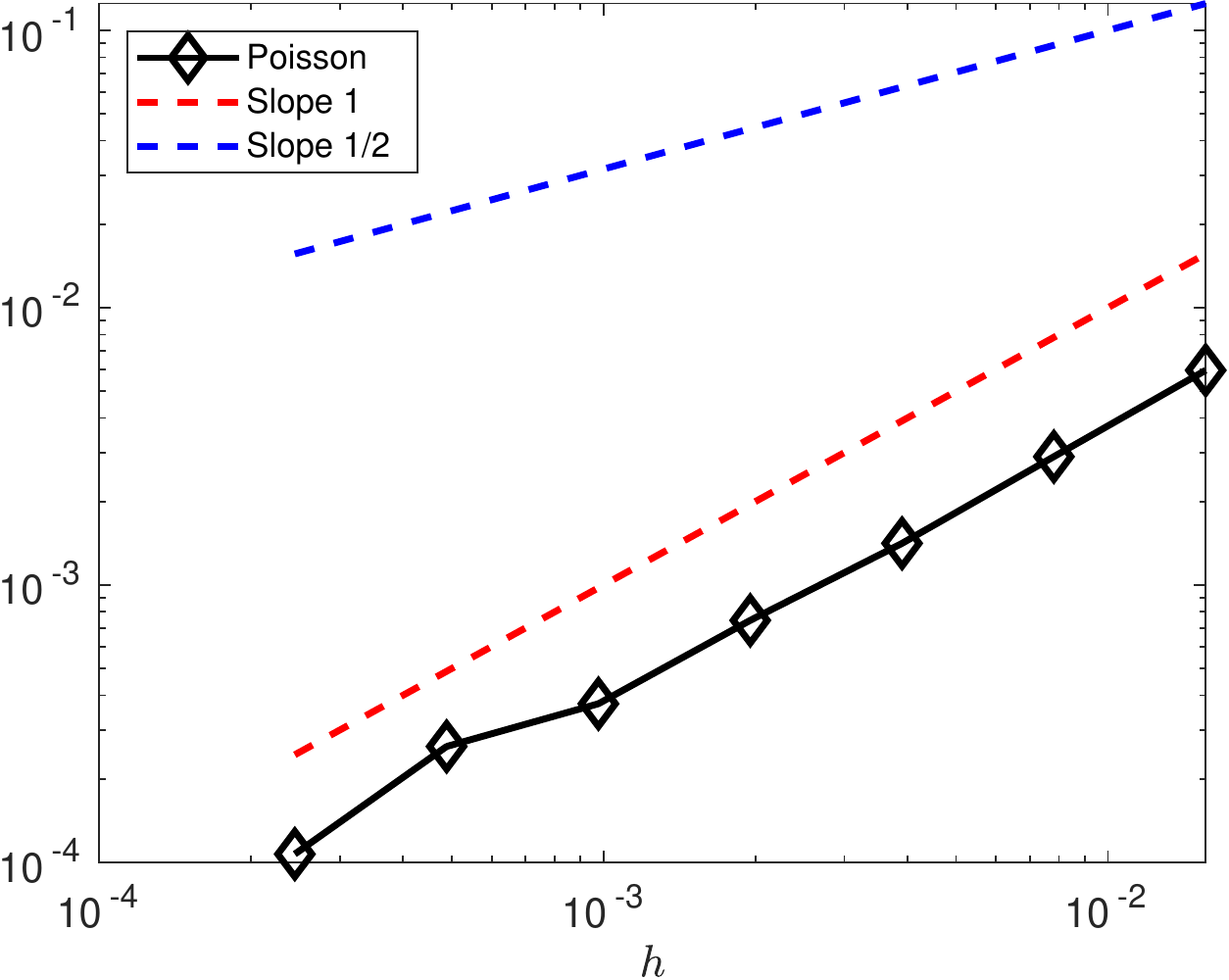}
\caption{Stochastic Maxwell--Bloch system:  Weak errors of the explicit stochastic Poisson integrator.} 
\label{fig:weakMaxBloc}
\end{figure}

Numerical experiments illustrating the behaviour of the weak error when the system is driven by a single noise are not reported: indeed, as seen in Figure~\ref{fig:msMaxBloc}, the stochastic Poisson integrator has strong order of convergence equal to $1$ if $\sigma_3=0$ (rigorous result, Proposition~\ref{thm-smb}) or if $\sigma_1=0$ (conjecture). In those cases, the weak error behaves like the strong error and the rate of convergence is $1$.

To conclude this subsection, let us provide a numerical experiment using the scheme~\eqref{eq:order2} of weak order $2$ presented in Remark~\ref{rem:order2}. 
For this numerical experiment, all the values of the parameters are the same as for Figure~\ref{fig:weakMaxBloc}, except $\sigma_1=\sigma_3=10^{-3}$. The results are presented on Figure~\ref{fig:weakMaxBloc2}. We observe that the weak convergence seems to be of order $2$ for the scheme~\eqref{eq:order2}, but for small values of $h$ the error saturates due to the Monte Carlo approximation. This is illustrated on the right figure, which gives results for different values of the Monte Carlo sample size.

\begin{figure}[h]
\centering
\includegraphics*[width=0.48\textwidth,keepaspectratio]{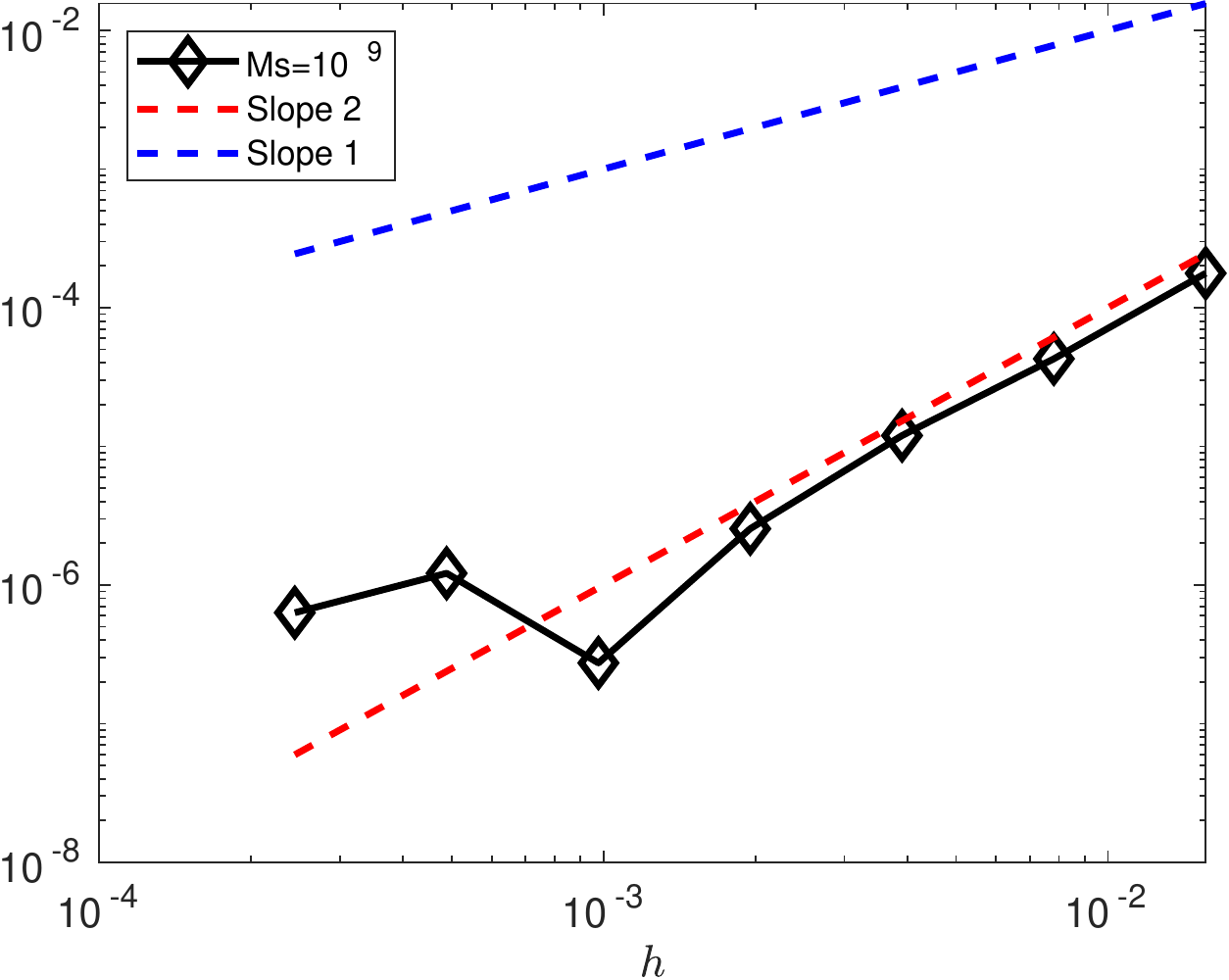}
\includegraphics*[width=0.48\textwidth,keepaspectratio]{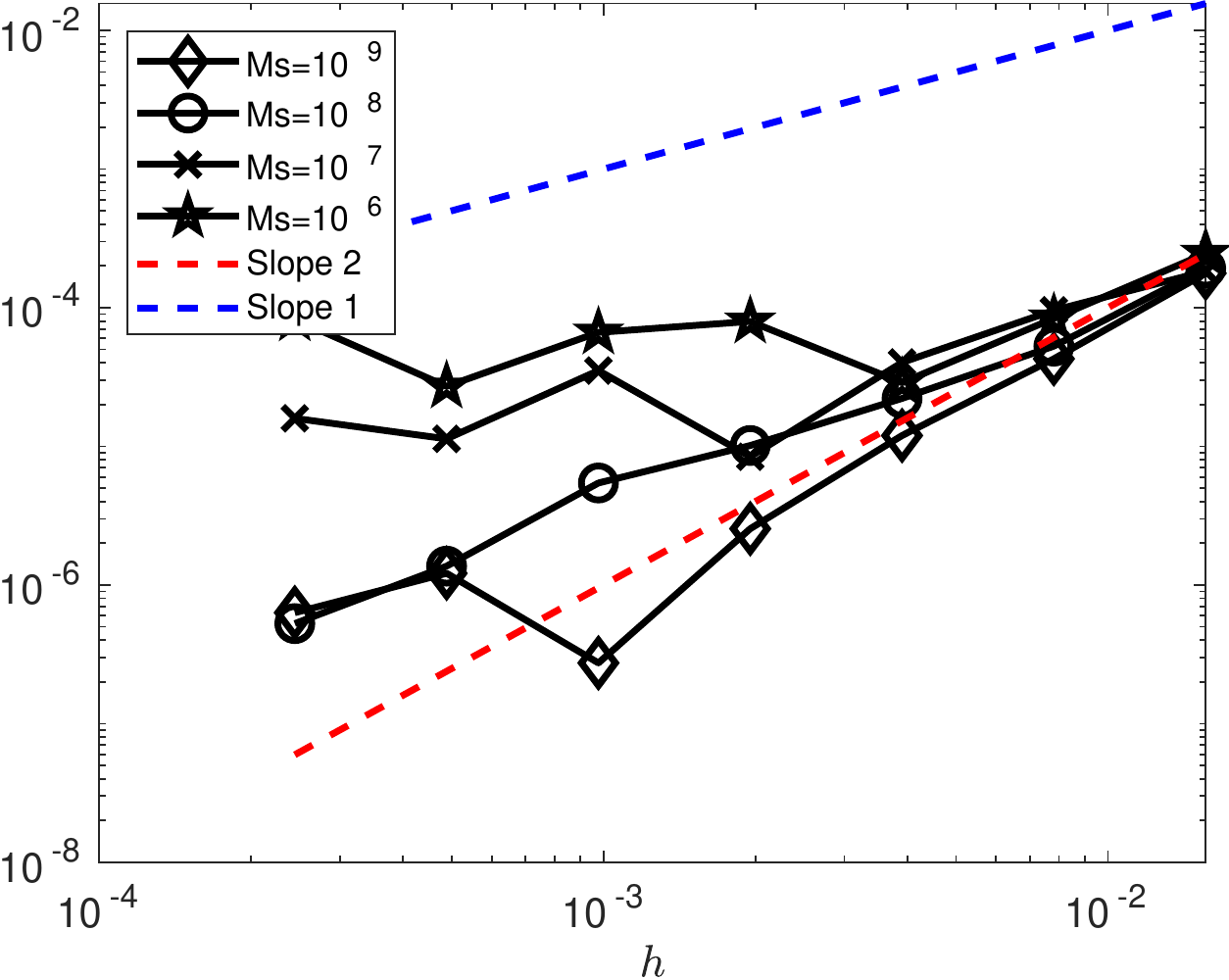}
\caption{Stochastic Maxwell--Bloch system: Weak errors for the weak order $2$ scheme~\eqref{eq:order2}. Left: $10^9$ Monte Carlo samples. Right: $10^6$ to $10^9$ Monte Carlo samples.} 
\label{fig:weakMaxBloc2}
\end{figure}

\subsubsection{Asymptotic preserving splitting scheme for the stochastic Maxwell--Bloch system}

In this subsection, we consider the multiscale version~\eqref{prob_eps}, parametrized by $\epsilon$, of the stochastic Maxwell--Bloch system. Based on the expression~\eqref{smbI} for the stochastic Poisson integrator, applying the general asymptotic preserving scheme~\eqref{APscheme} introduced in Section~\ref{APsplit} gives the scheme
\begin{equation}\label{smbIAP}
\left\lbrace
\begin{aligned}
y^{\epsilon,[n]}&=\exp(hY_{H_3})\circ\exp(hY_{H_1})\circ\exp(\frac{\sigma_3h\xi_{3}^{\epsilon,[n]}}{\epsilon}Y_{\widehat H_3})\circ\exp(\frac{\sigma_1h\xi_{1}^{\epsilon,[n]}}{\epsilon}Y_{\widehat H_1})(y^{\epsilon,[n-1]}),\\
\xi_k^{\epsilon,[n]}&=\xi_{k}^{\epsilon,[n-1]}-\frac{h}{\epsilon^2}\xi_{k}^{\epsilon,[n]}+\frac{\Delta_n W_k}{\epsilon}=\frac{1}{1+\frac{h}{\epsilon^2}}\Bigl(\xi_k^{\epsilon,[n-1]}+\frac{\Delta_n W_k}{\epsilon}\Bigr),\quad k=1,2.
\end{aligned}
\right.
\end{equation}
The initial values are $y^{\epsilon,[0]}=y^{[0]}=y(0)$ and $\xi_{k}^{\epsilon,[0]}=0$, $k=1,2$.

First, let us illustrate the qualitative behaviour of the scheme, for different values of $\epsilon$. For this numerical experiment, 
$\sigma_1=\sigma_3=0.1$, the initial value is $y(0)=(1,2,3)$ and the final time is $T=1$. The time step size is equal to $h=10^{-3}$. In Figure~\ref{fig:trajMBAPa}, we illustrate the preservation of the Casimir, up to an error of size $O(10^{-14})$, for the asymptotic preserving scheme applied with $\epsilon= 1,0.1,0.001$ (left) and for the stochastic Poisson integrator~\eqref{srbI}, formally $\epsilon=0$ (right). In Figure~\ref{fig:trajMBAPb}, we plot the evolution of the approximation of the trajectory $t_n\mapsto y(t_n)=(y_1(t_n),y_2(t_n),y_3(t_n))$, for different values of $\epsilon=1,0.1,0.001,0$. We observe that the trajectories are more regular when $\epsilon$ is large and converge to the solution of the stochastic Poisson integrator~\eqref{srbI} as $\epsilon$ tends to $0$.

\begin{figure}[h]
\begin{subfigure}{.5\textwidth}
\centering
\centering\includegraphics*[width=0.8\textwidth,keepaspectratio]{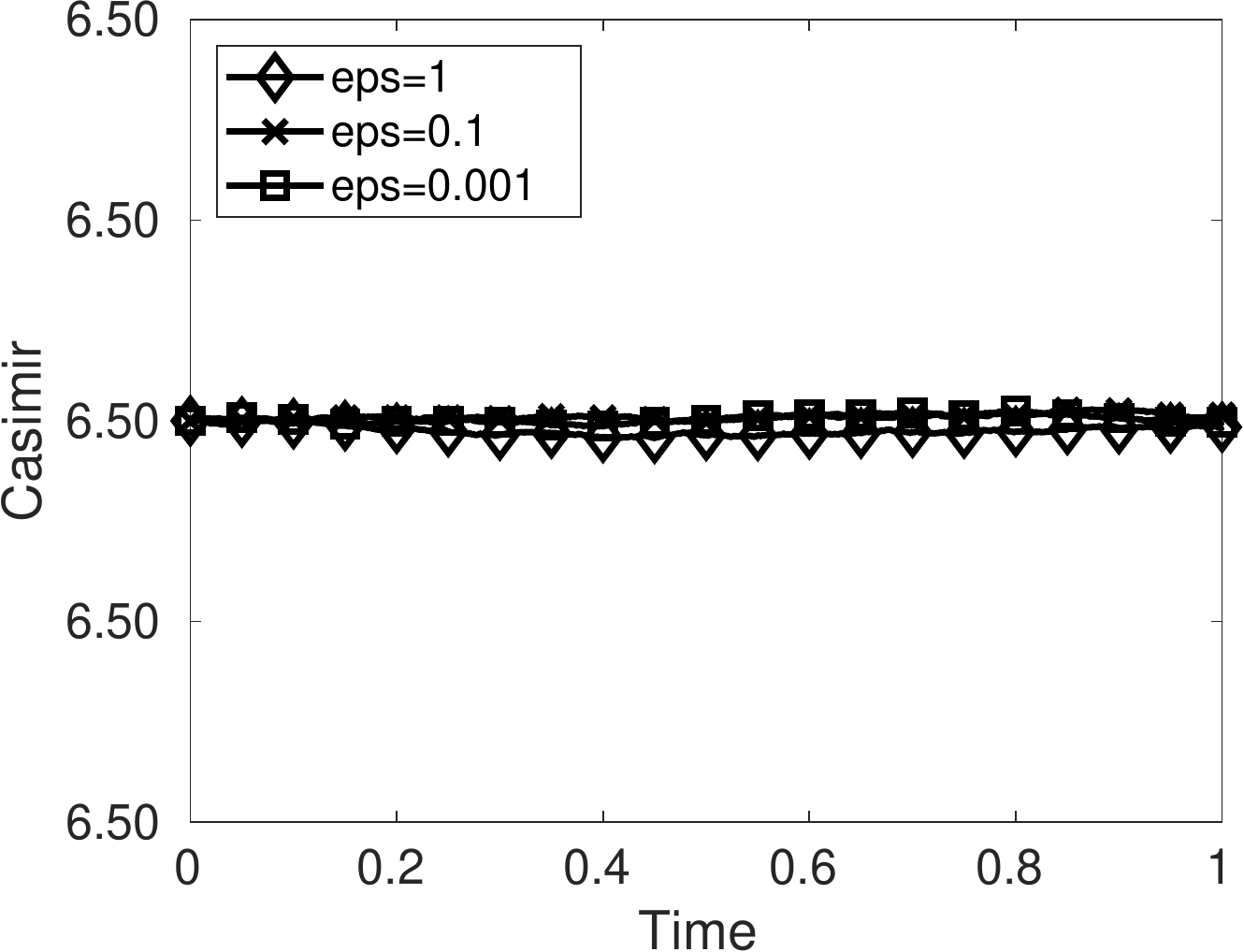}
\end{subfigure}
~
\begin{subfigure}{.5\textwidth}
\centering\includegraphics*[width=0.8\textwidth,keepaspectratio]{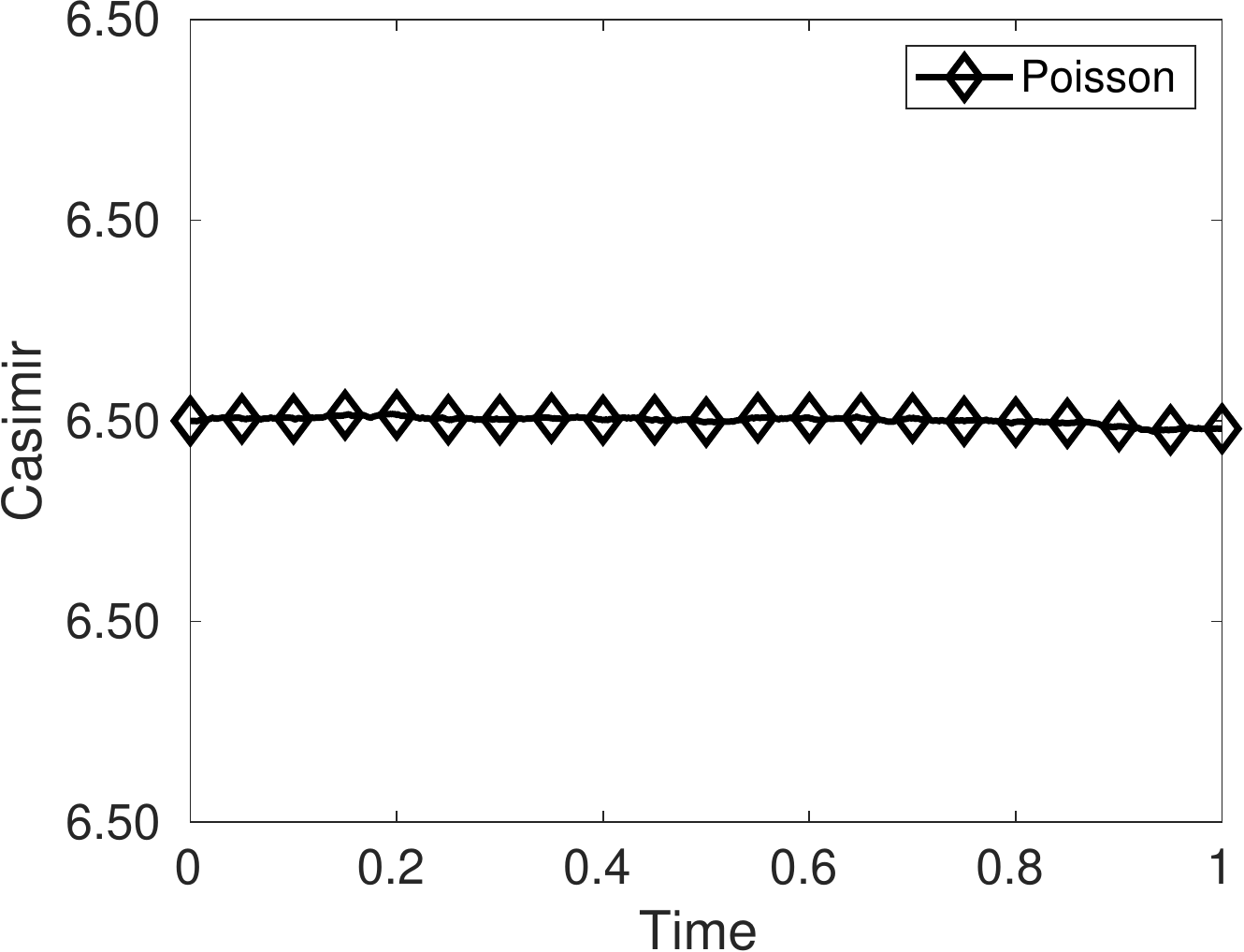}
\end{subfigure}
\caption{Stochastic Maxwell--Bloch with a single noise: evolution of the Casimir using the asymptotic preserving scheme~\eqref{smbIAP}. Left: $\epsilon=1,0.1,0.001$. Right: $\epsilon=0$.
}
\label{fig:trajMBAPa}
\end{figure}

\begin{figure}[h]
\centering
\begin{subfigure}{.5\textwidth}
\centering\includegraphics*[width=0.8\textwidth,keepaspectratio]{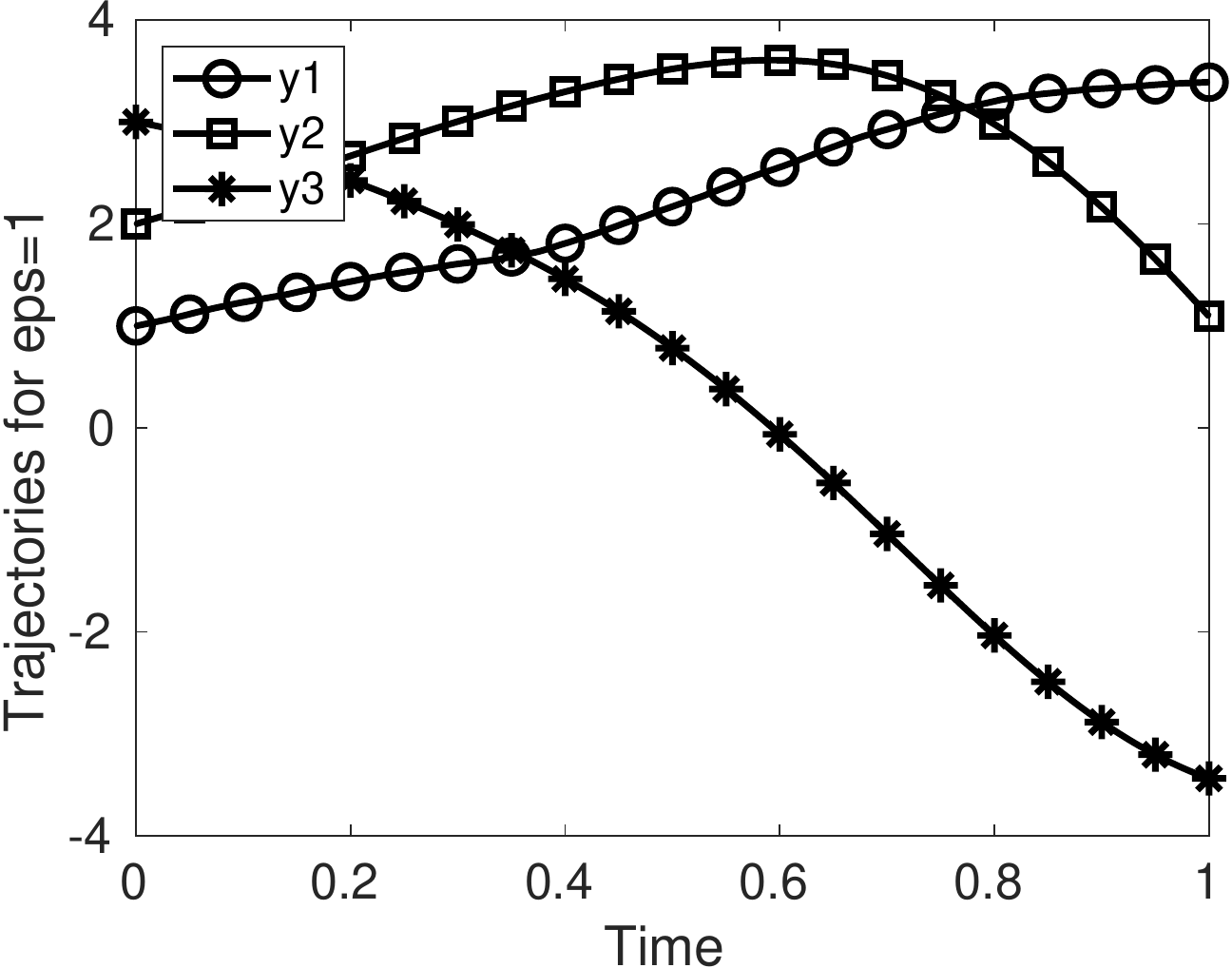}
\end{subfigure}
~
\begin{subfigure}{.5\textwidth}
\centering\includegraphics*[width=0.8\textwidth,keepaspectratio]{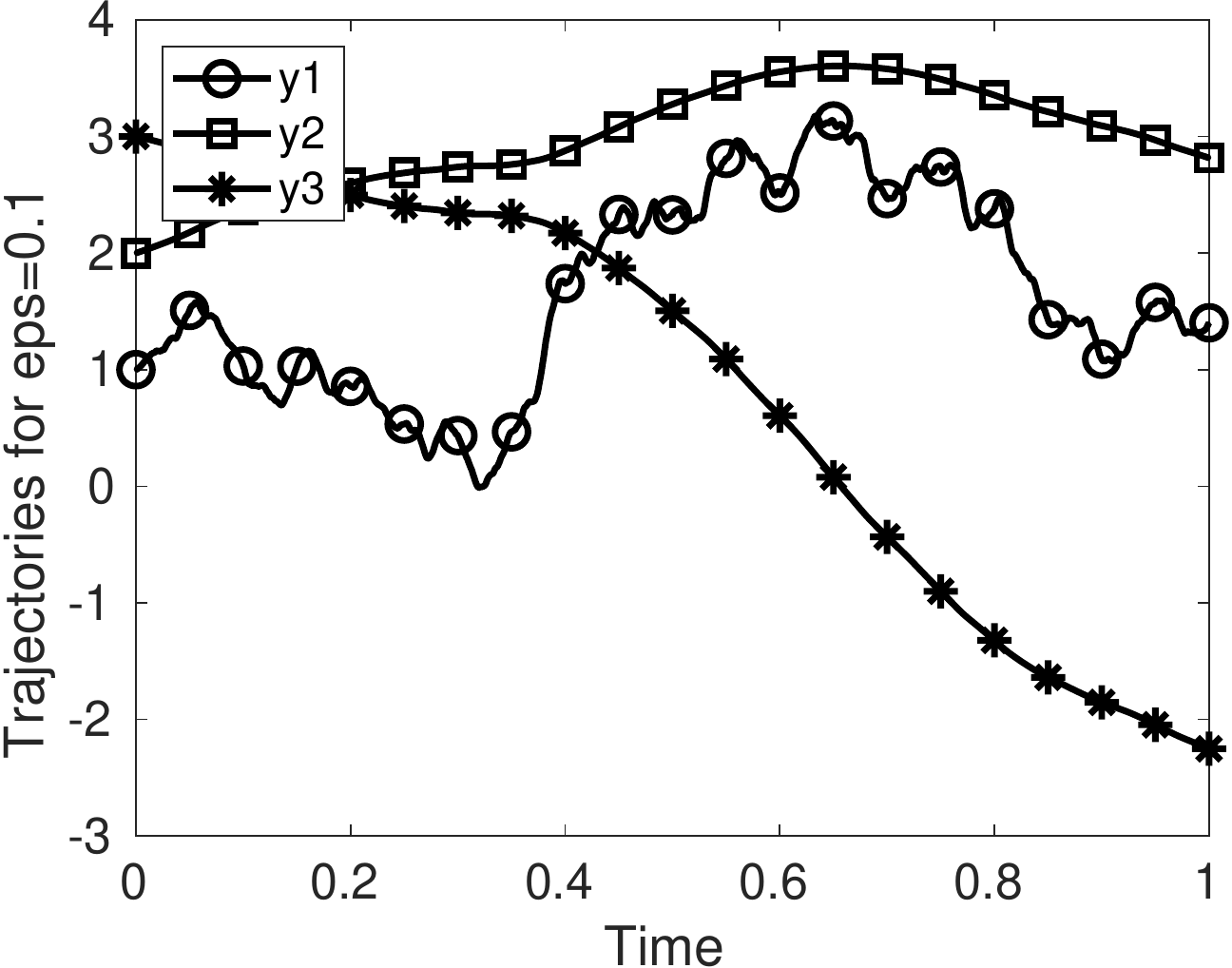}
\end{subfigure}
\newline
\begin{subfigure}{.5\textwidth}
\centering\includegraphics*[width=0.8\textwidth,keepaspectratio]{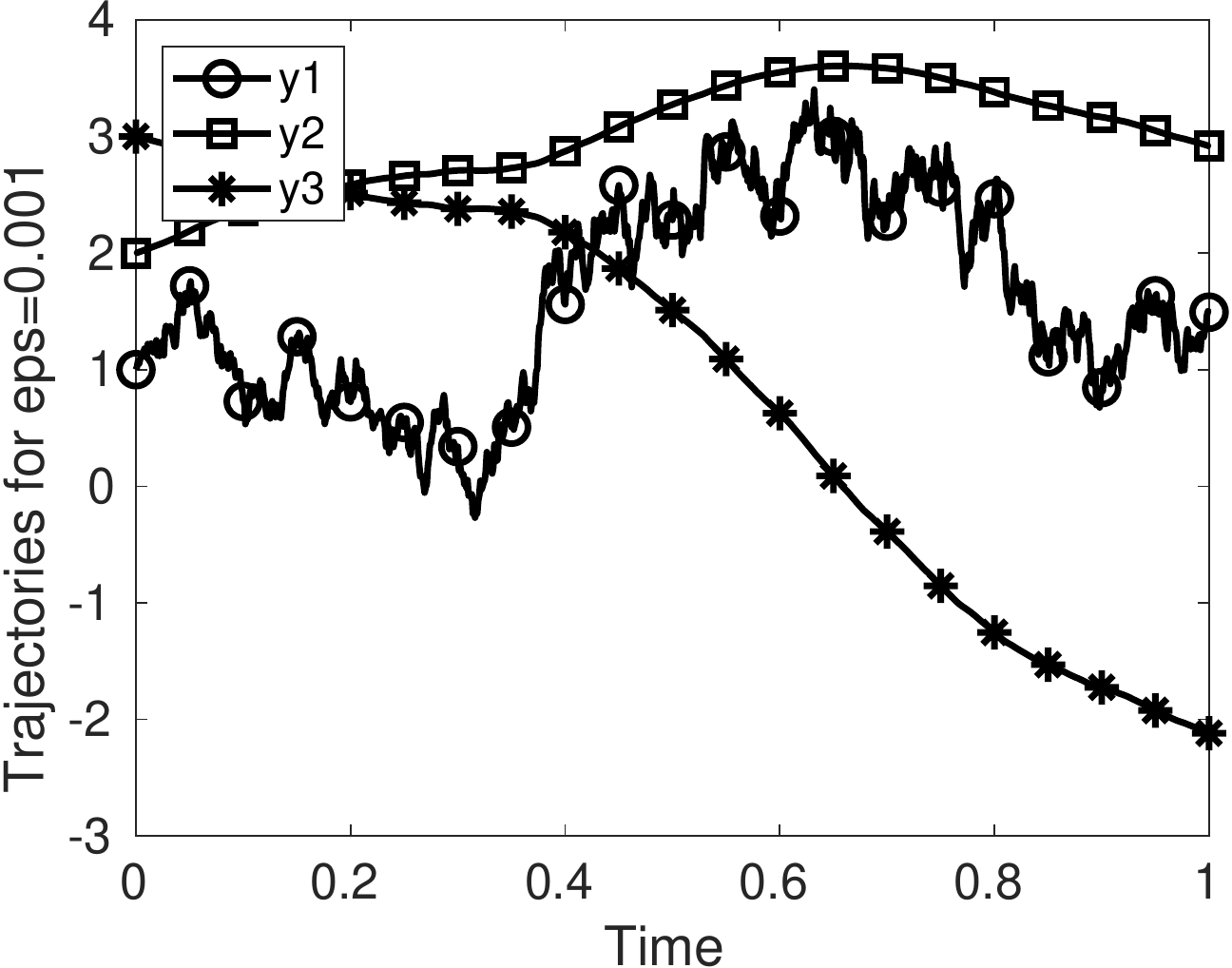}
\end{subfigure}
~
\begin{subfigure}{.5\textwidth}
\centering\includegraphics*[width=0.8\textwidth,keepaspectratio]{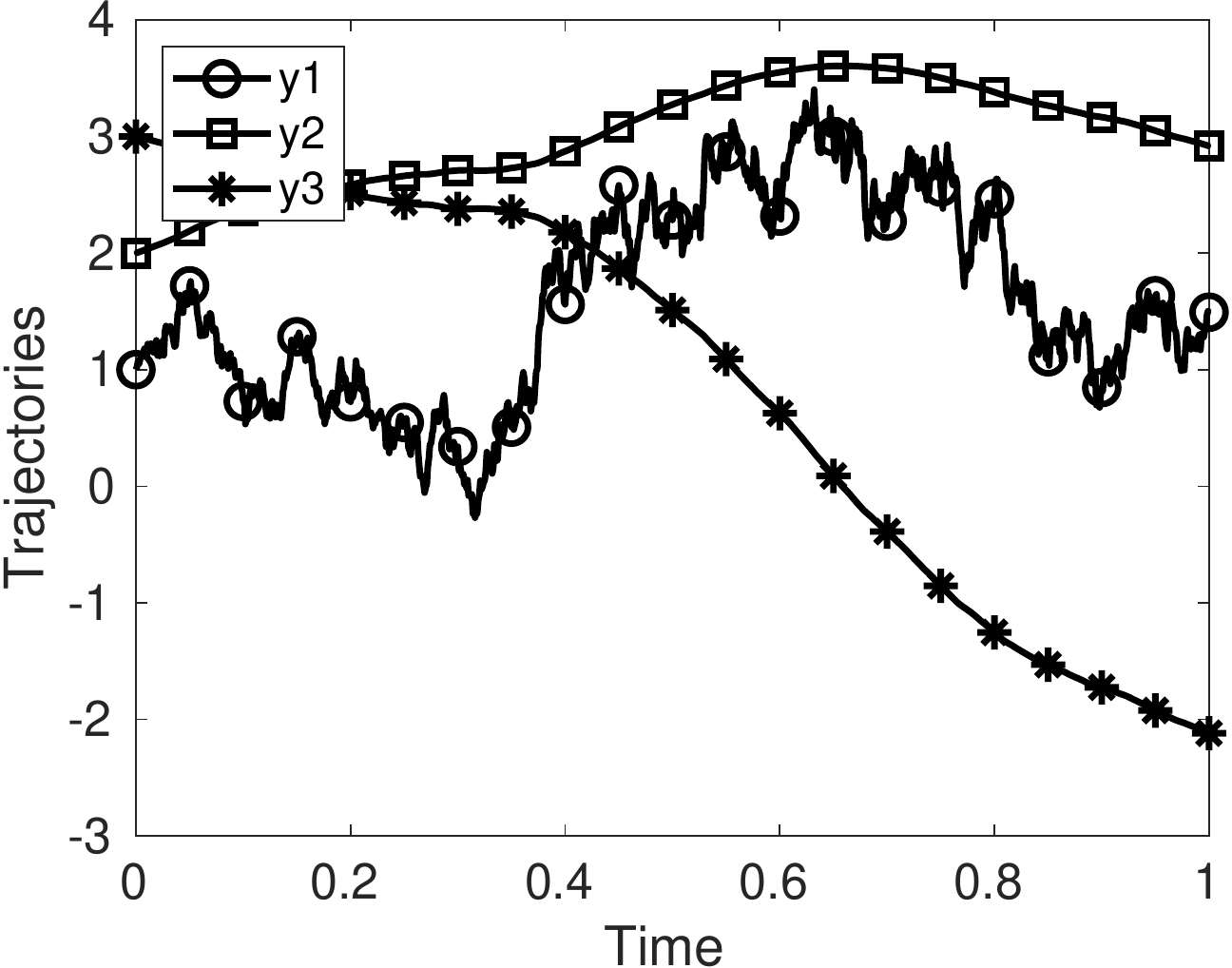}
\end{subfigure}
\caption{Stochastic Maxwell--Bloch system: trajectories of the numerical solution using the asymptotic preserving scheme~\eqref{smbIAP}. Top: left $\epsilon=1$, right $\epsilon=0.1$. Bottom: left $\epsilon=0.001$, right $\epsilon=0$.}
\label{fig:trajMBAPb}
\end{figure}

Finally, the last experiment of this subsection illustrates the uniform accuracy property of the splitting scheme~\eqref{smbIAP} with respect to the parameter $\epsilon$ in the weak sense. For this numerical experiment, $\sigma_1=\sigma_3=0.1$, the initial value is $y(0)=(1,2,3)$ and the final time is $T=1$. The reference solution is computed using each scheme with time step size $h_{\text{ref}}=2^{-16}$, and the schemes are applied with the range of time step sizes $h=2^{-6},\ldots,2^{-12}$. The expectation is approximated averaging the error over $M_s=10^9$ independent Monte Carlo samples. Finally, the test function is given by $\phi(y)=\sin(2\pi y_1)+\sin(2\pi y_2)+\sin(2\pi y_3)$. The parameter $\epsilon$ takes the following values: $\epsilon=10^{-2}, 10^{-3}, 10^{-4}, 10^{-5}$. The results are seen in Figure~\ref{fig:weakMaxBlocAP}. We observe that the weak error seems to be bounded uniformly with respect to $\epsilon$, with an order of convergence $1$. For a standard method such as the Euler--Maruyama scheme, the behaviour would be totally different: for fixed time step size $h$, the error is expected to be bounded away from $0$ when $\epsilon$ goes to $0$. Based on this numerical experiment, we conjecture that the asymptotic preserving scheme is uniformly accurate.

%

\begin{figure}[h]
\centering
\includegraphics*[width=0.48\textwidth,keepaspectratio]{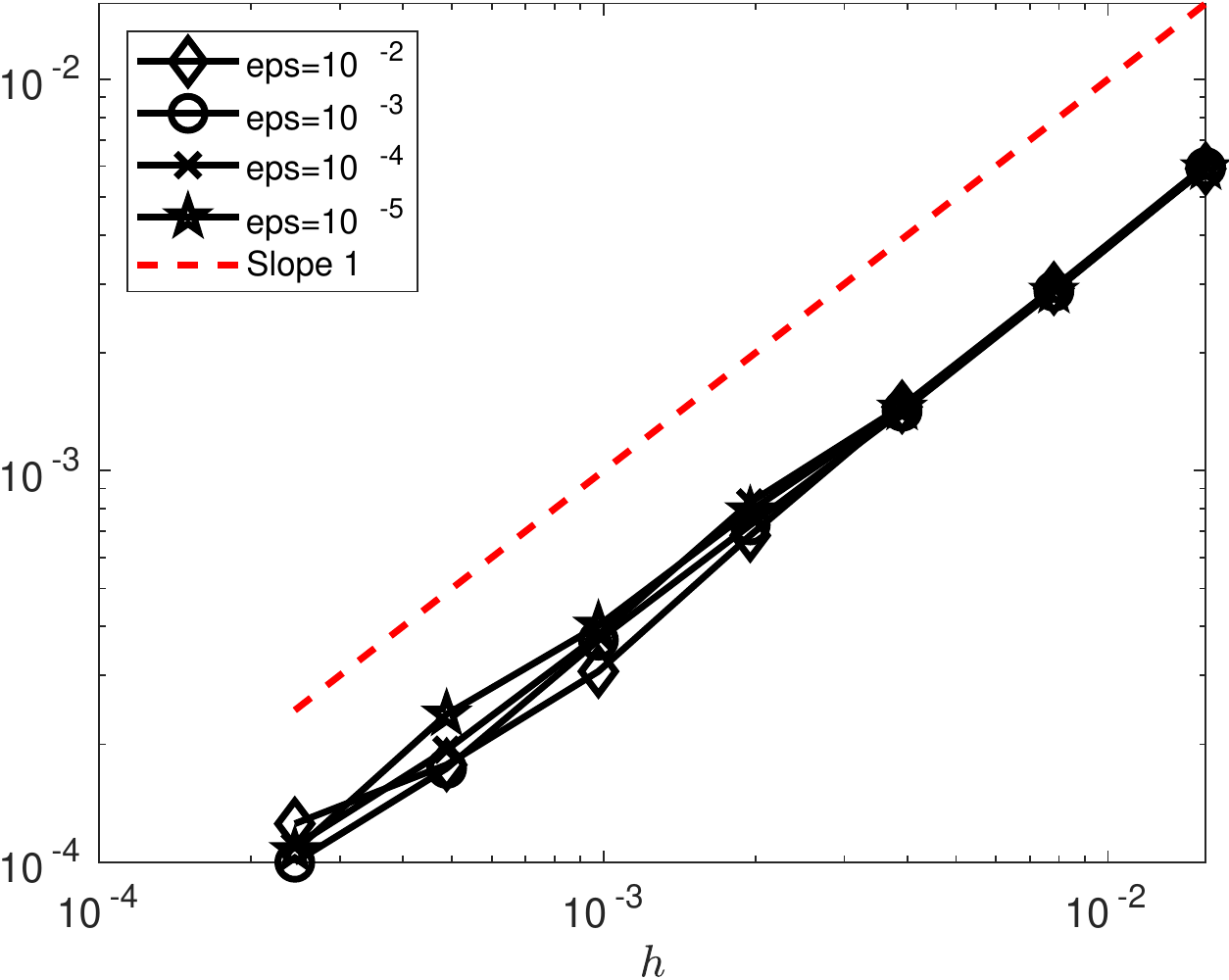}
\caption{Stochastic Maxwell--Bloch: Weak errors of the asymptotic preserving scheme~\eqref{smbIAP} 
for $\epsilon=10^{-2}, 10^{-3}, 10^{-4}, 10^{-5}$.} 
\label{fig:weakMaxBlocAP}
\end{figure}

\subsection{Explicit stochastic Poisson integrators for the stochastic rigid body system}\label{ssec:PRB}

This subsection presents explicit stochastic Poisson integrators for the stochastic rigid body system~\eqref{srb} (Example~\ref{expl-SRB}). We first give a detailed construction of the splitting scheme, which gives a stochastic Poisson integrator. We then illustrate its qualitative properties (preservation of the Casimir function) and strong and weak error estimates of the proposed scheme by numerical experiments. Finally, we illustrate the asymptotic preserving property (see Section~\ref{APsplit}) for a multiscale version of the system.

Below we state and prove Proposition~\ref{thm-srb}, which gives strong and weak rates of convergence of the proposed explicit scheme. This is a non-trivial result since the coefficients of the considered stochastic differential equation are not globally Lipschitz continuous.


\subsubsection{Presentation of the splitting scheme for the stochastic rigid body system}

To apply the strategy described in Section~\ref{ssec:LP} and construct explicit stochastic Poisson integrators, we first follow the approach from~\cite{MR1246065,MR2009376} for the deterministic rigid body system. The Hamiltonian function $H$ is split as $H=H_1+H_2+H_3$, with $H_j=\frac12\frac{y_j^2}{I_j}$ for $j=1,2,3$. Recall also that the Hamiltonian functions appearing in the stochastic part of the dynamics are given by $\widehat H_j=\frac12\frac{y_j^2}{\widehat I_j}$, for $j=1,2,3$.

The application of the general splitting integrator~\eqref{slpI} for the stochastic rigid body system~\eqref{srb} requires to compute the exact solutions of the deterministic subsystems
\[
\dot{y}_j=B(y_j)\nabla H_j(y_j)
\]
and of the stochastic subsystems
\[
\diff y_j=B(y_j)\nabla \widehat H_j(y_j)\circ \diff W_j(t).
\]
As explained for instance in~\cite{MR1246065,MR2009376} for the deterministic system, it is straightforward to solve such subsystems. 
We only provide the details when $j=1$. In that case, the deterministic subsystem is of the type
$$
\begin{cases}
\dot y_1&=0 \\
\dot y_2&=y_1y_3/I_1 \\
\dot y_3&=-y_1y_2/I_1.
\end{cases}
$$
The first equation yields that $y_1$ is constant, i.\,e. $y_1(t)=y_1(0)$ for all $t\ge 0$. As a consequence, $(y_2,y_3)$ is a solution of a linear ordinary differential equation.

The deterministic subsystem when $j=1$ thus admits the exact solution
$$
\begin{pmatrix}y_1(t)\\y_2(t)\\y_3(t)\end{pmatrix}=
\begin{pmatrix}1 & 0 & 0\\ 0 & \cos(\theta t) & \sin(\theta t)\\ 0 & -\sin(\theta t) & \cos(\theta t)\end{pmatrix}y(0),
$$
where $\theta=\frac{y_1(0)}{I_1}$. Similarly, the stochastic subsystem when $j=1$ is written as
$$
\begin{cases}
\diff y_1&=0\\
\diff y_2&= y_1y_3/\widehat I_1\circ\,\diff W_1\\
\diff y_3&=-y_1y_2/\widehat I_1 \circ\,\diff W_1
\end{cases}
$$
and it admits the exact solution
$$
\begin{pmatrix}y_1(t)\\y_2(t)\\y_3(t)\end{pmatrix}=
\begin{pmatrix}1 & 0 & 0\\ 0 & \cos(\theta W_1(t)) & \sin(\theta W_1(t))\\ 0 & -\sin(\theta W_1(t)) & \cos(\theta W_1(t))\end{pmatrix}y(0).
$$
where $\theta=\frac{y_1(0)}{\widehat I_1}$.

The solutions of the deterministic and stochastic subsystems when $j=2,3$ have similar expressions, which are not written here for brevity.

Finally, setting $\Phi_h^{\text{det}}=\exp(hY_{H_3})\circ\exp(hY_{H_2})\circ\exp(hY_{H_1})$ and $\Phi_{\Delta W}^{\text{stoch}}=\exp(\Delta W_3Y_{\widehat H_3})\circ 
\exp(\Delta W_2Y_{\widehat H_2})\circ\exp(\Delta W_1Y_{\widehat H_1})$, the general splitting integrator~\eqref{slpI} applied to the stochastic rigid body system~\eqref{srb} gives 
\begin{align}\label{srbI}
\Phi_h=\Phi_h^{\text{det}}\circ\Phi_{\Delta W}^{\text{stoch}}&=\exp(hY_{H_3})\circ\exp(hY_{H_2})\circ\exp(hY_{H_1})
\nonumber\\
&\quad \circ\exp(\Delta W_3Y_{\widehat H_3})\circ \exp(\Delta W_2Y_{\widehat H_2})\circ\exp(\Delta W_1Y_{\widehat H_1}).
\end{align}
Owing to Proposition~\ref{propo:sPi}, the explicit splitting scheme~\eqref{srbI} is a stochastic Poisson integrator. In particular, it preserves the Casimir function $C(y)=y_1^2+y_2^2+y_3^2$, which has compact level sets.



\subsubsection{Preservation of the Casimir of the stochastic rigid body system}

Let us first illustrate the qualitative behaviour of the stochastic Poisson integrator~\eqref{srbI} introduced above. In this numerical experiment, 
the moments of inertia are $I=(2,1,2/3)$, $\widehat I=(1,2,3)$, the initial value is $y(0)=(\cos(1.1),0,\sin(1.1))$ and the final time is $T=20$.
In Figure~\ref{fig:RB}, we compare the numerical solutions given 
by the classical Euler--Maruyama scheme (applied to the It\^o formulation of the system), the stochastic midpoint scheme from~\cite{Milstein2002a}, and the explicit splitting scheme~\eqref{srbI}. The time step size is equal to $h=0.2$ ($T/h=100$). A truncation of the noise is used for this experiment, 
see above for details. As proved in Proposition~\ref{propo:sPi}, we observe that the Casimir function $C(y)=y_1^2+y_2^2+y_3^2$ is preserved when using the stochastic Poisson integrator~\eqref{srbI}. The Casimir function is also preserved when using the stochastic midpoint scheme: indeed, this integrator is known to preserve quadratic invariants, see~\cite{MR3210739}.

In addition, a plot of the evolution of the Hamiltonian for the three schemes (middle figure), and of the trajectory on the sphere of the proposed splitting scheme (right figure) are presented in Figure~\ref{fig:RB}. 

\begin{figure}[h]
\centering
\includegraphics*[width=0.3\textwidth,keepaspectratio]{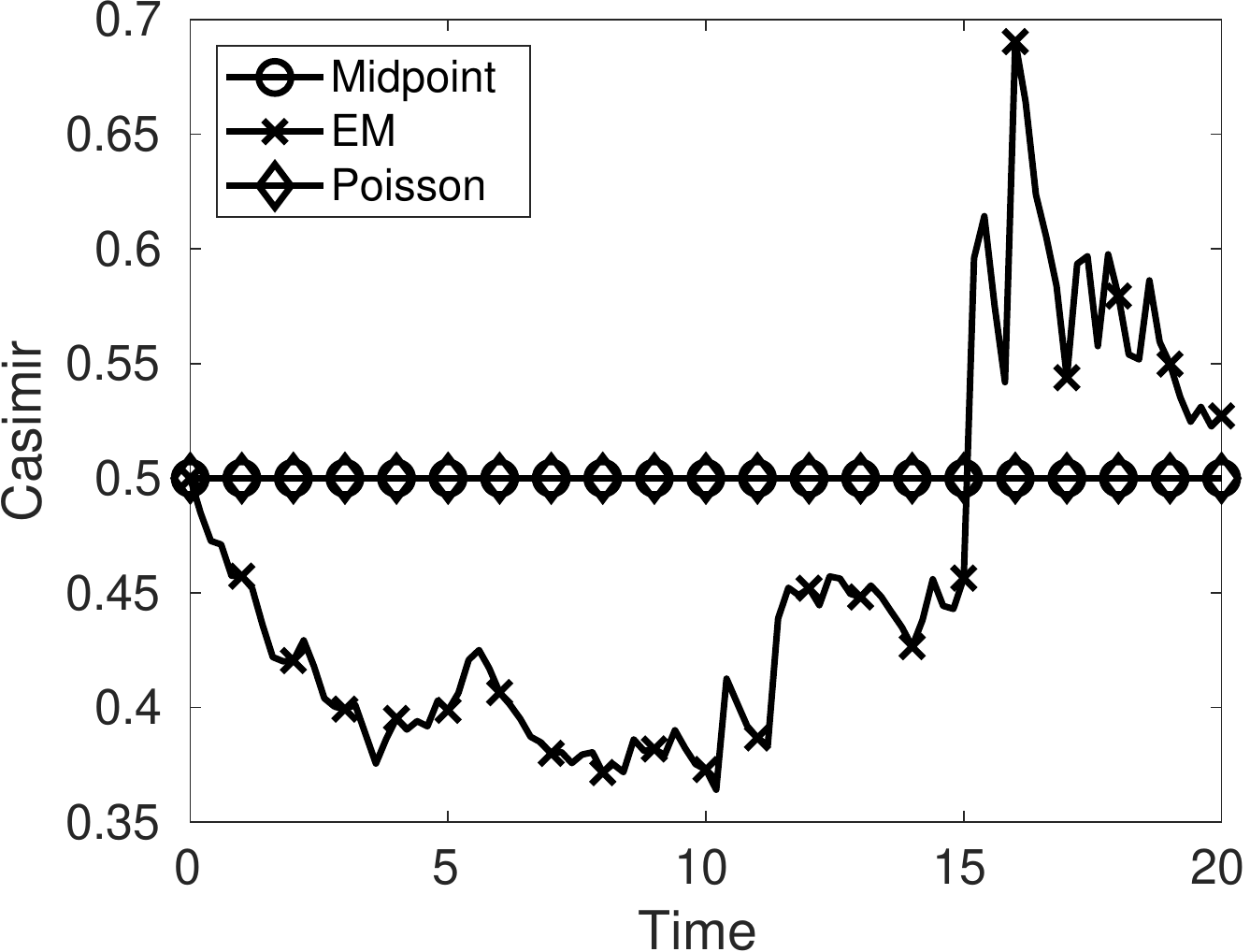}
\includegraphics*[width=0.3\textwidth,keepaspectratio]{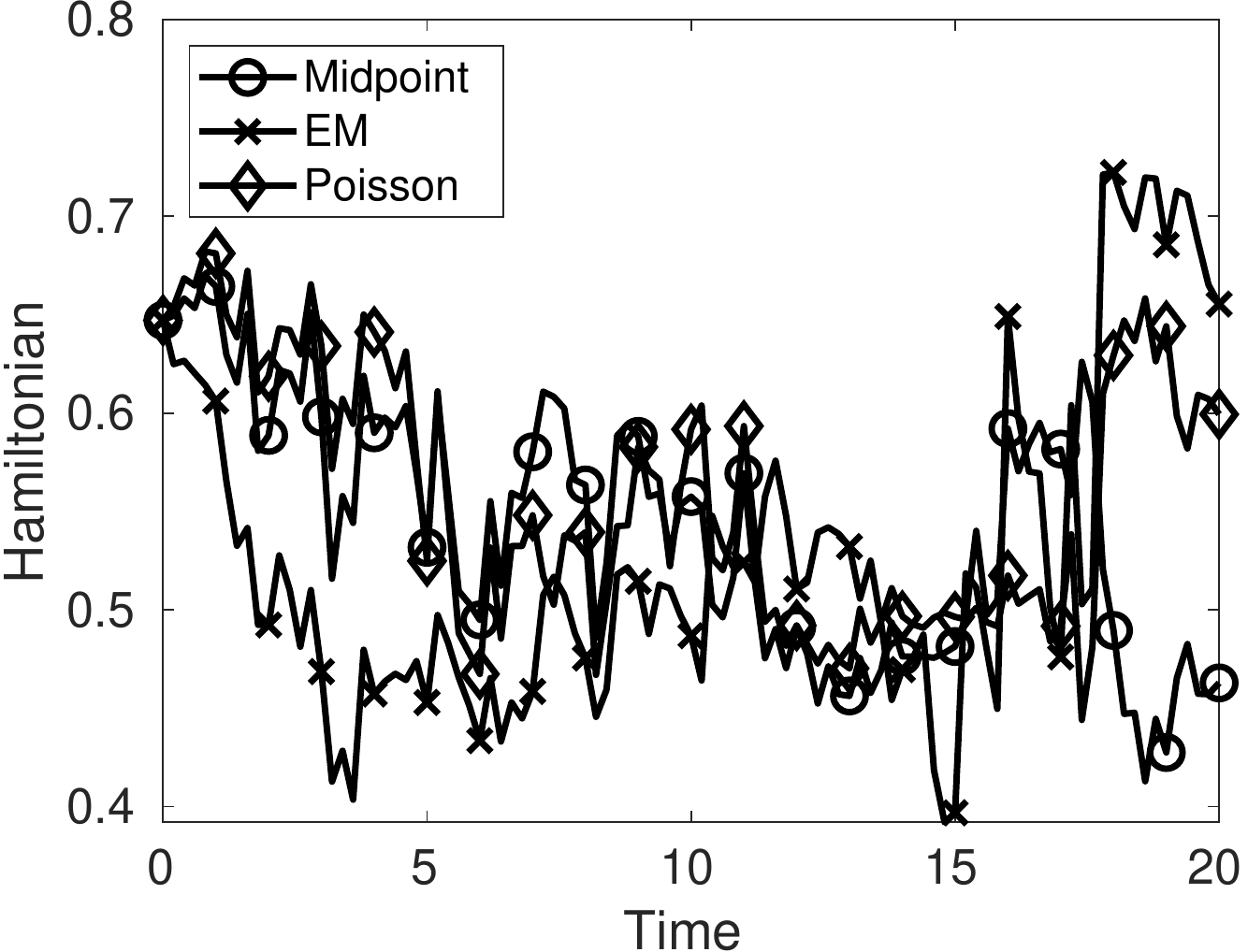}
\includegraphics*[width=0.3\textwidth,keepaspectratio]{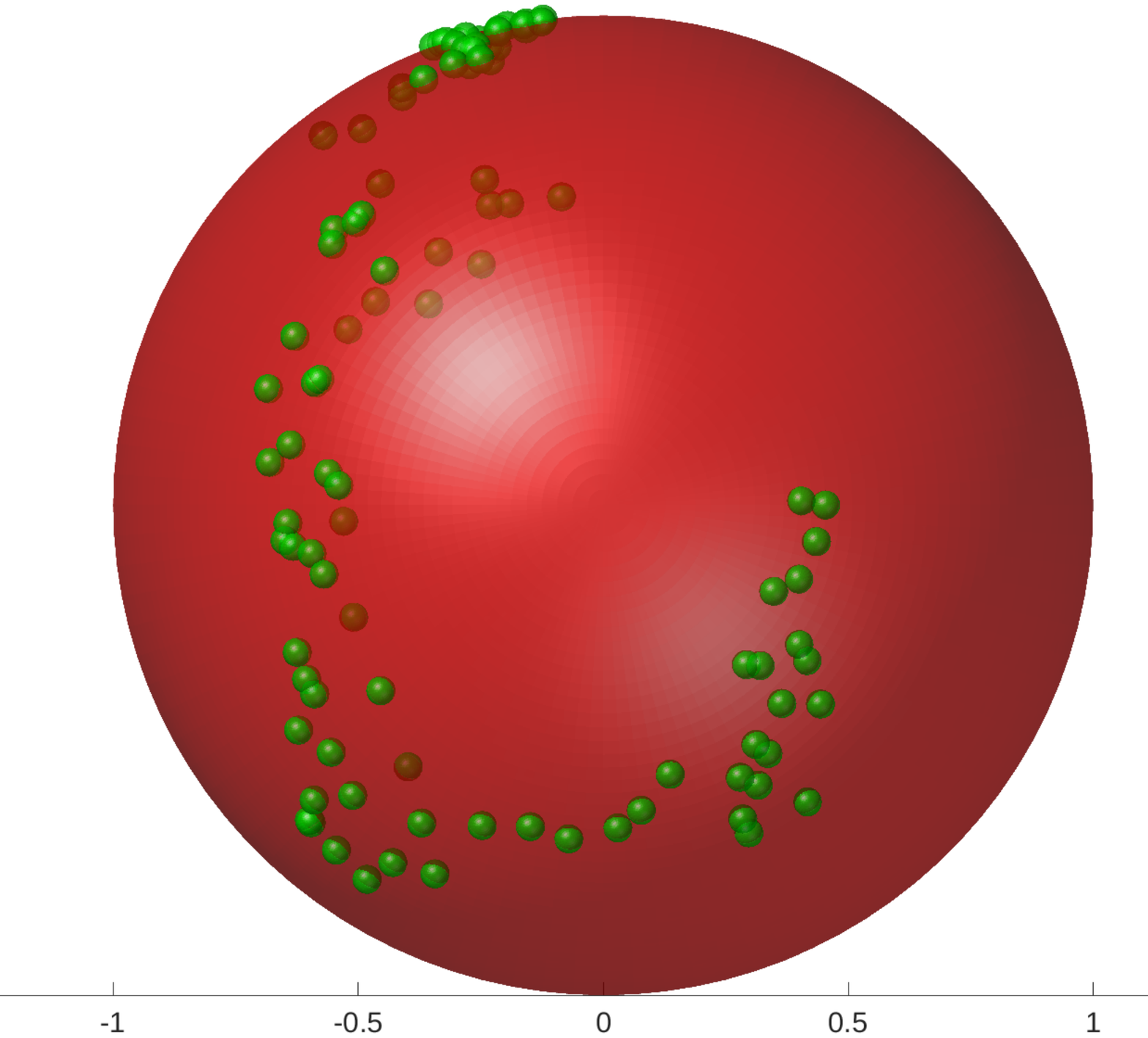}
\caption{Stochastic rigid body system: Qualitative behaviour of the Euler--Maruyama scheme ($\times$), the midpoint scheme ($\circ$), 
and the explicit stochastic Poisson integrator ($\diamond$). Left: preservation of the Casimir. Middle: evolution of the Hamiltonian. Right: trajectory on the sphere for the scheme~\eqref{srbI}.}
\label{fig:RB}
\end{figure}

\subsubsection{Strong and weak convergence of the explicit stochastic Poisson integrator for the stochastic rigid body system}


The general Theorem~\ref{thm-general} is applicable in the case of stochastic rigid body system since the Casimir function $C$ has compact level sets.

%
\begin{proposition}\label{thm-srb}
Consider a numerical discretisation of the stochastic rigid body system~\eqref{srb} by the stochastic Poisson integrator~\eqref{srbI}. 
Then, the strong order of convergence of this scheme is $1/2$ and the weak order of convergence is $1$. 
\end{proposition}

\begin{proof}
The stochastic Poisson system~\eqref{srb} admits the Casimir function $y\mapsto C(y)=y_1^2+y_2^2+y_3^2$, which has compact level sets.  
It then suffices to apply the general convergence result, Theorem~\ref{thm-general}, which 
concludes the proof of Proposition~\ref{thm-srb}.
\end{proof}

Let us first illustrate the strong convergence result. We compare the behaviours of the three integrators introduced above: the Euler--Maruyama scheme, the stochastic midpoint scheme, 
and the explicit stochastic Poisson integrator~\eqref{srbI}. Note that Proposition~\ref{thm-srb} is valid only for the splitting scheme~\eqref{srbI}. For this numerical experiment, the moments of inertia are $I=(2,1,2/3)$, $\widehat I=(1,2,3)$, the initial value is $y(0)=(\cos(1.1),0,\sin(1.1))$ and the final time is $T=1$. The reference solution is computed using each scheme with time step size $h_{\text{ref}}=2^{-16}$, and the schemes are applied with the range of time step sizes $h=2^{-5},\ldots,2^{-13}$. The expectation is approximated averaging the error over $M_s=500$ independent Monte Carlo samples.

The results are presented in Figure~\ref{fig:msSRBL}: we observe a strong order of convergence equal to $1/2$ for the proposed explicit stochastic Poisson integrator~\eqref{srbI}, which confirms the result of Proposition~\ref{thm-srb}.

\begin{figure}[h]
\centering
\includegraphics*[width=0.48\textwidth,keepaspectratio]{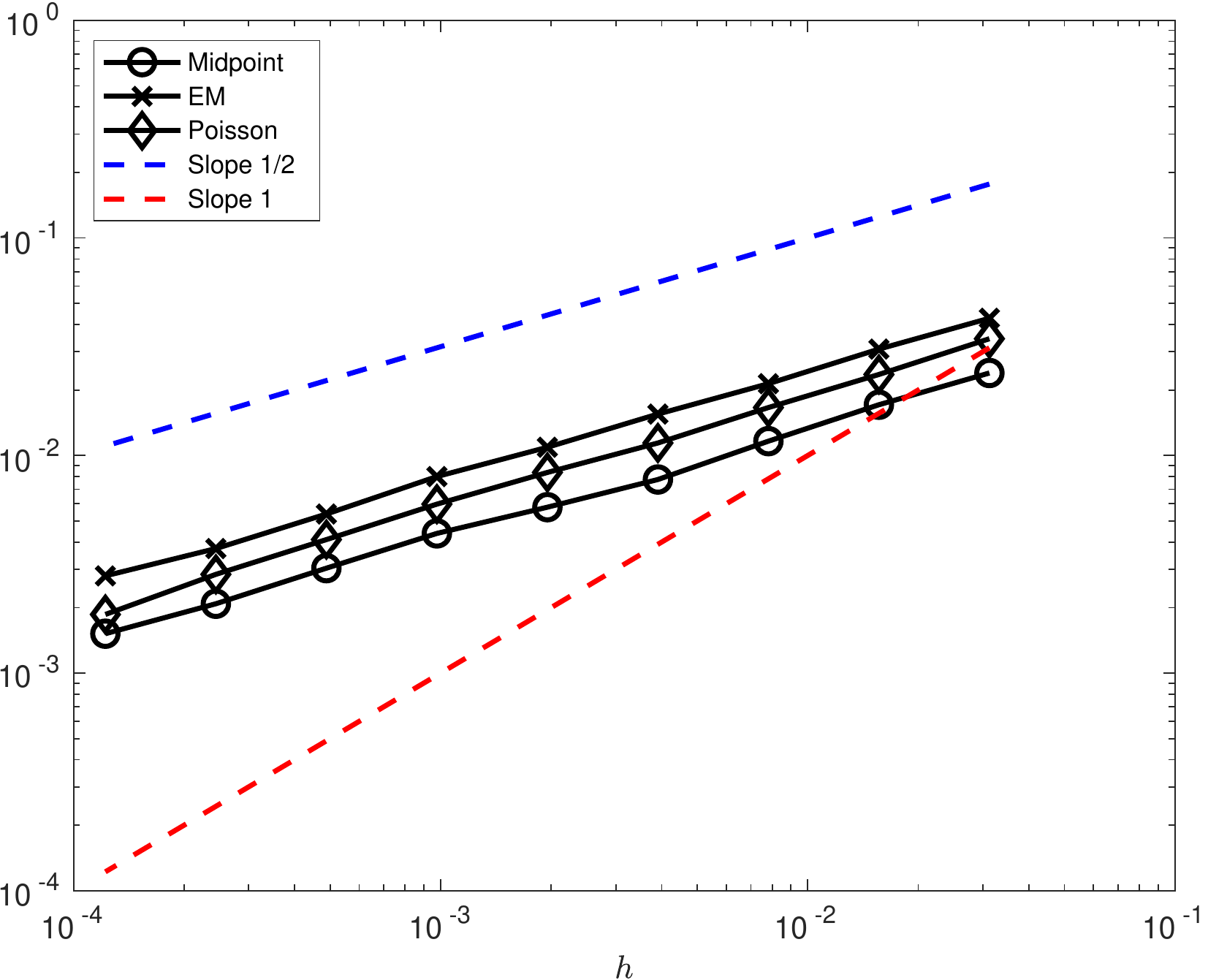}
\caption{Stochastic rigid body system: Strong errors for the Euler--Maruyama scheme ($\times$), the midpoint scheme ($\circ$), and the explicit stochastic Poisson integrator ($\diamond$).}
\label{fig:msSRBL}
\end{figure}

We now illustrate the weak convergence of the stochastic Poisson integrator~\eqref{srbI}. For this numerical experiment, the moments of inertia are $I=\hat{I}=(2,1,2/3)$, the initial value is $y(0)=(\cos(1.1),0,\sin(1.1))$ and the final time is $T=1$. The reference solution is computed using each scheme with time step size $h_{\text{ref}}=2^{-16}$, and the schemes are applied with the range of time step sizes $h=2^{-6},\ldots,2^{-12}$. The expectation is approximated averaging the error over $M_s=10^9$ independent Monte Carlo samples. Finally, the test function is given by $\phi(y)=\sin(2\pi y_1)+\sin(2\pi y_2)+\sin(2\pi y_3)$.
%
%
The results are presented in Figure~\ref{fig:weakRGB}. We observe a weak order $1$ for the proposed explicit stochastic Poisson integrator~\eqref{srbI}, which confirms the result of Proposition~\ref{thm-srb}.

\begin{figure}[h]
\centering
\includegraphics*[width=0.48\textwidth,keepaspectratio]{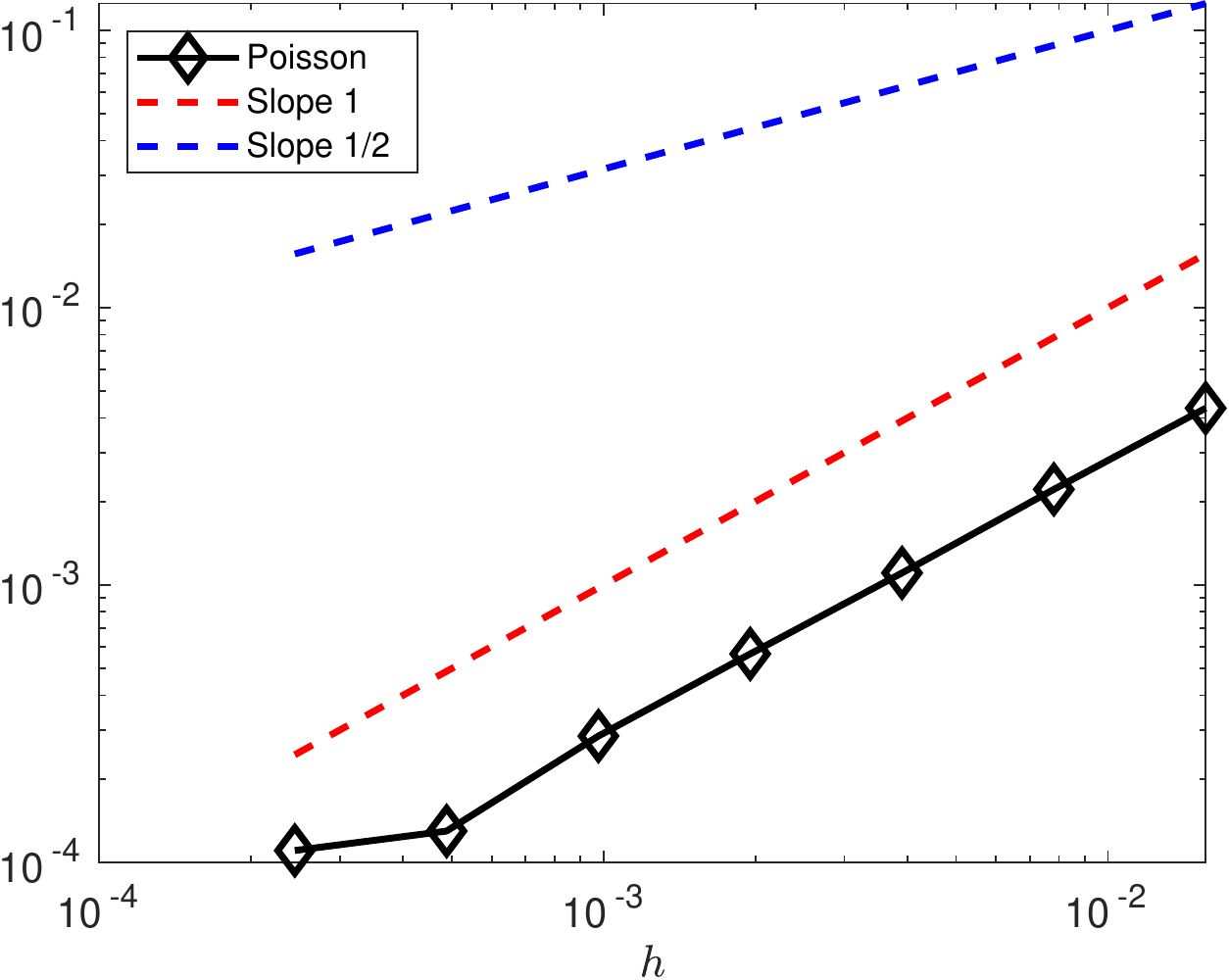}
\caption{Stochastic rigid body system: Weak error for the explicit stochastic Poisson integrator~\eqref{srbI}.}
\label{fig:weakRGB}
\end{figure}

To conclude this subsection, let us provide a numerical experiment using the scheme of weak order $2$ presented in Remark~\ref{rem:order2}. To do so, we consider the following variant of the stochastic rigid body system~\eqref{srb}:
\begin{align*}
\diff\begin{pmatrix}y_1\\y_2\\y_3\end{pmatrix}&=
B(y)
\left(\nabla H(y)\,\diff t+\sigma_1\nabla \widehat H_1(y)\circ\,\diff W_1(t)
+\sigma_2\nabla \widehat H_2(y)\circ\,\diff W_2(t)\right.\nonumber\\
&\quad+\left.\sigma_3\nabla \widehat H_3(y)\circ\,\diff W_3(t) \right),
\end{align*}
with nonnegative real numbers $\sigma_1,\sigma_2,\sigma_3$. For this numerical experiment, all the values of the parameters are the same as for Figure~\ref{fig:weakRGB}, except the values of the additional parameters $\sigma_1,\sigma_2,\sigma_3$: one has either three Wiener processes with $(\sigma_1,\sigma_2,\sigma_3)=(10^{-3},10^{-3},10^{-3})$ (left figure), or a single Wiener process, with three possible choices $(\sigma_1,\sigma_2,\sigma_3)=(10^{-3},0,0)$, $(\sigma_1,\sigma_2,\sigma_3)=(0,10^{-3},0)$ and $(\sigma_1,\sigma_2,\sigma_3)=(0,0,10^{-3})$ (right figure). The results are presented in Figure~\ref{fig:weakRGBorder2}. We observe that the weak convergence seems to be of order $2$ for the scheme~\eqref{eq:order2}, however for small values of $h$ the error saturates due to the Monte Carlo approximation.

\begin{figure}[h]
\centering
\includegraphics*[width=0.48\textwidth,keepaspectratio]{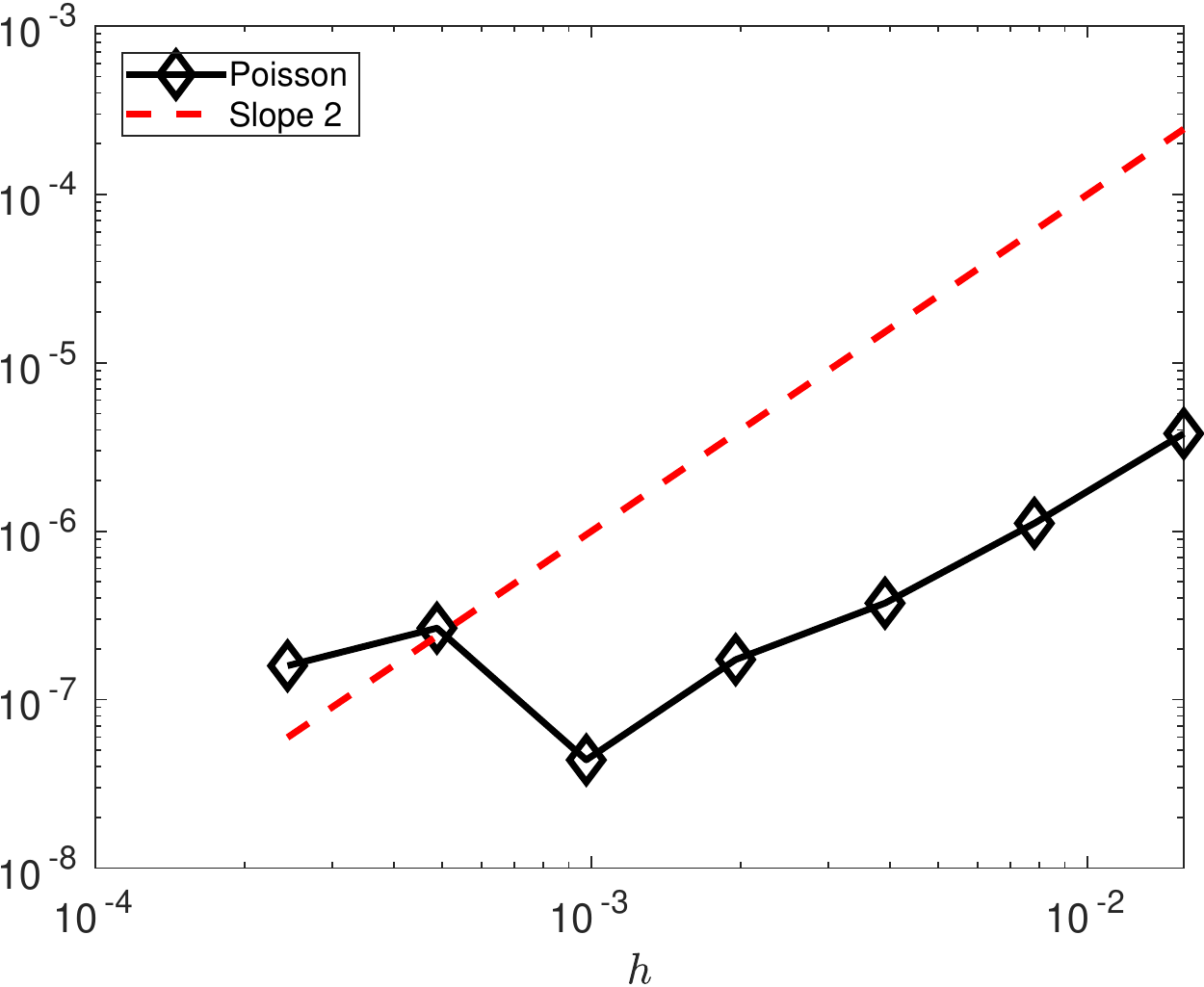}
\includegraphics*[width=0.48\textwidth,keepaspectratio]{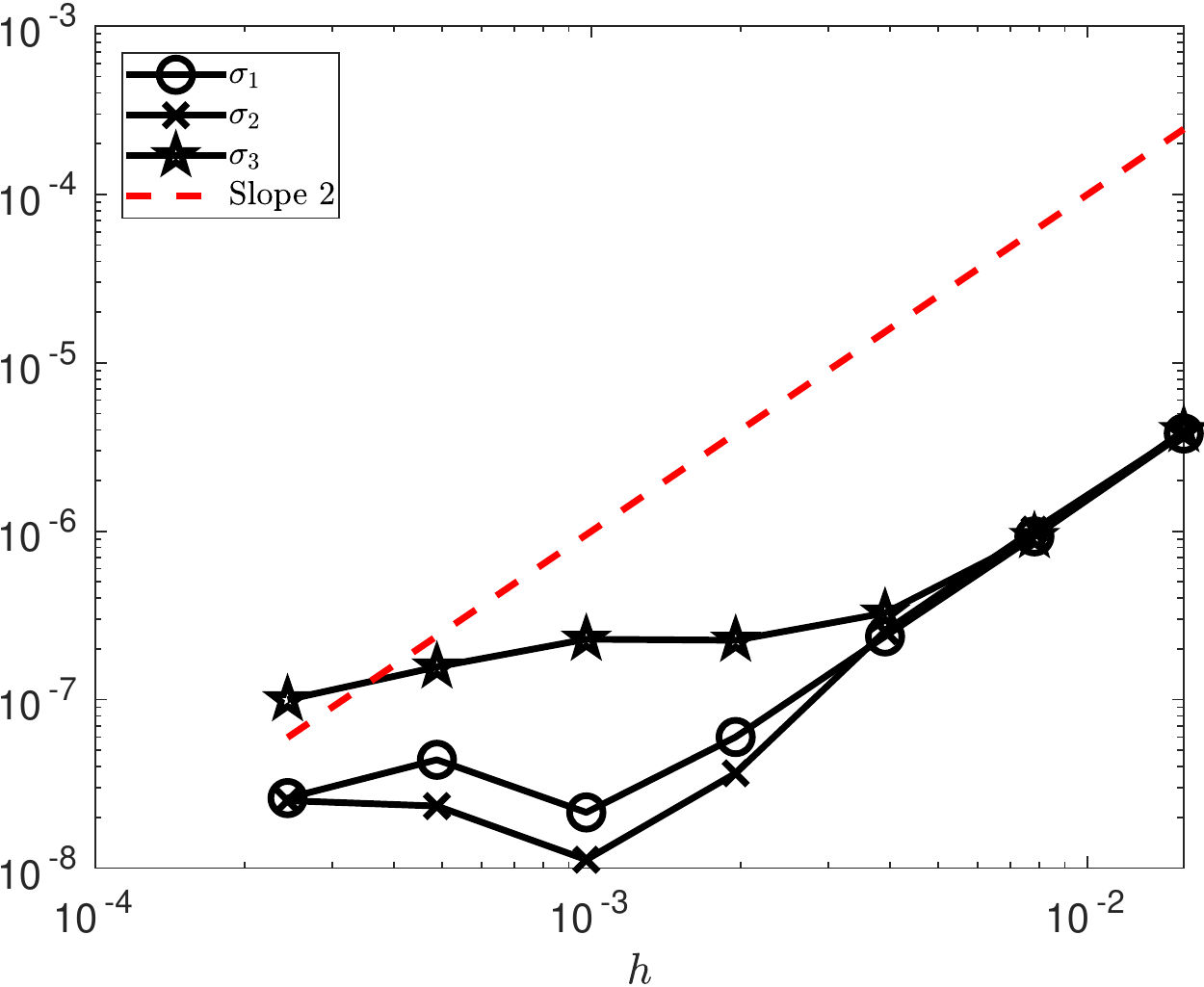}
\caption{Stochastic rigid body system: Weak error for the integrator~\eqref{eq:order2}. Left: $(\sigma_1,\sigma_2,\sigma_3)=(10^{-3},10^{-3},10^{-3})$. Right: $(\sigma_1,\sigma_2,\sigma_3)=(10^{-3},0,0)$, $(\sigma_1,\sigma_2,\sigma_3)=(0,10^{-3},0)$ and $(\sigma_1,\sigma_2,\sigma_3)=(0,0,10^{-3})$.} 
\label{fig:weakRGBorder2}
\end{figure}

\subsubsection{Asymptotic preserving splitting scheme for the stochastic rigid body system}

In this subsection, we consider the multiscale version~\eqref{prob_eps}, parametrized by $\epsilon$, of the stochastic rigid body system. Based on the expression~\eqref{srbI} for the stochastic Poisson integrator, applying the general asymptotic preserving scheme~\eqref{APscheme} introduced in Section~\ref{APsplit} gives the scheme
\begin{equation}\label{srbIAP}
\left\lbrace
\begin{aligned}
y^{\epsilon,n}&=\exp(hY_{H_3})\circ\exp(hY_{H_2})\circ\exp(hY_{H_1})\\
&\quad \circ\exp(\frac{h\xi_{3}^{\epsilon,[n]}}{\epsilon}Y_{\widehat H_3})\circ \exp(\frac{h\xi_{2}^{\epsilon,[n]}}{\epsilon}Y_{\widehat H_2})\circ\exp(\frac{h\xi_{1}^{\epsilon,[n]}}{\epsilon}Y_{\widehat H_1})\\
\xi_k^{\epsilon,[n]}&=\xi_{k}^{\epsilon,[n-1]}-\frac{h}{\epsilon^2}\xi_{k}^{\epsilon,[n]}+\frac{\Delta_n W_k}{\epsilon}=\frac{1}{1+\frac{h}{\epsilon^2}}\Bigl(\xi_k^{\epsilon,[n-1]}+\frac{\Delta_n W_k}{\epsilon}\Bigr),\quad k=1,2,3.
\end{aligned}
\right.
\end{equation}
The initial values are $y^{\epsilon,[0]}=y^{[0]}=y(0)$ and $\xi_{k}^{\epsilon,[0]}=0$, $k=1,2,3$.

First, let us illustrate the qualitative behaviour of the scheme~\eqref{srbIAP}, for different values of $\epsilon$. For this numerical experiment, the moments of inertia are $I=\hat{I}=(2,1,2/3)$, the initial value is $y(0)=(\cos(1.1),0,\sin(1.1))$ and the final time is $T=1$. The time step size is equal to $h=10^{-4}$. In Figure~\ref{fig:trajSRBAPa}, we plot the evolution of the Hamiltonian (top) and the Casimir (bottom) for the asymptotic preserving scheme~\eqref{srbIAP} applied with $\epsilon= 1,0.1,0.001$ (left) and for the stochastic Poisson integrator~\eqref{srbI}, formally $\epsilon=0$ (right). We observe the preservation of the Casimir function. In Figure~\ref{fig:trajSRBAPb}, we plot the evolution of the approximation of the trajectory $t_n\mapsto y(t_n)=(y_1(t_n),y_2(t_n),y_3(t_n))$, for different values of $\epsilon=1,0.1,0.001,0$. We observe that the trajectories are more regular when $\epsilon$ is large and converge to the solution of the stochastic Poisson integrator~\eqref{srbI} as $\epsilon$ tends to $0$.


\begin{figure}[h]
\begin{subfigure}{.5\textwidth}
\centering
\centering\includegraphics*[width=0.8\textwidth,keepaspectratio]{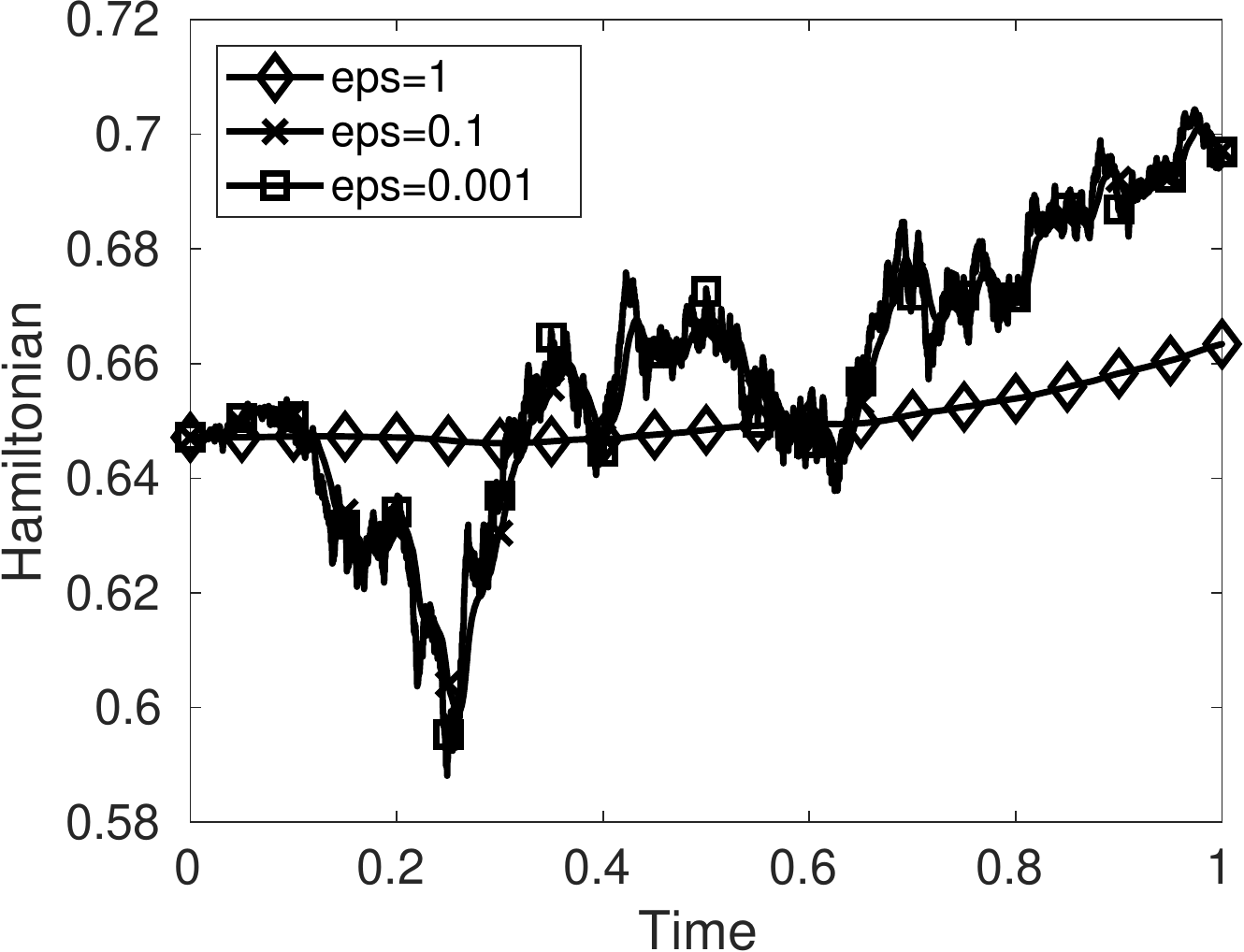}
\end{subfigure}
~
\begin{subfigure}{.5\textwidth}
\centering\includegraphics*[width=0.8\textwidth,keepaspectratio]{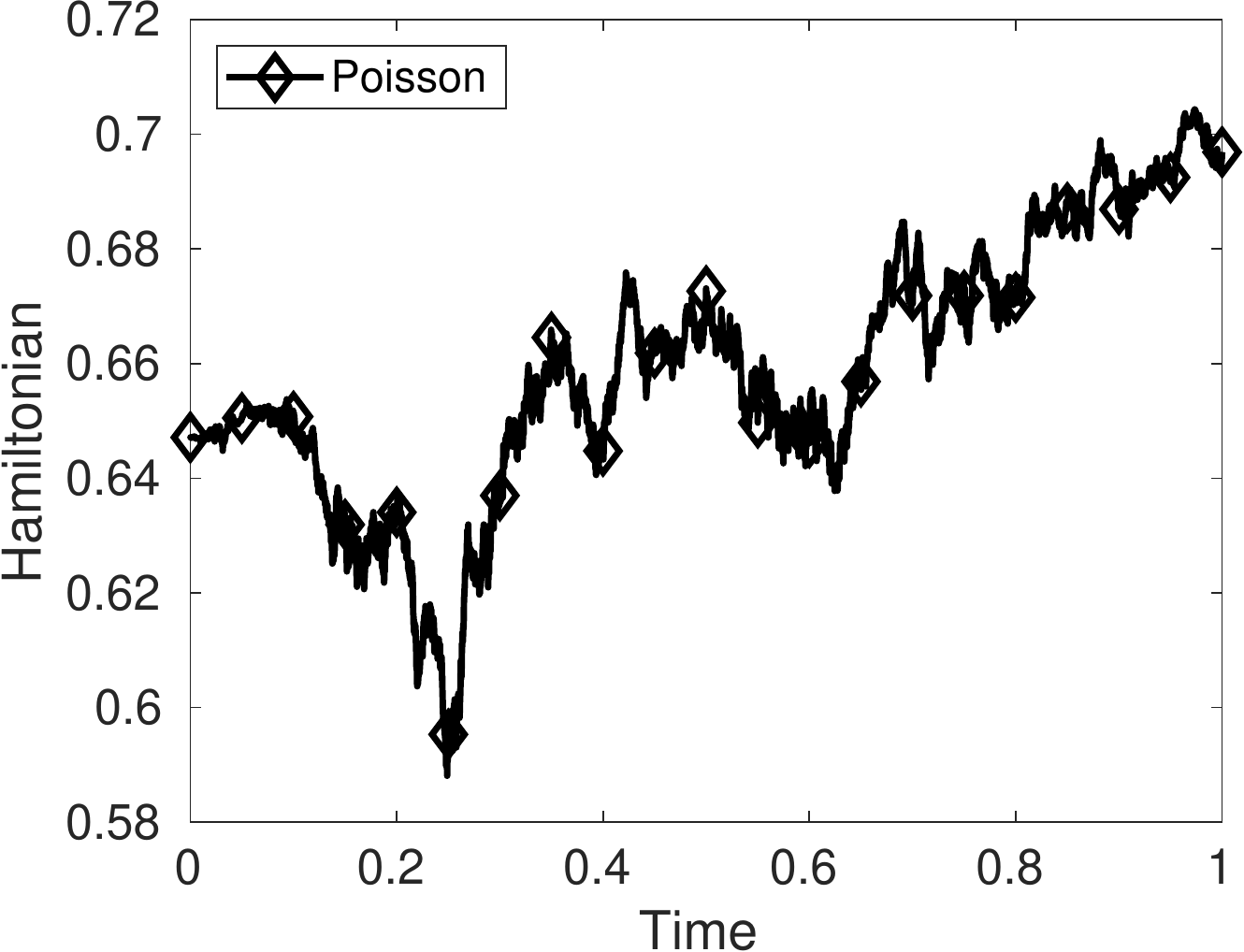}
\end{subfigure}
\newline
\begin{subfigure}{.5\textwidth}
\centering\includegraphics*[width=0.8\textwidth,keepaspectratio]{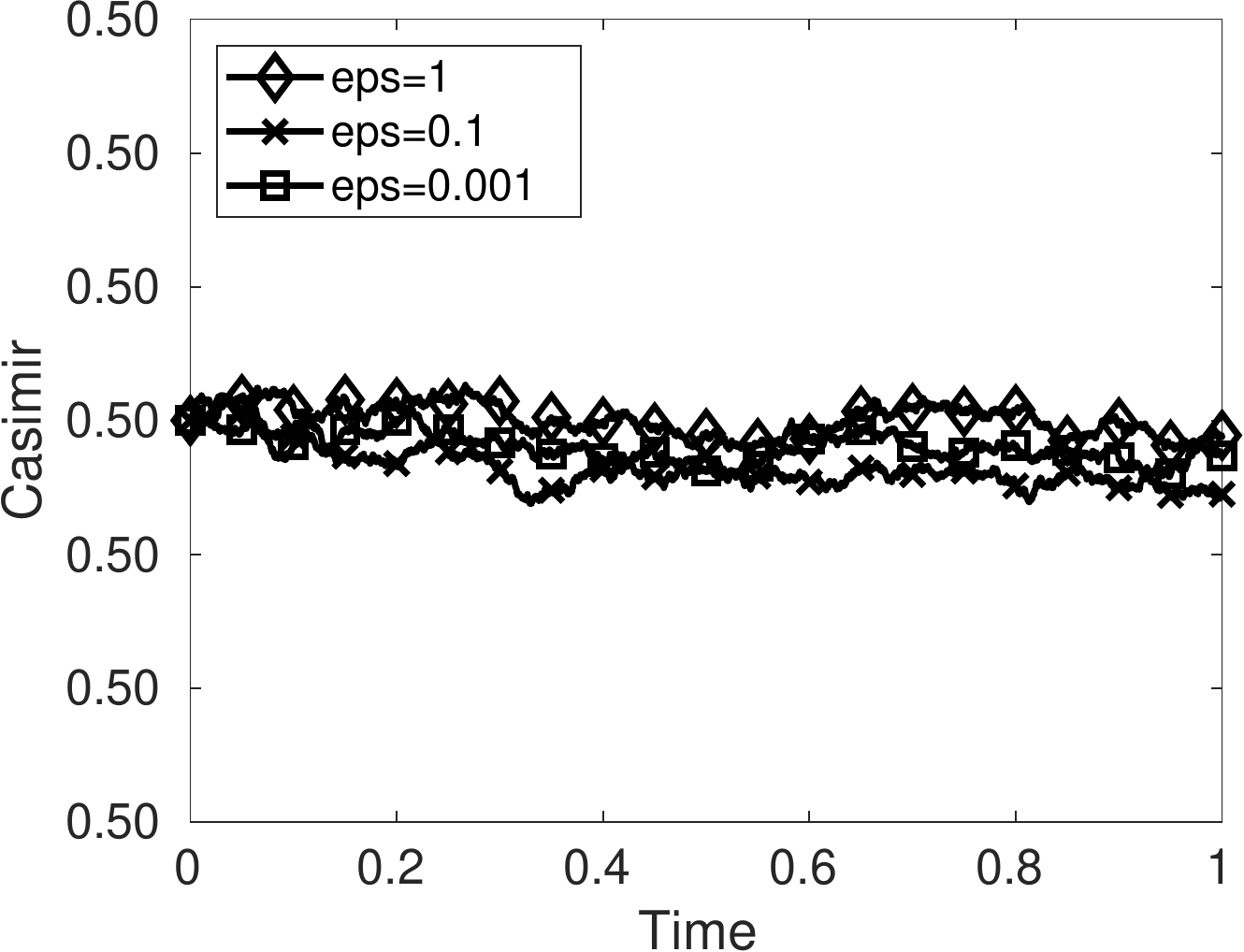}
\end{subfigure}
~
\begin{subfigure}{.5\textwidth}
\centering
\includegraphics*[width=0.8\textwidth,keepaspectratio]{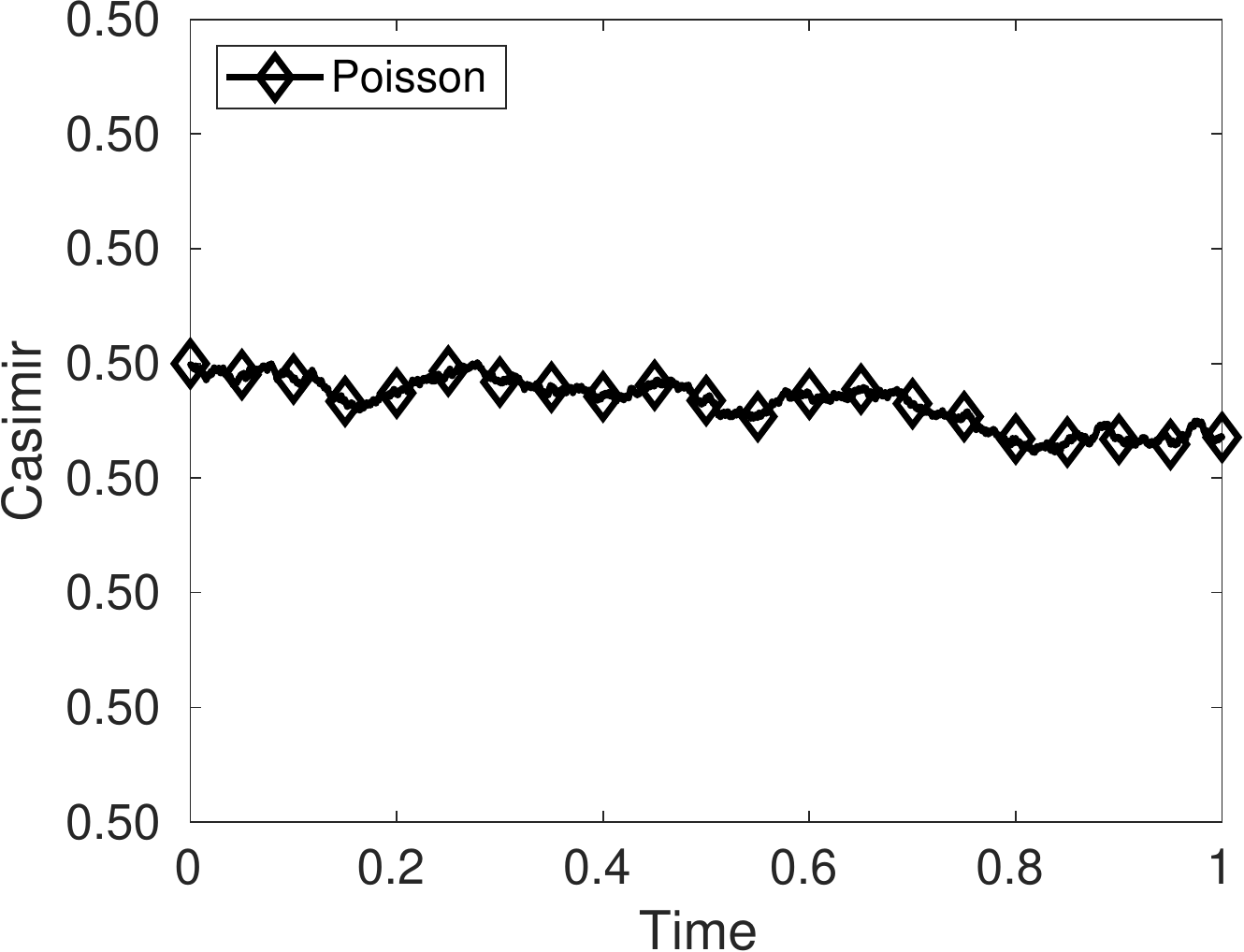}
\end{subfigure}
\caption{Stochastic rigid body system: evolution of the Hamiltonian (top) and the Casimir (bottom) of the numerical solution using the asymptotic preserving scheme~\eqref{srbIAP}. Left: $\epsilon=1,0.1,0.001$. Right: $\epsilon=0$.}
\label{fig:trajSRBAPa}
\end{figure}

\begin{figure}[h]
\centering
\begin{subfigure}{.5\textwidth}
\centering
\includegraphics*[width=0.8\textwidth,keepaspectratio]{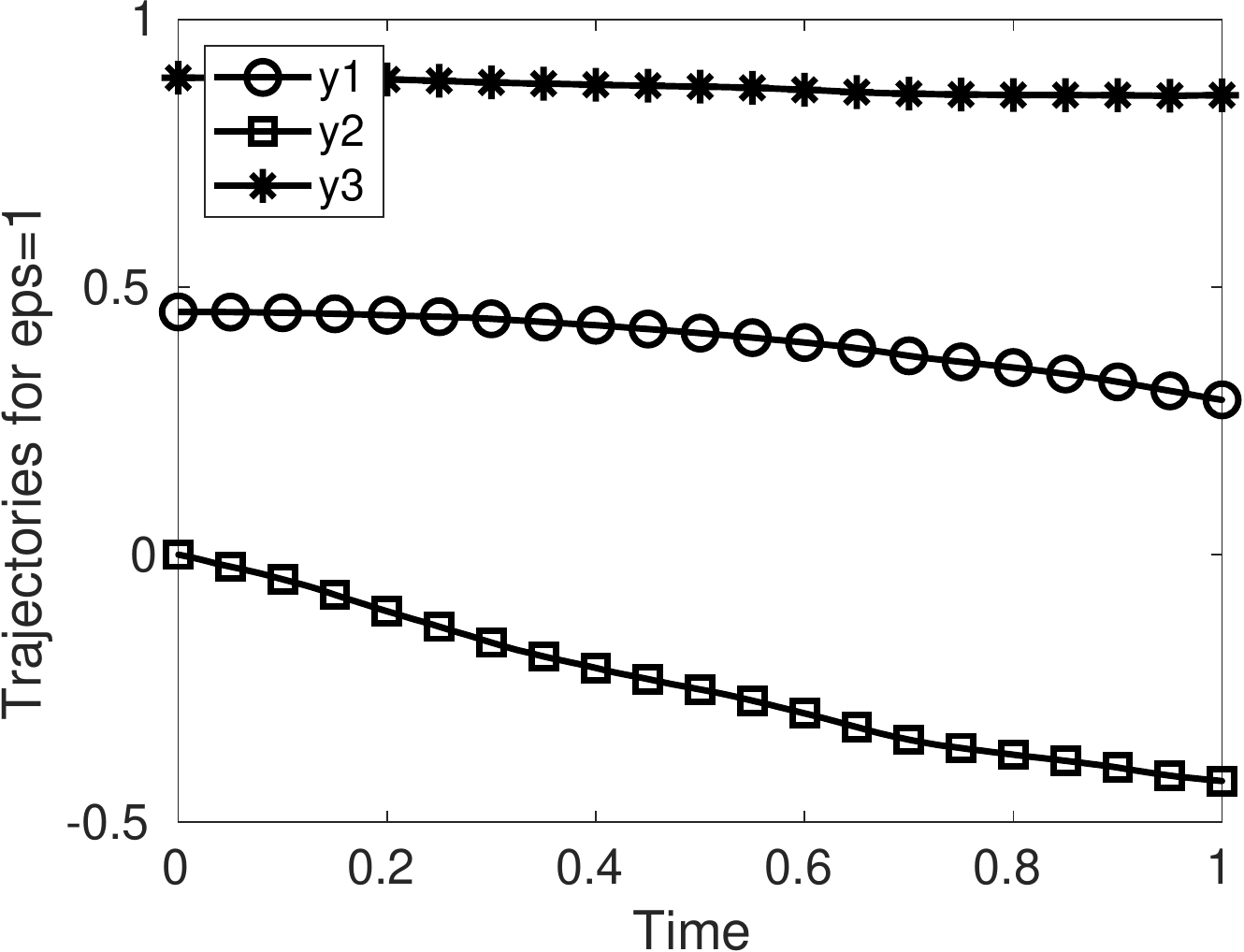}
\end{subfigure}
~
\begin{subfigure}{.5\textwidth}
\centering\includegraphics*[width=0.8\textwidth,keepaspectratio]{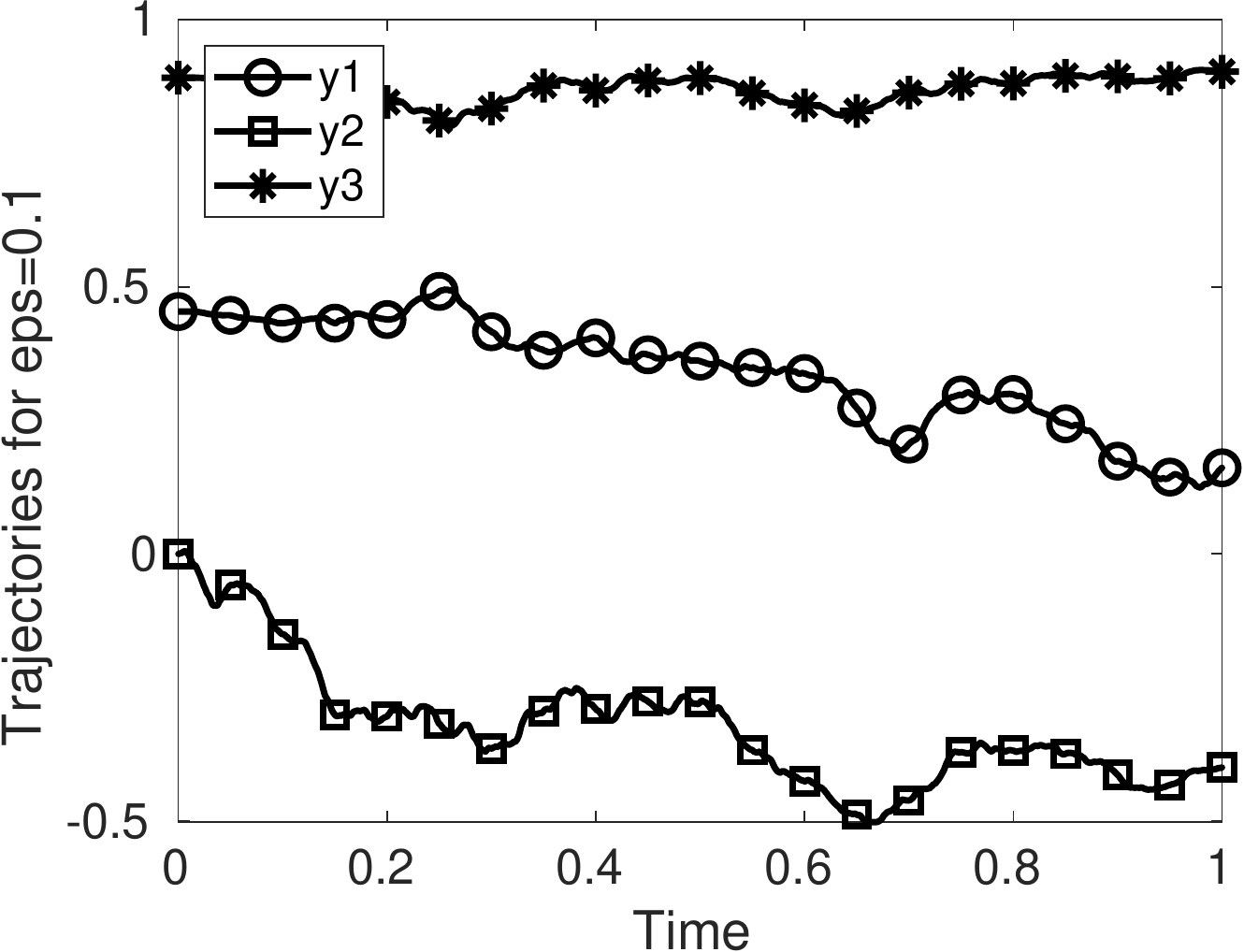}
\end{subfigure}

\begin{subfigure}{.5\textwidth}
\centering\includegraphics*[width=0.8\textwidth,keepaspectratio]{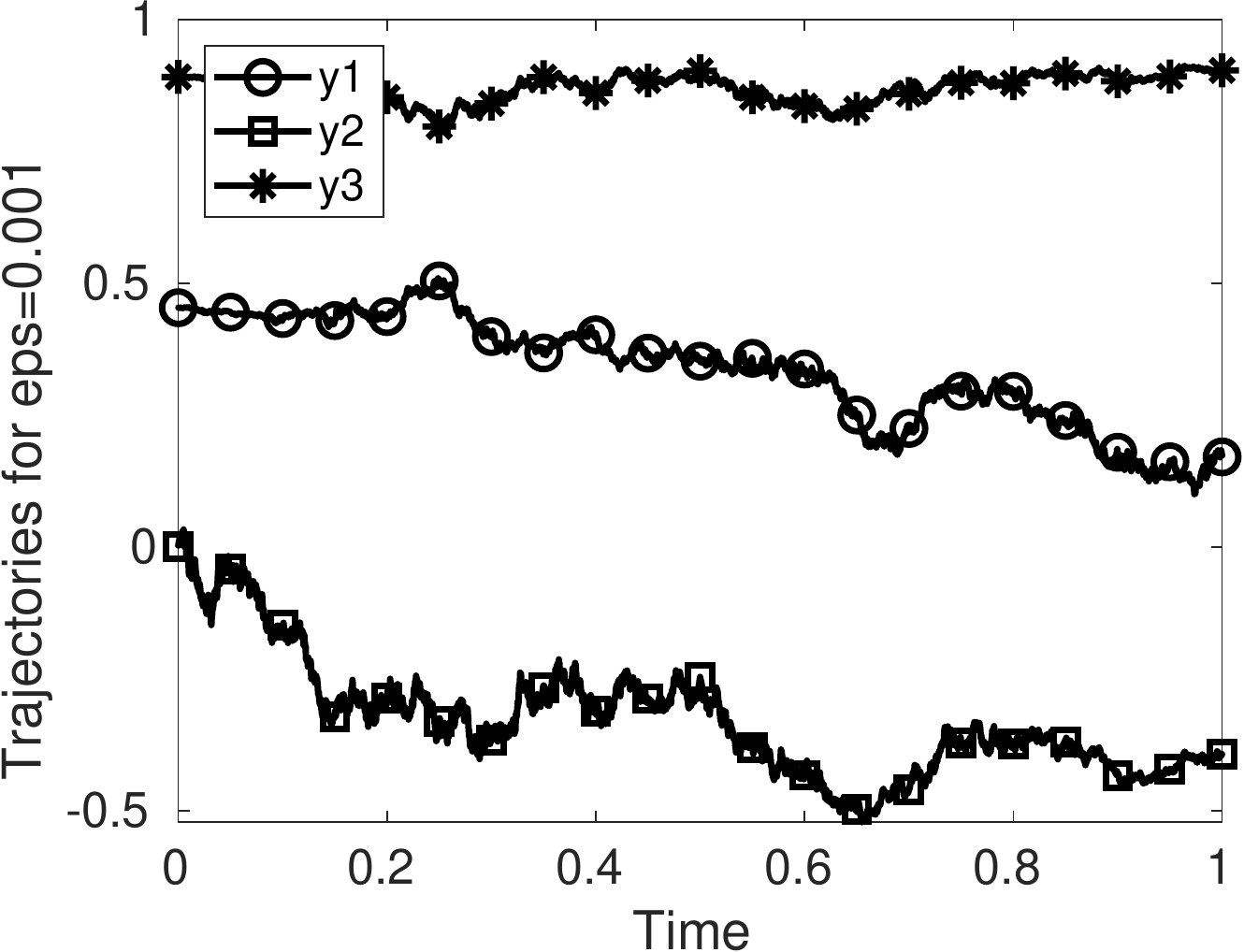}
\end{subfigure}
~
\begin{subfigure}{.5\textwidth}
\centering\includegraphics*[width=0.8\textwidth,keepaspectratio]{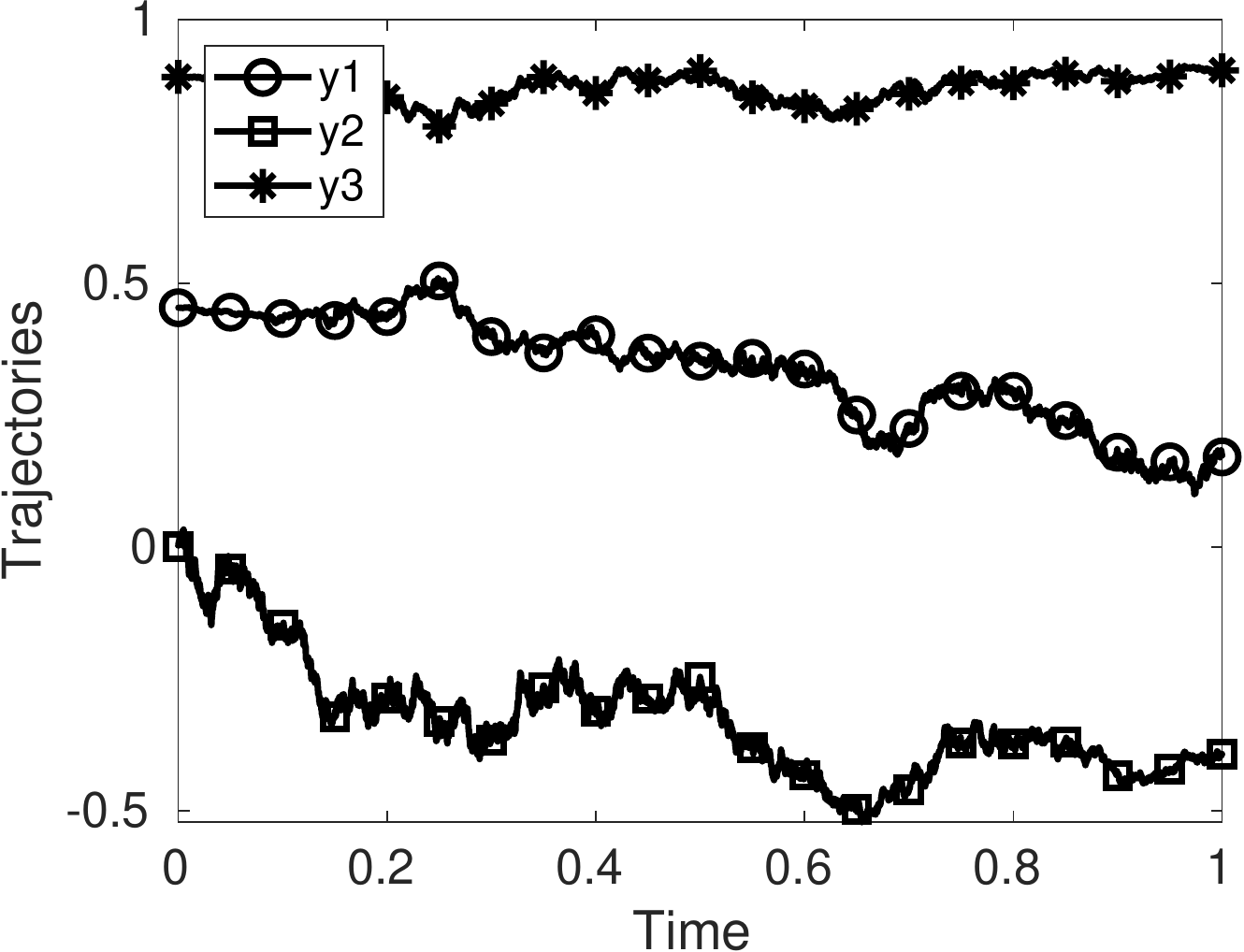}
\end{subfigure}
\caption{Stochastic rigid body system: Trajectory of the numerical solutions using the asymptotic preserving scheme~\eqref{srbIAP}. Top: $\epsilon=1,0.1$. Bottom: $\epsilon=0.001,0$.}
\label{fig:trajSRBAPb}
\end{figure}

Finally, the last experiment of this subsection illustrates the uniform accuracy property of the splitting scheme~\eqref{srbIAP} with respect to the parameter $\epsilon$ in the weak sense. For this numerical experiment, the moments of inertia are $I=(2,1,2/3)$ and $\hat{I}=(20,10,20/3)$, the initial value is $y(0)=(\cos(1.1),0,\sin(1.1))$ and the final time is $T=1$. The reference solution is computed using each scheme with time step size $h_{\text{ref}}=2^{-16}$, and the schemes are applied with the range of time step sizes $h=2^{-6},\ldots,2^{-12}$. The expectation is approximated averaging the error over $M_s=10^8$ independent Monte Carlo samples. The test function is given by $\phi(y)=\sin(2\pi y_1)+\sin(2\pi y_2)+\sin(2\pi y_3)$. The parameter $\epsilon$ takes the following values: $\epsilon=10^{-2}, 10^{-3}, 10^{-4}, 10^{-5}$. The results are seen in Figure~\ref{fig:plotmatlabRGBAP}. We observe that the weak error seems to be bounded uniformly with respect to $\epsilon$, with an order of convergence $1$. For a standard method such as the Euler--Maruyama scheme, the behaviour would be totally different: for fixed time step size $h$, the error is expected to be bounded away from $0$ when $\epsilon$ goes to $0$. Based on this numerical experiment, we conjecture that the asymptotic preserving scheme is uniformly accurate.

\begin{figure}[h]
\centering
\includegraphics*[width=0.48\textwidth,keepaspectratio]{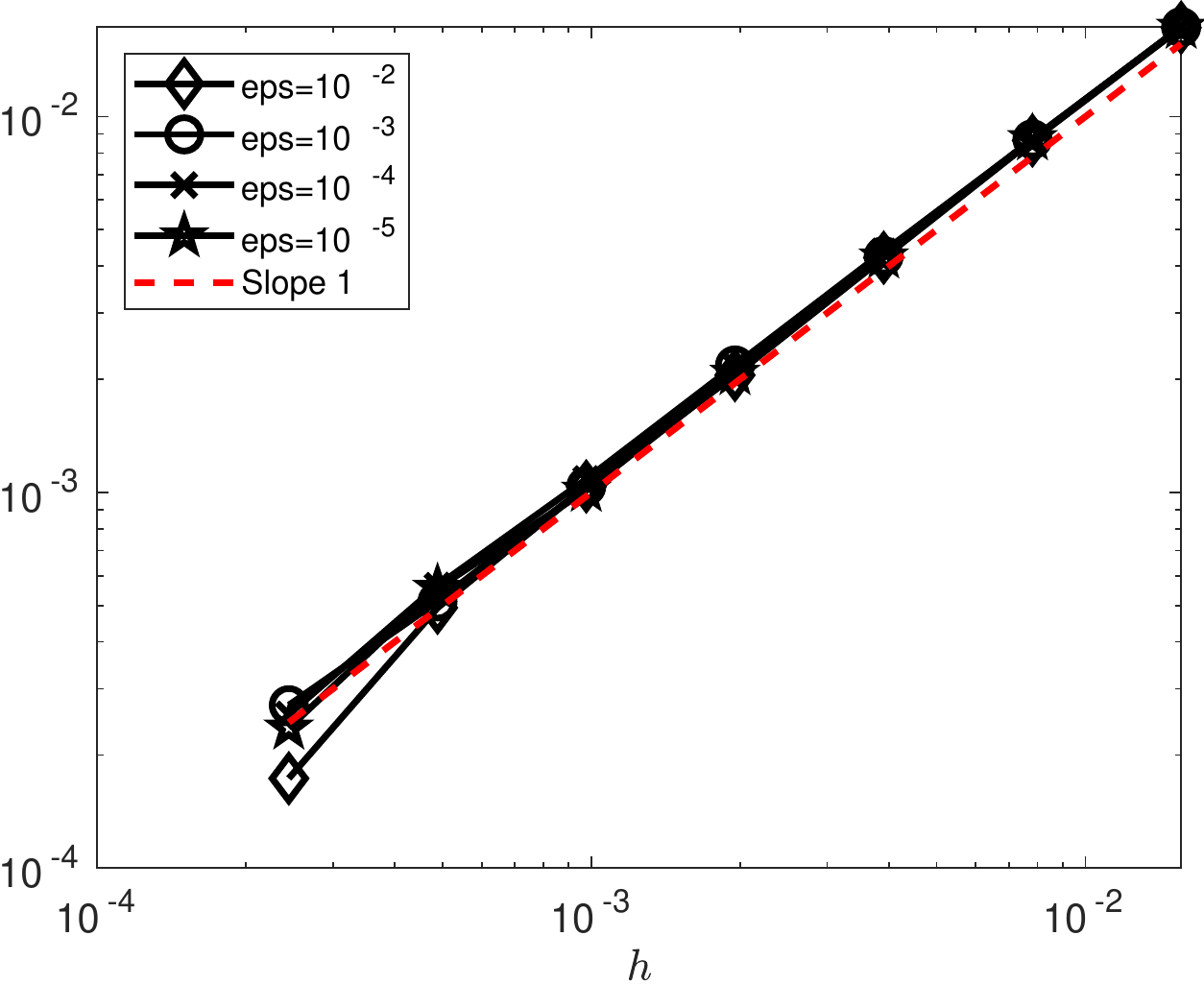}
\caption{Stochastic rigid body system: Weak errors of the asymptotic preserving scheme~\eqref{srbIAP} for $\epsilon=10^{-2}, 10^{-3}, 10^{-4}, 10^{-5}$.} 
\label{fig:plotmatlabRGBAP}
\end{figure}

\subsection{Explicit stochastic Poisson integrators for the stochastic sine--Euler system}\label{ssec:PSE} 

The last example of a stochastic Lie--Poisson system studied in this work is the stochastic version of the sine--Euler equations~\eqref{stochSE} introduced in Example~\ref{expl-SE}. Like for the previous examples, an explicit stochastic Poisson integrator is designed using a splitting strategy. We both illustrate the qualitative behaviour of the proposed integrator (preservation of Casimir functions) and strong error estimates. Note that numerical experiments which would illustrate weak error estimates or the asymptotic preserving property are not reported for this example: indeed, the results would be similar to those presented for the two other examples above.


Recall that the stochastic sine--Euler system~\eqref{stochSE} is of the type
\begin{align*}
\diff \omega&=B(\omega)\left(\nabla H(\omega)\,\diff t+
\sigma_{(1,0)}\nabla \widehat H_{(1,0)}(\omega)\circ\,\diff W_{(1,0)}(t)+\sigma_{(1,1)} \nabla \widehat H_{(1,1)}(\omega)\circ\,\diff W_{(1,1)}(t) \right.\nonumber\\
&\quad\left.+\sigma_{(0,1)}\nabla \widehat H_{(0,1)}(\omega)\circ\,\diff W_{(0,1)}(t) 
+\sigma_{(-1,1)}\nabla \widehat H_{(-1,1)}(\omega)\circ\,\diff W_{(-1,1)}(t)\right),
\end{align*}
This is a stochastic Lie--Poisson system. In order to design a splitting integrator, note that the Hamiltonian function $H$ can be split as $H=H_{(1,0)}+H_{(1,1)}+H_{(0,1)}+H_{(-1,1)}$, with $H_{\bf k}(\omega)=\widehat H_{\bf k}(\omega)=\frac{\omega_{\bf k}\omega_{\bf k}^\star}{|\bf k|^2}$. 

Like for the other examples, the deterministic subsystems
\[
\dot{\omega}_{\bf k}=B(\omega_{\bf k})\nabla H_{\bf k}(\omega_{\bf k})
\]
and the stochastic subsystems
\[
\diff \omega_{\bf k}=B(\omega_{\bf k})\nabla \widehat H_{\bf k}(\omega_{\bf k})\circ \diff W_{\bf k}(t)
\]
can be solved exactly: indeed, for each subsystem, the variable $\omega_{\bf k}$ is preserved and the three other variables evolve following a linear differential equation. We refer to~\cite{MR1860719,MR1246065} for the explanation of this idea for the deterministic subsystems. The treatment of the stochastic subsystems is straightforward using the exact solution of the subsystems $\dot{\omega}_{\bf k}=B(\omega_{\bf k})\nabla \widehat H_{\bf k}(\omega_{\bf k})$. 
The splitting scheme for the stochastic sine--Euler SDE \eqref{stochSE} then reads 
\begin{align}\label{stochSEI}
\Phi_h=\Phi_h^{\text{det}}\circ\Phi_{\Delta W}^{\text{stoch}}&=
\exp(h\Omega_{H_{(-1,1)}})\circ\exp(h\Omega_{H_{(0,1)}})\circ\exp(h\Omega_{H_{(1,1)}})\circ\exp(h\Omega_{H_{(1,0)}})\nonumber \\
&\quad\circ
\exp(\sigma_{(-1,1)}\Delta W_{(-1,1)}\Omega_{\widehat H_{(-1,1)}})\circ\exp(\sigma_{(0,1)}\Delta W_{(0,1)}\Omega_{\widehat H_{(0,1)}})\nonumber\\
&\quad\circ \exp(\sigma_{(1,1)}\Delta W_{(1,1)}\Omega_{\widehat H_{(1,1)}})
\circ\exp(\sigma_{(1,0)}\Delta W_{(1,0)}\Omega_{\widehat H_{(1,0)}}),
\end{align}
where we denote by $\Omega_{H_{\bf k}}$, resp. by $\Omega_{\widehat H_{\bf k}}$, the exact flow of the ODE subsystem, resp. SDE subsystem, with index ${\bf k}\in\{(1,0),(1,1),(0,1),(-1,1)\}$.

\subsubsection{Preservation of Casimir functions for the stochastic sine--Euler system}

Recall that the stochastic sine--Euler system admits two Casimir functions $C_1$ and $C_2$ given in Example~\ref{expl-SE}. Owing to Proposition~\ref{propo:sPi}, the numerical scheme~\eqref{stochSEI} is a stochastic Poisson integrator, in particular it preserves the two Casimir functions $C_1$ and $C_2$. 
We numerically illustrate this property in Figure~\ref{fig:trajSE}, where one sample of the numerical solution is computed with the time step size $h=0.02$.
The initial value is 
$\omega(0)=(\omega_{(1,0)}(0),\omega_{(1,1)}(0),\omega_{(0,1)}(0),\omega_{(-1,1)}(0))=(0.1+0.3i, 0.2+0.3i, 0.3+0.2i, 0.4+0.1i)$. In addition, $\sigma_k=1$ for all $k$ in this experiment. Figure~\ref{fig:trajSE} confirms that the two Casimir functions are preserved by the proposed integrator~\eqref{stochSEI}.

\begin{figure}[h]
\centering
\includegraphics*[width=0.48\textwidth,keepaspectratio]{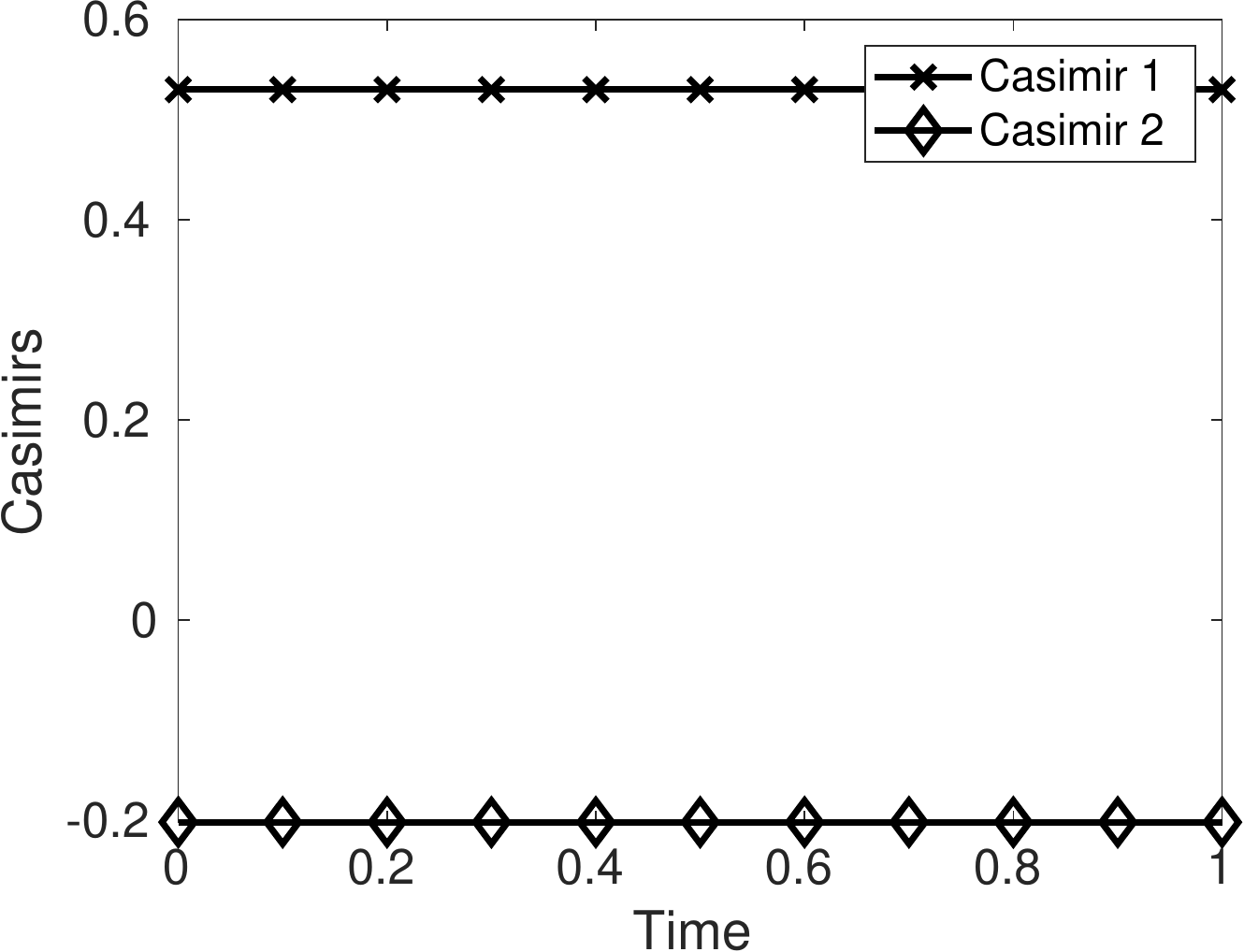}
\caption{Stochastic sine--Euler system: Preservation of the two Casimir functions by the explicit stochastic Poisson integrator~\eqref{stochSEI}.}
\label{fig:trajSE}
\end{figure}

\subsubsection{Strong convergence of the explicit stochastic Poisson integrator for the stochastic sine--Euler system}


Like for the stochastic rigid body system, we have the following convergence result due to the preservation of the Casimir function $C_1$. 
\begin{proposition}\label{thm-se}
Consider a numerical discretisation of the stochastic sine--Euler system~\eqref{stochSE} by the explicit stochastic Poisson integrator~\eqref{srbI}. 
Then, the strong order of convergence of this scheme is $1/2$, and the weak order of convergence is $1$. If the system is driven by a single Wiener process, the strong order of convergence is $1$.
\end{proposition}

As already explained, the convergence result above is not trivial since the scheme is explicit: since the coefficients of the equations are not globally Lipschitz continuous, the explicit Euler--Maruyama scheme does not converge in the strong sense.

\begin{proof}
The stochastic Poisson system~\eqref{stochSE} admits the Casimir function $\omega\mapsto C_1(\omega)=|\omega_1|^2+\ldots+|\omega_4|^2$, which has compact level sets.  
It then suffices to apply the general convergence result, Theorem~\ref{thm-general}, which 
concludes the proof of Proposition~\ref{thm-se}.
\end{proof}

We now numerically illustrate the strong rate of convergence of the proposed integrator~\eqref{stochSEI}
when applied to the SDE \eqref{stochSE}. The final time is $T=1$ and the initial value is 

$\omega(0)=(\omega_{(1,0)}(0),\omega_{(1,1)}(0),\omega_{(0,1)}(0),\omega_{(-1,1)}(0))=(0.1+0.3i, 0.2+0.3i, 0.3+0.2i, 0.4+0.1i)$. The reference solution is computed using the proposed scheme with time step size $h_{\text{ref}}=2^{-15}$, and the scheme is applied with the range of time step sizes $h=2^{-5},\ldots,2^{-13}$. The expectation is approximated averaging the error over $M_s=500$ independent Monte Carlo samples. The results are presented in Figure~\ref{fig:msSE}. On the left, $\sigma_1=1$ and $\sigma_j=0$, for $j=2,\ldots,4$: we observe an order of convergence equal to $1$, which confirms the result in Proposition~\ref{thm-se} when the system is driven by a single Wiener process. On the right, $\sigma_j=1$, for $j=1,\ldots,4$, which means that the system is driven by four independent Wiener processes. We observe an order of convergence equal to $1/2$, which confirms the result in Proposition~\ref{thm-se}.


\begin{figure}[h]
\centering
\includegraphics*[width=0.48\textwidth,keepaspectratio]{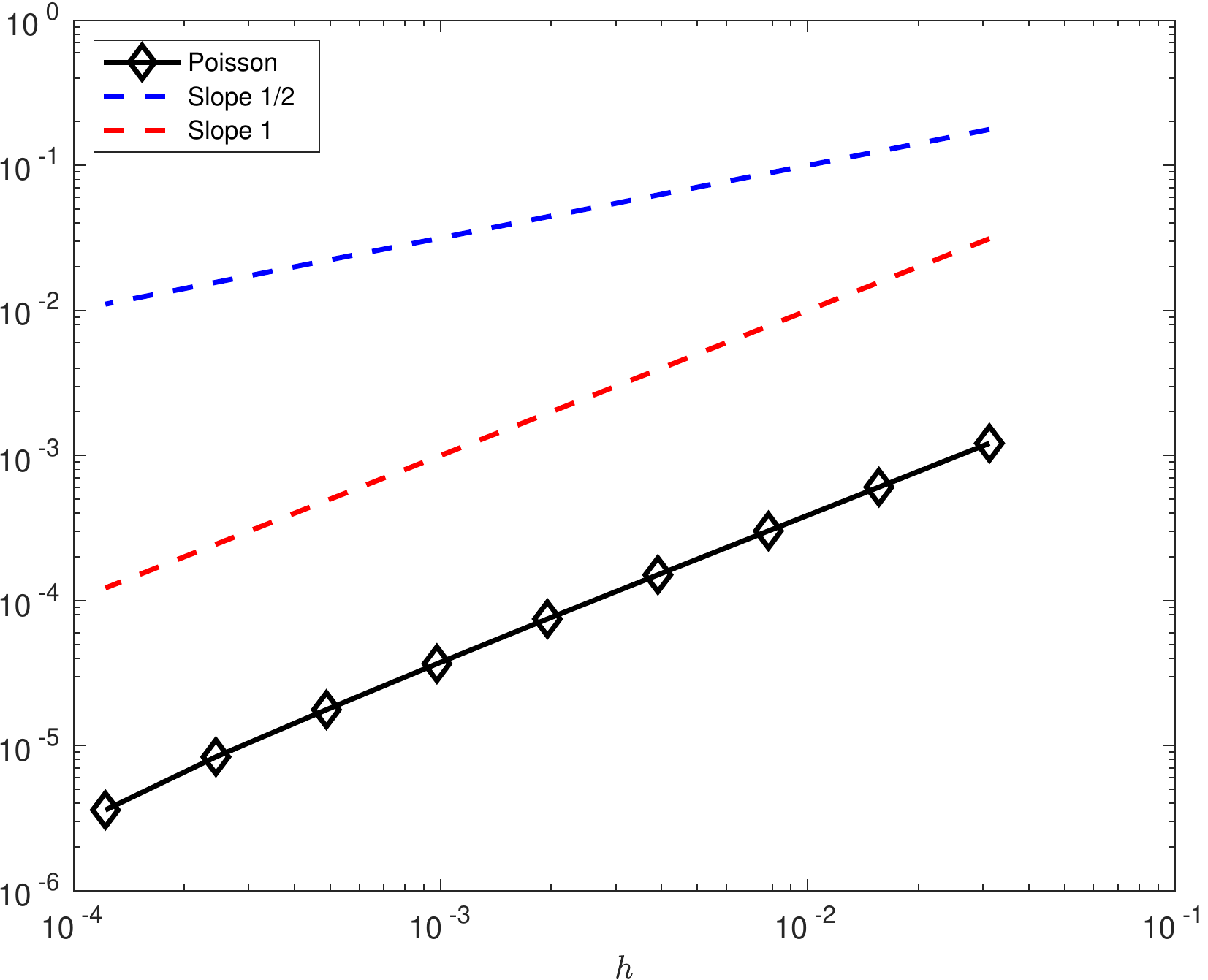}
\includegraphics*[width=0.48\textwidth,keepaspectratio]{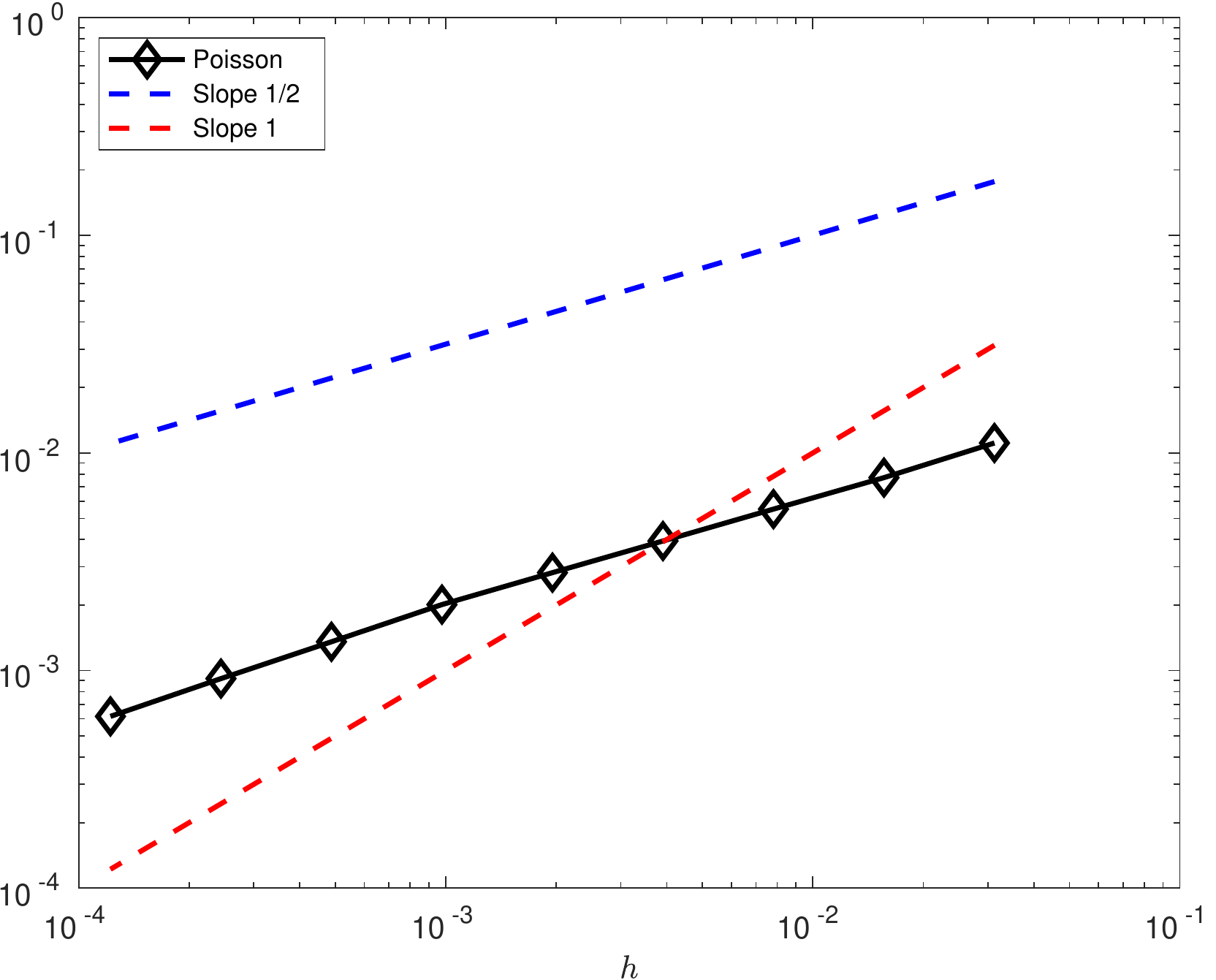}
\caption{Stochastic sine--Euler system: Strong errors of the explicit stochastic Poisson integrator~\eqref{stochSEI}. Left: single noise. Right: multiple noise.}
\label{fig:msSE}
\end{figure}

\appendix
\section{Proof of Lemma~\ref{lemm-aux}}\label{app-auxlem}
The goal of this section is to present the proof of Lemma~\ref{lemm-aux}. The arguments are standard and not specific to the present stochastic Poisson systems and integrators, and they are provided for the convenience of the reader.

Recall that the auxiliary SDE~\eqref{auxSDE} is given by
\[
\diff z(t)=\sum_{k=1}^{p}f_k(z(t))\,\diff t+\sum_{k=1}^m\widehat f_k(z(t))\circ\,\diff W_k(t),
\]
and that the associated splitting integrator~\eqref{auxscheme} is given by
\begin{equation}
z^{[n]}=\varphi_p(h,\cdot)\circ\ldots\circ \varphi_1(h,\cdot)\circ\widehat\varphi_m(\Delta W_m^n,\cdot)\circ \ldots\circ \widehat\varphi_1(\Delta W_1^n,\cdot)(z^{[n-1]}).
\end{equation}

\begin{proof}[Proof of Lemma~\ref{lemm-aux}]
To establish the strong convergence estimate~\eqref{eq:strongaux}, we apply the fundamental theorem on the strong order of convergence, see~\cite[Theorem~1.1]{MR2069903}. Note that the It\^o formulation of the SDE~\eqref{auxSDE} is an equation with globally Lipschitz nonlinearities, due to the assumptions on $f_1,\ldots,f_p$ and $\widehat f_1,\ldots,\widehat f_m$. Since the coefficients $f_k$ and $\widehat{f}_k$ do not depend on time, it is sufficient to prove the following local error estimates:
\begin{align}\label{strongorderconditions}
\left(\mathbb E\left[ \norm{ z(h)-z^{[1]} }^2 \right] \right)^{1/2}\leq \CC\bigl(1+\norm{z_0}^2\bigr)^{1/2}h\quad\text{and}\nonumber\\
\quad\norm{\mathbb E\left[z(h)-z^{[1]}\right]}\leq \CC\bigl(1+\norm{z_0}^2\bigr)^{1/2} h^{3/2},
\end{align}
for some real number $\CC\in(0,\infty)$ which does not depend on the time step size $h$. In the proof, the value of $\CC$ may change from line to line.


The proof is based on the comparison of Stratonovich--Taylor expansions of the exact and numerical solutions. On the one hand, the exact solution of the auxiliary SDE~\eqref{auxSDE} satisfies the following Stratonovich--Taylor expansion formula: 
for all $t\ge 0$,
\begin{align*}
z(t)-z(0)&=\sum_{k=1}^p\int_0^t f_k(z(s))\diff s+\sum_{k=1}^m\int_0^t \widehat f_k(z(s))\circ\diff W_k(s) \\
&=\sum_{k=1}^pf_k(z(0))t+\sum_{k=1}^m\widehat f_k(z(0))W_k(t)\\
&\quad+\sum_{k,\ell=1}^m\widehat f_k'(z(0))\widehat f_\ell(z(0))\int_0^t\int_0^s\diff W_\ell(r)\circ\diff W_k(s)\\
&\quad+R_e(t,z_0),
\end{align*}
where the random variable $R_e(t,z_0)$ satisfies
\[
\left(\mathbb E\left[\norm{R_e(t,z_0)}^{2}\right]\right)^{1/2}\le \CC(1+\norm{z_0}^2)^{1/2} t^{3/2}.
\]
On the other hand, we claim that the numerical solution satisfies the following local error expansion:
\begin{align*}
z^{[1]}-z^{[0]}&=\sum_{k=1}^pf_k(z^{[0]})h+\sum_{k=1}^m\widehat f_k(z^{[0]})W_k(h)\\
&\quad+\frac12\sum_{k=1}^{m}\widehat f_k'(z^{[0]})\widehat f_k(z^{[0]})W_k(h)^2\\
&\quad+\sum_{1\le\ell<k\le m}\widehat f_k'(z^{[0]})\widehat f_\ell(z^{[0]})W_k(h)W_\ell(h)\\
&\quad+R_s(h,z_0),
\end{align*}
where the random variable $R_s(h,z_0)$ satisfies
\[
\left(\mathbb E\left[\norm{R_s(h,z_0)}^{2}\right]\right)^{1/2}\le \CC(1+\norm{z_0}^2)^{1/2} h^{3/2}.
\]
Finally, recall that for all $h>0$ and all $k=1,\ldots,m$,
\begin{itemize}
\item $\int_0^h\int_0^s\diff W_k(r)\circ\diff W_k(s)=\frac12 W_k(h)^2$,
\item $\mathbb E[\int_0^h\int_0^s\diff W_\ell(r)\circ\diff W_k(s)]=0$ if $\ell\neq k$,
\item $\mathbb E\left[\norm{\int_0^h\int_0^s\diff W_\ell(r)\circ\diff W_k(s)}^2\right]={\rm O}(h^2)$.
\end{itemize}
Comparing the Stratonovich--Taylor expansions of the exact and numerical solutions then provides the local error estimates~\eqref{strongorderconditions}. As already explained, the local error estimates~\eqref{strongorderconditions} imply the strong error estimate~\eqref{eq:strongaux}. Thus, it remains to prove the claim above for the local error expansion of the numerical solution.

Note that the arguments below illustrate why the independence of $W_1,\ldots,W_m$ is essential for the consistency of the scheme (see Remark~\ref{rem-consistent}). To simplify the notation, we write ${\rm O}(h^\alpha)$ for random variables $r$ 
which satisfy $\left(\mathbb E\left[\norm{r}^2\right]\right)^{1/2}\le \CC(1+|z^{[0]}|^2)^{1/2}h^\alpha$ for some real number $\CC\in(0,\infty)$ which does not depend on $h$.

The local error $z^{[1]}-z^{[0]}$ is decomposed as
\[
z^{[1]}-z^{[0]}=(Z_{p}-Z_{p-1})+\ldots+(Z_1-\widehat{Z}_m)+(\widehat{Z}_{m}-\widehat{Z}_{m-1})+\ldots+(\widehat{Z}_1-\widehat{Z_0}),
\]
where $\widehat{Z}_0=z^{[0]}$, $\widehat{Z}_k=\widehat\varphi_k(W_k(h),\widehat{Z}_{k-1})$  
for all $k=1,\ldots,m$, 
$Z_0=\widehat{Z}_m$ and $Z_k=\varphi_k(h,Z_{k-1})$. Observe that $z^{[1]}=Z_p$.

For all $k\in\{1,\ldots,m\}$, the random variables $\widehat{Z}_{k-1}$ and $W_k(h)$ are independent 
(since by construction $\widehat{Z}_{k-1}$ only depends on $W_1(h),\ldots,W_{k-1}(h)$). 
Since $\widehat\varphi_k$ is the flow associated with the ODE $\dot{z}_k=\widehat f_k(z_k)$, a Stratonovich--Taylor 
expansion for the solution $t\in[0,h]\mapsto \widehat\varphi_k(W_k(t),\widehat{Z}_{k-1})$ of the Stratonovich SDE 
$\diff \widehat{z}_k=\widehat f_k(\widehat{z}_k)\circ \diff W_k(t)$ yields
\begin{align*}
\widehat{Z}_k-\widehat{Z}_{k-1}&=\widehat\varphi_{k}(W_k(h),\widehat{Z}_{k-1})-\widehat{Z}_{k-1}\\
&=
W_k(h)\widehat f_k(\widehat{Z}_{k-1})+\frac12 W_k(h)^2\widehat f_k'(\widehat{Z}_{k-1})\widehat f_k(\widehat{Z}_{k-1})+{\rm O}(h^{3/2}).
\end{align*}
Using this result successively for $k=1,\ldots,m$, one gets the more precise results 
\begin{align*}
\widehat{Z}_k-z^{[0]}&=\sum_{\ell=1}^{k}W_{\ell}(h)\widehat f_\ell(\widehat Z_{\ell-1})+{\rm O}(h)\\
&={\rm O}(h^{1/2})\\
&=\sum_{\ell=1}^{k}W_{{\ell}}(h)\widehat f_\ell(z^{[0]})+{\rm O}(h)
\end{align*}
using that $\widehat Z_{\ell-1}=z^{[0]}+O(h^{1/2})$ in the last step. Finally, one has 
\begin{align*}
\widehat{Z}_k-\widehat{Z}_{k-1}&=W_k(h)\widehat f_k(z^{[0]})+\frac12 W_k(h)^2\widehat f_k'(z^{[0]}){\widehat f}_k(z^{[0]})\\
&\quad+\sum_{\ell=1}^{k-1}W_k(h)W_\ell(h)\widehat f_k'(z^{[0]}){\widehat f}_\ell(z^{[0]})+{\rm O}(h^{3/2}).
\end{align*}
Similarly, for all $k\in\{1,\ldots,p\}$, one obtains in a first step
\[
Z_k-Z_{k-1}=\varphi_k(h,Z_{k-1})-Z_{k-1}=hf_k(Z_{k-1})+{\rm O}(h^2),
\]
and in a second step (recall that $Z_0=\widehat{Z}_m=z^{[0]}+{\rm O}(h^{1/2})$)
\[
Z_k-Z_{k-1}=hf_k(Z_0)+{\rm O}(h^2)=hf_k(z^{[0]})+{\rm O}(h^{3/2}).
\]
Summing all the local error terms then concludes the proof of the claim, i.e. of the Taylor--Stratonovich expansion formula for the numerical solution.

Note that in the commutative noise case (in particular, when $m=1$), one has the identity
\[
\sum_{1\le\ell<k\le m}\widehat f_k'(z^{[0]})\widehat f_\ell(z^{[0]})W_k(h)W_\ell(h)=\sum_{1\le k\neq \ell\le m}\widehat f_k'(z(0))\widehat f_\ell(z(0))\int_0^h\int_0^s\diff W_\ell(r)\circ\diff W_k(s).
\]
Using the two Taylor--Stratonovich expansion formulas written above, one obtains the following version of the local error estimates~\eqref{strongorderconditions} in the commutative noise case:
\begin{align*}
\left(\mathbb E\left[ \norm{ z(h)-z^{[1]} }^2 \right] \right)^{1/2}\leq \CC\bigl(1+\norm{z_0}^2\bigr)^{1/2}h^{3/2}\quad\text{and}\\
\quad\norm{\mathbb E\left[z(h)-z^{[1]}\right]}\leq \CC\bigl(1+\norm{z_0}^2\bigr)^{1/2} h^{3/2}.
\end{align*}
Owing to the fundamental theorem on the strong order of convergence, see~\cite[Theorem~1.1]{MR2069903}, the strong order of convergence is thus equal to $1$ in the commutative noise case.

This concludes the proof of the strong error estimate~\eqref{eq:strongaux}, in the general and in the commutative noise cases.

Let us now prove the weak error estimate~\eqref{eq:weakaux}. To simplify the notation, we only prove the weak error estimate for $n=N$. Indeed, the extension to the case 
$n=0,\ldots,N$ with a uniform upper bound is straightforward and details are omitted.

The proof is a variant of the Talay--Tubaro argument, see~\cite{MR1091544}: the weak error is written in terms of the solution of a backward Kolmogorov equation. First, the infinitesimal generator $\mathcal{L}$ of the auxiliary SDE~\eqref{auxSDE} is written as
\[
\mathcal{L}=\sum_{k=1}^{p}\mathcal{L}_k+\sum_{k=1}^{m}\widehat{\mathcal{L}}_k,
\]
where, for all $k\in\{1,\ldots,p\}$, $\mathcal{L}_k=f_k\cdot\nabla$ is associated with the ODE $\dot{z}_k=f_k(z_k)$, 
and for all $k\in\{1,\ldots,m\}$, $\widehat{\mathcal{L}}_k=\frac12f_k'f_k \cdot \nabla+\frac12 f_k f_k^T:\nabla^2$ is associated with the SDE $\diff z_k=\widehat f_k(z_k)\circ \diff W_k(t)$ rewritten in It\^o form as $\diff z_k=\frac12f_k'(z_k)f_k(z_k)+\widehat f_k(z_k)\diff W_k(t)$. The formula for $\mathcal{L}$ above employs the independence of the Wiener processes $W_1,\ldots,W_m$. 

Introduce the auxiliary function $u$ which is the solution of the backward Kolmogorov equation
\[
\partial_tu(t,z)=\mathcal{L}u(t,z)
\]
for all $t\in[0,T]$ and $z\in\R^d$, with the initial condition $u(0,\cdot)=\phi$, assumed to be of class $\mathcal{C}^4$ with bounded derivatives. Since the mappings $f_1,\ldots,f_p$ are of class $\mathcal{C}^4$ and the mappings $\widehat f_1,\ldots,\widehat f_m$ are bounded and of class $\mathcal{C}^5$, with bounded derivatives of any order, the function $u$ is of class $\mathcal{C}_b^{2,4}$, with bounded derivatives of any order. This means that $t\mapsto u(t,z)$ is of class $\mathcal{C}^2$ for all $z\in\R^d$, and that $z\mapsto u(t,z)$ is of class $\mathcal{C}^4$, with bounded derivatives. The proof is omitted, see for instance~\cite{MR1840644} for such results (after transformation to the It\^o formulation).

The weak error at time $T=Nh$ is expressed in terms of the solution $u$ of the Kolmogorov equation as follows (using a standard telescoping sum argument):
\begin{align*}
\mathbb E[\phi(z(Nh))]-\mathbb E[\phi(z^{[N]})]&=\mathbb E[u(Nh,z_0)]-\mathbb E[u(0,z^{[N]})]\\
&=\sum_{n=0}^{N-1}\bigl(\mathbb E[u(T-t_n,z^{[n]})]-\mathbb E[u(T-t_{n+1},z^{[n+1]})]\bigr)\\
&=\sum_{n=0}^{N-1}\bigl(\mathbb E[u(T-t_n,z^{[n]})]-\mathbb E[u(T-t_{n+1},z^{[n]})]\bigr)\\
&\quad+\sum_{n=0}^{N-1}\bigl(\mathbb E[u(T-t_{n+1},z^{[n]})]-\mathbb E[u(T-t_{n+1},z^{[n+1]})]\bigr),
\end{align*}
with $t_n=nh$, and recalling that $h=T/N$ and $z(0)=z^{[0]}=z_0$.

The first term in the right-hand side of the weak error decomposition above is treated as follows:
\[
\mathbb E[u(T-t_n,z^{[n]})]-\mathbb E[u(T-t_{n+1},z^{[n]})]=h\mathbb E[\partial_tu(T-t_{n+1},z^{[n]})]+{\rm O}(h^2)
\]
by a straightforward Taylor expansion of $u$ with respect to the time variable. 
Indeed, the mapping $t\mapsto u(T-t,z^{[n]})$ is of class $\mathcal{C}^2$ with bounded first and second order derivatives.
To simplify the exposition, we use the notation ${\rm O}(h^2)$ for the reminder terms $r$, which satisfy $|r|\le \CC h^2$ for some real number $\CC\in(0,\infty)$, uniformly with respect to $n$.

To obtain the weak convergence result, it suffices to prove the following claim: 
for any function $\psi\colon\R^d\to\R$ of class $\mathcal{C}^4$ with bounded derivatives, one has
\[
\mathbb E[\psi(z^{[n+1]}) | z^{[n]}]=\psi(z^{[n]})+h\mathcal{L}\psi(z^{[n]})+{\rm O}(h^2).
\]
Indeed, applying that claim with $\psi=u(T-t_{n+1},\cdot)$, for each $n=0,\ldots,N-1$, in the above decomposition of the weak error  
and taking expectation, one obtains
\[
\mathbb E[\phi(z(Nh))]-\mathbb E[\phi(z^{[N]})]=h\sum_{n=0}^{N-1}\mathbb E[(\partial_tu-\mathcal{L}u)(T-t_{n+1},z^{[n]})]+{\rm O}(h)={\rm O}(h),
\]
since $u$ is the solution of the Kolmogorov equation $\partial_tu=\mathcal{L}u$.

It remains to prove the claim. This is done for $n=0$ and the general case is obtained using the Markov property. 
We employ the same notation as in the proof of the strong convergence estimate above. 
First, observe that
\[
\mathbb E[\psi(z^{[1]})]=\mathbb E[\psi(Z_p)]=\mathbb E[\psi(\varphi_p(h,Z_{p-1}))]=\mathbb E[e^{h\mathcal{L}_p}\psi(Z_{p-1})],
\]
where, by definition, $u_k(t,\cdot)=e^{t\mathcal{L}_k}\psi$ is the solution of the partial differential equation $\partial_tu_k=\mathcal{L}_ku_k$, 
with $u_k(0,\cdot)=\psi$, associated with the ODE $\dot{z}_k=f_k(z_k)$.

It is then straightforward to check that 
\[
\mathbb E[\psi(z^{[1]})]=\mathbb E[e^{h\mathcal{L}_1}\ldots e^{h\mathcal{L}_p}\psi(Z_0)]=
\mathbb E[e^{h\mathcal{L}_1}\ldots e^{h\mathcal{L}_p}\psi(\widehat{Z}_m)].
\]
It remains to treat the stochastic part in the splitting scheme. 
To simplify notation, let $\widehat{\psi}=e^{h\mathcal{L}_1}\ldots e^{h\mathcal{L}_p}\psi$. 
Recall that $\widehat{Z}_m=\varphi_m(W_m(h),\widehat{Z}_{m-1})$. Since $W_m(h)$ and $\widehat{Z}_{m-1}$ 
are independent ($\widehat{Z}_{m-1}$ depends only on $W_{m-1}(h),\ldots,W_1(h)$), one has
\[
\mathbb E[\widehat{\psi}(\widehat{Z}_m)]=\mathbb E\left[\mathbb E[\widehat{\psi}(\varphi_m(W_m(h),\widehat{Z}_{m-1}))|\widehat{Z}_{m-1}]\right]=
\mathbb E[e^{h\widehat{\mathcal{L}}_m}\widehat\psi(\widehat{Z}_{m-1})],
\]
where, by definition, $\widehat{u}_k(t,\cdot)=e^{t\widehat{\mathcal{L}}_k}\widehat{\psi}$ is the solution of 
the partial differential equation $\partial_t\widehat{u}_k=\widehat{\mathcal{L}}_k\widehat{u}_k$, 
with $\widehat{u}_k(0,\cdot)=\widehat{\psi}$, associated with the SDE $\diff z_k=\widehat f_k(z_k)\circ \diff W_k$.

It is then straightforward to check that 
\[
\mathbb E[\widehat{\psi}(\widehat{Z}_m)]=\mathbb E[e^{h\widehat{\mathcal{L}}_1}\ldots e^{h\widehat{\mathcal{L}}_m}\widehat{\psi}(\widehat{Z}_0)]=e^{h\widehat{\mathcal{L}}_1}\ldots e^{h\widehat{\mathcal{L}}_m}\widehat{\psi}(z^{[0]}),
\]
using the independence of $W_k(h)$ and $\widehat{Z}_{k-1}$ for each $k=m-1,\ldots,1$ successively, 
and the equality $Z_0=z^{[0]}$ at the last step.

Finally, it remains to use the fact that, if $\psi$ is of class $\mathcal{C}^4$ with bounded derivatives, one has
\begin{align*}
e^{h\widehat{\mathcal{L}}_1}\ldots e^{h\widehat{\mathcal{L}}_m}\widehat{\psi}&=e^{h\widehat{\mathcal{L}}_1}\ldots e^{h\widehat{\mathcal{L}}_m}e^{h\mathcal{L}_1}\ldots e^{h\mathcal{L}_p}\psi\\
&=\psi+h\sum_{k=1}^{m}\widehat{\mathcal{L}}_k\psi+h\sum_{k=1}^{p}\mathcal{L}_k\psi+{\rm O}(h^2)\\
&=\psi+h\mathcal{L}\psi+{\rm O}(h^2)
\end{align*}
to conclude the proof of the claim.

This concludes the proof of the weak error estimate, and of the auxiliary lemma.
\end{proof}

\section{Acknowledgements}
The authors would like to thank Klas Modin, Gilles Vilmart, and Milo Viviani for interesting discussions. 
They also would like to thank Beno\^it Fabr\`eges for his help for the numerical experiments on the weak error.  
The work of C.-E.~B. is partially supported by the  project SIMALIN (ANR-19-CE40-0016) operated by the French National Research Agency. 
The work of DC was partially supported by the Swedish Research Council (VR) (projects nr. $2018-04443$). 
The work of TJ was funded by the Deutsche Forschungsgemeinschaft (DFG, German Research Foundation) – Project-ID 258734477 – SFB 1173.
Some computations were performed on resources provided by the Swedish National Infrastructure 
for Computing (SNIC) at HPC2N, Ume{\aa} University and at UPPMAX, Uppsala University. 

\bibliographystyle{plain}
\bibliography{biblo}
\end{document}